\documentclass[12pt]{amsart}
\usepackage{amsmath}
\usepackage{amsxtra}
\usepackage{amscd}
\usepackage{amsthm}
\usepackage{amsfonts}
\usepackage{amssymb}
\usepackage {pstricks}
\usepackage{pstricks,pst-node}
\usepackage[all]{xy}

\usepackage{mathrsfs}

\usepackage{verbatim}

\newtheorem{theorem}{Theorem}[section]
\newtheorem*{theorem*}{Theorem}
\newtheorem{lemma}{Lemma}[section]
\newtheorem{proposition}{Proposition}[section]
\newtheorem*{proposition*}{Proposition}
\newtheorem{corollary}{Corollary}[section]

\theoremstyle{definition}
\newtheorem{definition}{Definition}[section]
\newtheorem{remark}{Remark}[section]
\newtheorem*{conjecture*}{Conjecture}
\newtheorem*{notation*}{Notation}
\newtheorem{notation}{Notation}[section]
\newtheorem{example}{Example}[section]

\numberwithin{equation}{section}

\def\1{1\kern-.3em1}

\newcommand{\Z}{{\mathbb Z}}
\newcommand{\Q}{{\mathbb Q}}
\newcommand{\R}{{\mathbb R}}
\newcommand{\C}{{\mathbb C}}
\newcommand{\K}{{\mathbb K}}
\newcommand{\F}{{\mathbb F}}

\newcommand{\U}{{\rm U}}

\newcommand{\osp}{{\mathfrak{osp}}}
\newcommand{\gl}{{\mathfrak{gl}}}

\newcommand{\g}{{\mathfrak g}}
\newcommand{\fg}{{\mathfrak g}}
\newcommand{\fb}{{\mathfrak b}}
\newcommand{\fh}{{\mathfrak h}}

\newcommand{\id}{{\rm{id}}}
\newcommand{\Hom}{{\rm{Hom}}}
\newcommand{\End}{{\rm{End}}}
\newcommand{\ot}{\otimes}

\newcommand{\cT}{{\mathscr T}}
\newcommand{\cH}{{\mathscr H}}
\newcommand{\cC}{{\mathscr C}}
\newcommand{\cD}{{\mathscr D}}
\newcommand{\cE}{{\mathscr E}}

\newcommand{\cF}{{\mathcal F}}
\newcommand{\cG}{{\mathcal G}}
\newcommand{\cS}{{\mathcal S}}
\newcommand{\cI}{{\mathcal I}}

\newcommand{\cM}{{\mathcal M}}
\newcommand{\cV}{{\mathcal V}}

\newcommand{\HMod}{\text{$H$-Mod}}

\newcommand{\Sym}{{\rm{Sym}}}

\newcommand{\sdim}{{\rm{sdim}}}

\begin{document}
\title[Invariant theory of quantum supergroups]
{First fundamental theorems of invariant theory for quantum supergroups}
\author[G. I. Lehrer]{G. I. Lehrer}
\author[Hechun Zhang]{Hechun Zhang}
\author[R. B. Zhang]{R. B. Zhang}
\address{School of Mathematics and Statistics,
The University of Sydney, Sydney, NSW 2006,  Australia}
\email{gustav.lehrer@sydney.edu.au, ruibin.zhang@sydney.edu.au}
\address{Department of Mathematical Sciences, Tsinghua University, Beijing,  China}
\email{hzhang@math.tsinghua.edu.cn}

\begin{abstract}
Let $\U_q(\fg)$ be the quantum supergroup of $\gl_{m|n}$ or the
modified quantum supergroup of $\osp_{m|2n}$ over the field of rational functions in $q$,
and let $V_q$ be the natural module for $\U_q(\fg)$.
There exists a unique tensor functor, associated with $V_q$, from the category of ribbon graphs to
the category of finite dimensional representations of $\U_q(\fg)$, which preserves ribbon category structures.  
We show that
this functor is full in the cases $\fg=\gl_{m|n}$ or $\osp_{2\ell+1|2n}$.
For $\fg=\osp_{2\ell|2n}$,
we show that the space $\Hom_{\U_q(\fg)}(V_q^{\otimes r}, V_q^{\otimes s})$ is
spanned by images of ribbon graphs if $r+s< 2\ell(2n+1)$. The proofs involve an equivalence of module categories for 
two versions of the quantisation of $\U(\fg)$.
\end{abstract}
\maketitle

\tableofcontents

\section{Introduction}\label{sect:intro}

Quantum supergroups \cite{BGZ, Y91, ZGB91b} are a class of quasi-triangular \cite{D86}
Hopf superalgebras introduced in the early 90s,
which have since been studied quite extensively; see e.g., \cite{L, MZ, WZ, Z92, Z92b, Z93, Zo98}
for results on their finite dimensional irreducible representations.
Quantum supergroups have been applied to obtain interesting results in several areas,
most notably,  in the study of Yang-Baxter type integrable models \cite{BGZ, BGLZ, ZBG91},
construction of topological invariants of knots and $3$-manifolds \cite{LGZ, Z92a, Z94, Z95}
and development of quantum supergeometry \cite{Z98, Z04}.
It is the quasi triangular Hopf superalgebraic structure \cite{GZB, KT, Y91}
that renders quantum supergroups so useful in so many areas.
A brief review of early works on the theory and applications of
quantum supergroups can be found in \cite{Z98}. Partially successful constructions
of crystal and canonical bases for quantum supergroups were given in \cite{BKK, {Du15}, MZ, ZZ05, ZZ06, Z, Zo99, Zo01}.

This paper develops the invariant theory of the quantum supergroups associated with the classical series of 
Lie superalgebras $\gl_{m|n}$ and $\osp_{m|2n}$ \cite{Kac, S}.
Invariant theory lies at the foundation of symmetry in physics; the very concept of a symmetry is an invariant theoretical notion. 
A quantum system is said to have a certain symmetry if some algebraic structure, e.g., a Lie (super)group or quantum  (super)group,  
acts on the Hilbert space and on the algebra of operators on the Hilbert space of the quantum system in such a way 
that the Hamiltonian of the system is an invariant of the structure. 
It is of crucial importance to determine the invariants of such actions in order to understand physical properties of the system.  
This underscores the relevance of this paper to fundamental physics.

We shall generalise to the quantum supergroup context results on the invariant theory of ordinary quantum groups \cite{D86, J} 
obtained in the endomorphsim algebra setting in \cite{LZ06}. 
We will do this by following the categorical approach to invariant theory developed in \cite{LZ12, LZ14a, LZ14b} for the orthogonal group 
and orthosymplectic supergroup. The ground work for this program has already been done in \cite{Z95, Z02}
in the process of constructing invariants of knots and $3$-manifolds using 
quantum supergroups.
In that work, a braided tensor functor from the category of coloured ribbon graphs to the category of
finite dimensional representations of the quantum supergroup \cite[Theorem 1]{Z95} played a crucial role.
It is a special case of this functor which provides the natural pathway to the invariant theory of the quantum supergroups.

Let $\U_q(\F)$ denote either the quantum general linear supergroup or
(modified) quantum orthosymplectic supergroup
over the field $\F:=\C(q)$ of rational functions of $q$,
and let $V_q$ be the natural module for $\U_q(\F)$.
Denote by $\cT_\fg(\F)$ the corresponding category of tensor modules for $\U_q(\F)$, that is,
the full subcategory of the category of finite dimensional
$\Z_2$-graded $\U_q(\F)$-modules whose objects are repeated tensor products
of $V_q$ and its dual  $V^*_q$ ($V^*_q\cong V_q$
in the case of quantum $\osp_{m|2n}$).
There is a unique braided tensor
functor $\cF_{q, \F}$ from the category of ribbon graphs (or non-directed ribbon graphs) to $\cT_\fg(\F)$,
which preserves ribbon category structures.

When $\U_q(\F)$ is quantum $\gl_{m|n}$ or the modified quantum supergroup
of $\osp_{2\ell +1|2n}$, Theorems \ref{thm:main2-gl} and \ref{thm:main2-osp} state  that
the braided tensor functor $\cF_{q, \F}$ is full. This statement is the first fundamental theorem
(FFT) of invariant theory
for these quantum supergroups in the categorical formulation of \cite{LZ12, LZ14a, LZ14b}.
It implies that the endomorphism algebra $\End_{\U_q(\F)}(V_q^{\otimes r})$
is the representation of the braid group on $r$-strings generated by the $R$-matrix of the natural $\U_q(\F)$-module.
As in the case of ordinary quantum groups, this representation factors through
the Hecke algebra if $\U_q(\F)$ is quantum $\gl_{m|n}$ (see Proposition \ref{prop:Hecke}),
and the Birman-Wenzl-Murakami algebra
with appropriate parameters if $\U_q(\F)$ is quantum $\osp_{2\ell+1|2n}$
(see Proposition \ref{prop:BMW}).

For the quantum supergroup of $\osp_{2\ell|2n}$,
we show in Proposition \ref{prop:main-even-osp} that the space
$\Hom_{\U_q(\F)}(V_q^{\otimes r}, V_q^{\otimes s})$ is
spanned by images of ribbon graphs under $\cF_{q, \F}$ if $r+s< 2\ell(2n+1)$.
This implies that for $r<\ell(2n+1)$, $\End_{\U_q(\F)}(V_q^{\otimes r})$
is the image  of the Birman-Wenzl-Murakami algebra of degree $r$
in the representation generated by the $R$-matrix of the natural module
(see Theorem \ref{prop:BMW}).

Our study is based on the deformation theoretical treatment of quantum supergroups  \cite{G06, G07, SZ98, Z02} 
and makes essential use of the Etingof-Kazhdan quantisation \cite{EK-I, EK-II}.
It was shown in \cite{Z02} (also see \cite{G06, G07}) that
the universal enveloping superalgebra $\U(\fg; T)$ over the power series ring $T=\C[[t]]$ can be endowed with the structure of 
a braided quasi Hopf superalgebra, where the universal $R$-matrix was constructed from  the quadratic Casimir operator and 
the associator was obtained by using solutions of a KZ equation following the strategy of Drinfeld \cite{D90}.
The Drinfeld-Jimbo quantum supergroup $\U_q(\fg; T)$
\cite{BGZ, Y91, ZGB91b} over $T$ is the Etingof-Kazhdan quantisation of $\U(\fg; T)$ \cite{G06, G07}. 
It follows from a general property of the Etingof-Kazhdan quantisation that $\U_q(\fg; T)$ and $\U(\fg; T)$ 
are equivalent (cf. Definition \ref{def:equiv-ribbon-quasi-Hopf}) as braided quasi Hopf superalgebras.

This result enables one to bring the invariant theory for classical supergroups developed in \cite{DLZ, LZ14a, LZ14b, LZ15}
into the quantum supergroup setting.  One translates the first fundamental theorems (FFTs)
for classical supergroups \cite{DLZ, LZ14a}
to (modified) $\U(\fg; T)$.
The equivalence between $\U(\fg; T)$ and $\U_q(\fg; T)$ as braided quasi Hopf superalgebras  in turn enables us to obtain 
an FFT for $\U_q(\fg; T)$. The FFT remains valid when we change scalars from $T$ to its fraction field $T_t$, the formal Laurent series ring.
Thus Theorem \ref{thm:main2-gl}, Theorem \ref{thm:main2-osp} and Proposition \ref{prop:main-even-osp} essentially
follow from the fact that the specialisation of $\U_q(\F)$
to $T_t$ (by $\F\to T_t$, $q\mapsto \exp(t/2)$) is a dense subalgebra of
the Drinfeld-Jimbo quantum supergroup over $T_t$ completed in the $t$-adic topology.

We mention the work \cite{LZZ}, which is closely related to \cite{LZ06},  on invariants of 
quantised coordinate rings of modules over ordinary quantum groups.  A similar investigation 
on invariants of quantised coordinate rings of modules over the quantum supergroups will 
be carried out in a future publication.  
Another aspect of the invariant theory of quantum supergroups 
studies the centers of the quantum supergroups.  In \cite{Z92, ZG},  generators of the centers  
were constructed using the method developed in \cite{GZB91, ZGB91a}.
The same method was employed to construct knot invariants using quantum groups and
quantum supergroups in \cite{LGZ, Z92a, ZGB91a}.

The organisation of the paper is as follows.  Section \ref{sect:quea} gives a
deformation theoretical treatment of quantum supergroups,
where we also explain how to place quantum supergroups in the framework of
Etingof-Kazhdan quantisation \cite{EK-I, EK-II}.
Sections \ref{sect:inv-theory-gl} and \ref{sect:inv-theory-osp}
contain the main results on the invariant theory of the quantum general linear supergroup and
quantum orthosymplectic supergroup respectively, in particular, the FFTs.
The two appendices contain basic facts on braided tensor categories and braided quasi Hopf  superalgebras,
which are used throughout the paper. 


\section{Quantum supergroups}\label{sect:quea}

\subsection{The power series and Laurent series rings}
In some parts of this paper, we will work over the power series ring in the indeterminate $t$
\[
T:=\C[[t]]=\left\{\left.\sum_{i=0}^\infty f_i t^i \right| f_i\in\C\right\}
\]
endowed with the $t$-adic topology. We can regard
$T$ as the inverse limit of an inverse system defined in the following way.
Let $\C[t]$ be the ring of polynomials in $t$, and set $K_n=\C[t]/(t^n)$ for all $n\in\Z_+$.
We have the inverse system $(p_n: K_n)_{n>0}$ with
$p_n$ being the natural projection $p_n: K_n\to K_{n-1}$.
Then $T=\underleftarrow\lim T_n$ equipped with the inverse limit topology.
Given any $T$-module $M$, the quotient modules $M_n=M/(t^n M)$ and natural projections $M_n \to M_{n-1}$
form an inverse system. The $t$-adic completion of $M$ is the inverse limit $\underleftarrow\lim M_n$ equipped
with the inverse limit topology.
Given $T$-modules $M$ and $N$, we define the  topological tensor product $M\hat\otimes_T N$
to be the $t$-adic completion of $M\otimes_T N$.
For a $\C$-vector space $V$, denote by $V[[t]]$ the complex
vector space of formal series
$\sum_{n\ge 0} v_n t^n$, where $v_n\in V$.  We give
$V[[t]]$ the (obvious) $T$-module structure defined for any
$f=\sum_{k\ge 0} f_k t^k\in T$ and $v(t)=\sum_{n\ge 0} v_n t^n\in V[[t]]$ by
$f v(t)= \sum_{m\ge 0}(\sum_{k=0}^m f_k v_{m-k}) t^m$.  Call such $T$-modules
topologically free.
If $V$ and $W$ are $\C$-vector spaces,
$V[[t]]\hat\otimes_T W[[t]]=(V\otimes_\C W)[[t]]$.

The definition of any type of superalgebra can be extended to
the topological setting of $T=\C[[t]]$ by replacing the algebraic tensor product 
by the topological tensor product,
leading to the notion of topological superalgebras of that type.

The Laurent series ring
\[
T_t:=\left\{\left.\sum_{i\in\Z} f_i t^i \right| f_i\in\C,  \  f_j=0 \text{ \ if $j<<0$}\right\}
\]
is the quotient field of $T$.
For any topological $T$-module $M$, we let $M_t:= M\hat\otimes_T T_t$. 
If  $M$ is a topologically free module of finite rank $r$ over $T$, then $M_t$ is an $r$-dimensional vector space over $T_t$.

\smallskip

\noindent
{\bf Notation}.
Throughout the paper,  $\hat\otimes_T$ will simply be written as $\otimes_T$, and superalgebras 
over $T$ are understood to be topological.

\subsection{Quantum universal enveloping superalgebras}\label{sect:uea}
To consider quantum supergroups  \cite{BGZ, Y91, ZGB91b}
from a deformation theoretical point of view \cite{ G06, G07, SZ98, Z02},
we will need the notion of quasi Hopf superalgebras \cite[\S II]{Z02},
which will be briefly discussed in Appendix \ref{append:quasi-Hopf}.

Let $\fg=\fg_{\bar 0}\oplus\fg_{\bar 1}$ be a complex Lie superalgebra \cite{Kac, S} with
even subspace $\fg_{\bar 0}$ and odd subspace $\fg_{\bar 1}$.
Then  $\fg_{\bar 0}$ is a Lie subalgebra called the even subalgebra,
and the odd subspace is a $\fg_{\bar 0}$-module under the adjoint action.
Kac \cite{Kac} called $\fg$ classical if $\fg_{\bar 0}$
is a reductive Lie algebra and $\fg_{\bar 1}$ is
a semi-simple $\fg_{\bar 0}$-module \cite{Kac, S}.
If  in addition $\fg$ admits a non-degenerate invariant bilinear form,
it is called contragredient.

Given a simple contragredient Lie superalgebra $\fg$ over $\C$,
we choose a
Borel subalgebra $\fb\subseteq\fg$ containing a Cartan subalgebra $\fh$.
The bilinear form above is then unique up to multiplication by a non-zero scalar, and 
its restriction to $\fh$ remains non-degenerate.
Using this restriction to identify $\fh$ with $\fh^*$, we obtain a
non-degenerate symmetric bilinear form $(\, , \, )$ on the dual space $\fh^*$ of $\fh$.
Let $\Pi_{\fb}=\{\alpha_i | i=1, 2, ..., r\}$ be
the set of simple roots relative to $\fb$, and let $\Psi$ be the set of all roots.
Denote by $\phi\subset\Pi_\fb$ the subset consisting of all the odd simple roots.
Let $\ell_m^2$ be the minimum of $|(\beta, \beta)|$ over
all non-isotropic $\beta\in\Psi$ if $\fg\ne D(2, 1; \alpha)$. If
$\fg$ is $D(2, 1; \alpha)$, let $\ell_m^2$ be the minimum of all
$|(\beta, \beta)|>0$ ($\beta\in\Psi$), which are independent of
the arbitrary parameter $\alpha$. Let
$d_i=\frac{(\alpha_i, \alpha_i)}{2}$ if $(\alpha_i, \alpha_i)\ne 0$, and
$d_i=\frac{\ell_m^2}{2^\kappa}$ if $(\alpha_i, \alpha_i)= 0$,
where $\kappa=0$ if $\fg$ is of type $B$ and $\kappa=1$ otherwise.
Introduce the matrices
$D=\text{diag}(d_1, \dots, d_r)$ and
$B=(b_{i j})_{i, j=1}^r$ with $b_{i j} = (\alpha_i, \alpha_j)$.
Then the Cartan matrix $A$ associated to the set of simple
roots $\Pi_\fb$ is defined by
$
A=D^{-1} B.
$
Note that $A$ has rank $r$ except when $\fg=A(n, n)$ with $n=\frac{r-1}{2}$
 for odd $r$.
In the latter case, there exists a unique integral row vector $J=(J_1,\dots,J_r)$ of length $r=2n+1$
with $J_1=1$ such that $JA=0$.

Given the Borel subalgebra $\fb$, the Lie superalgebra $\fg$ has the 
following Serre type presentation \cite{Z13}.
The generators are  $e^0_i, f^0_i$ and $h^0_i$ ($1\le i\le r$)
with $e^0_i$ and $f^0_i$ odd when
$\alpha_i\in\phi$ and even otherwise. The relations are: \\
$\bullet$ quadratic relations
\begin{eqnarray*}
\begin{aligned}
&{[}h^0_i, h^0_j] =0, \\
&[h^0_i, e^0_j] =a_{i j} e^0_j, \quad [h^0_i, f^0_j] =-a_{i j} f^0_{j}, \\
&[e^0_i, f^0_j] =\delta_{i j} h^0_i, \quad \forall i, j;
\end{aligned}
\end{eqnarray*}
$\bullet$ standard Serre relations
\begin{eqnarray*}
\begin{aligned}
&[e^0_t, e^0_t]=0, \quad  [f^0_t, f^0_t]=0, \quad \text{for $a_{t t}=0$}\\
&(ad_{e^0_i})^{1-a_{i j}}(e^0_j)=0, \quad (ad_{f^0_i})^{1-a_{i j}}(f^0_j)=0, \\
& \text{for  $i\ne j$ with $a_{i i}\ne 0$ or $a_{i j}=0$};\quad \text{and}
\end{aligned}
\end{eqnarray*}
$\bullet$ higher order Serre relations given in \cite[Theorem 3.3]{Z13}; and \\
$\bullet$ the additional linear relation
$
\sum_{i=1}^r J_i h^0_i =0
$
if $\fg=A(\frac{r-1}{2},\frac{r-1}{2})$ for odd $r$.

\begin{remark}
The general linear Lie superalgebra $\gl_{m|n}$ can be similarly described.
\end{remark}

\begin{remark} In general, given a Cartan subalgebra $\fh$ of a simple Lie superalgebra $\fg$,
there are many
Borel subalgebras $\fb\supseteq\fh$ which are not Weyl group conjugate,
and the Serre type presentations corresponding to
different choices of Borel subalgebras can be very different.
\end{remark}

Now take $\fg$ to be either a simple contragedient Lie superalgebra or $\gl_{m|n}$.
Denote by $\U(\fg; \C)$ the universal enveloping superalgebra of $\fg$ over $\C$,
and let
$\U(\fg; T)=\U(\fg; \C)\otimes_\C T$ be the universal enveloping superalgebra
over the power series ring $T$. Then $\U(\fg; T)$ has
a standard Hopf superalgebra
structure with the co-multiplication $\Delta$, co-unit $\epsilon$ and antipode $S$
respectively given by
\begin{eqnarray}\label{eq:co-prod}
\Delta(X)= X\otimes 1 + 1\otimes X, \quad
\epsilon(X)=0, \quad
S(X)= -X, \quad \forall X\in\g.
\end{eqnarray}
Note the co-commutativity of $\U(\fg; T)$, which we express by writing  $\Delta^{op}=\Delta$.

If $(X_a)$, $(Y_a)$ are dual bases of $\fg$ with respect to the 
invariant form, i.e. $(Y_a,X_b)=\delta_{ab}$ for all $a,b$, define 
$C=\sum_a X_a\ot Y_a$, and let 
$\omega=\sum_a X_aY_a$ be
the quadratic Casimir operator of $\g$; it is
the central element of the universal enveloping superalgebra $\U(\fg; \C)$ which acts with eigenvalue
$(\lambda+2\rho,\ \lambda)$ on any simple $\U(\fg; \C)$-module
with highest weight $\lambda$. The following result is clear.

\begin{lemma} \label{lem:Casimir}\label{eq:Casimir}
Let $C={\frac{1}{2}} \left(\Delta(\omega)- \omega\otimes 1 - 1\otimes \omega\right)$. Then
\[
C\Delta(x) - \Delta(x) C=0, \quad \forall x\in \U(\fg; \C).
\]
 \end{lemma}
As above, write $C=\sum_{a=1}^{\dim\fg} X_a\otimes Y_a$ with $(Y_a), (X_a)\subset\fg$ a pair of dual bases,
and introduce the following elements of $\U(\fg; \C)\otimes\U(\fg; \C)\otimes\U(\fg; \C)$:
\[
C_{12}=C\otimes 1, \quad C_{23}=1\otimes C, \quad C_{13}=\sum_a X_a\otimes 1\otimes Y_a.
\]
We will still loosely refer to $C$ and the $C_{i j}$ as quadratic Casimir operators.
\begin{lemma} \label{lem:four-term}
The quadratic Casimir operators satisfy the following relations
\begin{eqnarray}
\begin{aligned}
&(\id\otimes \Delta) C= C_{1 2} + C_{1 3}, \quad
(\Delta\otimes \id) C = C_{1 3} +C_{2 3},
\end{aligned}\\
\begin{aligned}
&{[C_{1 2}, \ C_{1 3} +C_{2 3}]}=0, \quad [C_{1 2}+C_{1 3}, C_{2 3}]=0.\phantom{XXX}
\end{aligned}\label{eq:4-term}
\end{eqnarray}
\end{lemma}
The relations \eqref{eq:4-term} are known as the `four-term relations';
they are essential for
the construction of the associator to be described below.

Now consider the following power series (see \cite[(3.3)]{Z02})
\begin{eqnarray}\label{eq:R}
R:=\exp( t C/2)=\sum_{n=0}^\infty \frac{(tC/2)^n}{n!}.
\end{eqnarray}
As $\U(\fg; T)$ is co-commutative, it is immediate from Lemma \ref{eq:Casimir}
that
\[R \Delta(x) = \Delta^{op}(x) R, \quad \forall x\in\U(\fg; T).\]
Our aim is to interpret $R$ as a  universal $R$-matrix for $\U(\fg; T)$. This
requires us to work in the setting of braided quasi Hopf superalgebras (cf. Section \ref{sect:bqHsa}).

As shown in \cite{Z02}, Cartier's theory \cite{C} of
infinitesimal symmetric categories applies to $\U(\fg; T)$.
This follows from Lemma \ref{lem:four-term}, that is,  the quadratic Casimir operators
$C_{i j}$ satisfy the four-term relations \eqref{eq:4-term}.
Thus Drinfeld's construction \cite{D92} of associators for enveloping algebras of semi-simple Lie algebras
can be directly generalized to the present context to construct
the desired associator $\Phi=\Phi_{KZ}\in \U(\fg; T)^{\otimes 3}$.

To explain the construction (see \cite[\S III.B]{Z02}),
we consider the following differential
equation on $\C\backslash\{0, 1\}$
\begin{eqnarray}\label{eq:KZ}
{\frac{d G(z)}{d z}}&=& {\frac{t}{2\pi i}}\left( { \frac{C_{1 2}}{z}}
+ { \frac{C_{2 3}}{z-1}}\right) G(z)
\end{eqnarray}
for the analytic function $G: \C\backslash\{0, 1\}\rightarrow
\U(\fg; T)^{\otimes 3}$. This is a special case of the celebrated
Knizhnik-Zamolodchikov equations, which first arose in the
context of Wess-Zumino-Witten models of two dimensional conformal
quantum field theory.  The classical theory of Fuchsian differential
equations guarantees the existence and uniqueness of solutions
$G_0$ and $G_1$ of equation \eqref{eq:KZ} with the asymptotes
\begin{eqnarray*}
\begin{aligned}
&G_0(z)\to z^{ \frac{t C_{1 2}}{2\pi i}},  \quad z\to 0; \quad
&G_1(z)\to (1-z)^{ \frac{t C_{ 2 3}}{2\pi i}}, \quad z\to 1.
\end{aligned}
\end{eqnarray*}
Furthermore, the two solutions can only differ by a $z$-independent
factor
\begin{eqnarray}\label{eq:Phi}
\Phi_{KZ}&:=& (G_0(z))^{-1} G_1(z). \label{KZ}
\end{eqnarray}
The factor $\Phi_{KZ}$ can be expressed as a power series in $t$ with coefficients
which are linear combinations of Lie words (that is, repeated commutators)
in $C_{12}$ and $C_{2 3}$,
whose coefficients involve Chen's iterated integrals.

Now $\Phi_{KZ}$ in \eqref{eq:Phi} yields the desired associator for $\U(\fg; T)$
as it satisfies the defining
relations \eqref{eq:associator} and \eqref{eq:braid}.
Furthermore, given  a simple contragredient
Lie superalgebra, $\Phi_{KZ}$ is the unique associator for the co-multiplication \eqref{eq:co-prod}
and universal $R$-matrix \eqref{eq:R}.

We note next that $\U(\fg; T)$ has the structure of a braided quasi Hopf superalgebra \cite{Z02}.
Its antipode triple (see \eqref{eq:antipode}) $(S, \alpha, \beta)$ can be taken to be
\begin{eqnarray}\label{eq:beta}
\alpha=1, \quad \beta^{-1}= m (m\otimes\id) (\id\otimes S\otimes\id)\Phi_{KZ},
\end{eqnarray}
where $m$ denotes multiplication in $\U(\fg; T)$.

Let $u\in \U(\fg; T)$ be defined by equation \eqref{eq:u}. Then $uS(u)=1+o(t)$,
and thus there exists
$v=1+o(t)$ in $\U(\fg; T)$ such that $v^2=uS(u)$. Hence
\begin{eqnarray}\label{eq:U-T}
(\U(\fg; T), \Delta, \epsilon, \Phi_{KZ}, S, \alpha, \beta, R, v)
\end{eqnarray}
is a ribbon quasi Hopf superalgebra, which will be called the {\em quantum enveloping superalgebra}
of $\fg$.

Denote by $\U(\fg; T)\text{-Mod}$ the category of $\U(\fg; T)$-modules which are
topologically free over $T$, and by $\U(\fg; T)\text{-mod}$
the full subcategory with objects which are free $T$-modules of finite rank.

\begin{remark}\label{rem:non-simple}
The above construction of the ribbon quasi Hopf superalgebra  works for any finite dimensional
Lie superalgebra $\fg$ with a quadratic element $C\in\fg\otimes\fg$ satisfying Lemma \ref{eq:Casimir}
and Lemma \ref{lem:four-term}.
\end{remark}

\begin{remark}
The ribbon quasi Hopf superalgebras \eqref{eq:U-T} were introduced
and applied in \cite{Z02} to construct a class of Vassiliev invariants of knots
\cite{BL, Ko}. These invariants are the coefficients in the power series expansions in $t$ 
of the quantum supergroup invariants of knots constructed in \cite{LGZ, Z92}.
They, as well as related $3$-manifold invariants \cite{Z94, Z95}
have recently been investigated in \cite{MW} from
a quantum field theoretical point of view \cite{W}.
\end{remark}

\subsection{Quantum supergroups}\label{sect:DJ-alg}

Given any simple contragredient Lie superalgebra $\fg$ of rank $r$, we make
a choice $\fg\supset\fb\supset\fh$ of Borel and Cartan subalgebras. Let  $(\Pi_\fb, \phi, A)$
denote the corresponding root datum, associated to which are the Cartan matrix
$A=(a_{i j})$ and diagonal matrix $D=\text{diag}(d_1, \dots, d_r)$ defined in Section \ref{sect:uea}.
Let
$
q=\exp(t/2)$ and $q_i=q^{d_i},
$
then $q_i^{a_{i j}} = q_j^{a_{j i}}$ for all $i, j=1, 2, \dots, r$.

A Drinfeld-Jimbo quantum supergroup $\U_q(\fg, \phi; T)$
was introduced in \cite{BGZ, Y91, ZGB91b};
it is a deformation of the universal enveloping superalgebra
of $\fg$ as a Hopf superalgebra.
The quantum supergroup
$\U_q(\fg, \phi; T)$
is generated as an associative superalgebra by the homogeneous
generators $e_i, f_i, h_i$  $(i=1, 2, ..., r)$,
where $e_i$ and $f_i$ are odd for $\alpha_i\in\phi$ and all the other generators are even, subject to the
following relations:
\begin{eqnarray}\label{eq:q-group}
\begin{aligned}
&k_i k_i^{-1} = k_i^{-1}k_i =1, \quad
    k_i k_j= k_j k_i, \\
&k_i e_j k_i^{-1} = q_i^{a_{i j}} e_j, \quad
    k_i f_j k_i^{-1} = q_i^{-a_{i j}} f_j,\\
&e_if_j-(-1)^{[e_i][f_j]}  f_j e_i=\delta_{i j} \frac{k_i - k_i^{-1}}{q_i - q_i^{-1}}, \\
&(e_s)^2=0, \quad (f_s)^2=0, \quad \text{if\ } a_{s s}=0,\\
& Ad_{e_i}^{1-a_{i j}} (e_j)=0, \quad  Ad_{f_i}^{1-a_{i j}} (f_j)=0, \\
&\text{for  $i\ne j$ with $a_{i i}\ne 0$ or $a_{i j}=0$},\\
&\text{higher order quantum Serre relations, and }\\
&\text{$\sum_{i=1}^r J_i h_i=0$ if $\fg=A(\frac{r-1}{2},\frac{r-1}{2})$ for odd $r$},
\end{aligned}
\end{eqnarray}
where
$
k_i:=\sum_{p=0}^\infty \frac{(t d_i h_i/2)^p}{p!},
$
and the quantum adjoint operations are defined by
\begin{eqnarray*}
Ad_{e_i}(x)&=& e_i x - (-1)^{[e_i][x]} q^{h_i} x q^{-h_i} e_i, \\
Ad_{f_i}(x)&=& f_i x - (-1)^{[f_i][x]} q^{-h_i} x q^{h_i} f_i.
\end{eqnarray*}

For the distinguished root data \cite[Appendix A.2.1]{Z13},
higher order Serre relations appear if the Dynkin diagram contains a sub-diagram of the following types:
\begin{enumerate}
\item
\begin{picture}(85, 15)(0, 7)
\put(10, 10){\circle{10}} \put(0, -2){\tiny $s-1$} \put(15,
10){\line(1, 0){20}} { \color{gray} \put(35, 10){\circle*{10}} }
\put(30, -2){\tiny $s$} \put(40, 10){\line(1, 0){20}} \put(65,
10){\circle{10}} \put(55, -2){\tiny $s+1$} \put(72, 7){,}
\end{picture}
the higher order quantum Serre relations are
\vspace{2mm}
\begin{eqnarray*}
\begin{aligned}
e_s E_{s-1; s; s+1} +  E_{s-1; s; s+1}e_s=0,\quad
f_s F_{s-1; s; s+1} +  F_{s-1; s; s+1}f_s=0;
\end{aligned}
\end{eqnarray*}

\item \begin{picture}(85, 15)(0, 7)
\put(10, 10){\circle{10}}
\put(0, -2){\tiny $s-1$}
\put(15, 10){\line(1, 0){20}}
{\color{gray} \put(35, 10){\circle*{10}} }
\put(30, -2){\tiny $s$}
\put(40, 11){\line(1, 0){20}}
\put(40, 9){\line(1, 0){20}}
\put(50, 7){$>$}
\put(65, 10){\circle{10}}
\put(58, -2){\tiny $s+1$}
\put(72, 7){,}
\end{picture}
the higher order quantum Serre relations are
\vspace{2mm}
\begin{eqnarray*}
\begin{aligned}
e_s E_{s-1; s; s+1} +  E_{s-1; s; s+1}e_s=0,\quad
f_s F_{s-1; s; s+1} +  F_{s-1; s; s+1}f_s=0;
\end{aligned}
\end{eqnarray*}

\item \begin{picture}(100, 25)(0, 7)
\put(10, 10){\circle{10}}
\put(0, -2){\tiny $s-1$}
\put(15, 10){\line(1, 0){20}}
{ \color{gray} \put(35, 10){\circle*{10}} }
\put(30, -2){\tiny $s$}
\put(40, 10){\line(2, 1){20}}
\put(40, 10){\line(2, -1){20}}
\put(65, 20){\circle{10}}
\put(72, 18){\tiny
$s+1$} \put(65, 0){\circle{10}}
\put(72, -3){\tiny $s+2$}
\put(90, 7){,}
\end{picture}
the higher order quantum Serre relations are
\vspace{4mm}
\begin{eqnarray*}
\begin{aligned}
e_s E_{s-1; s; s+1} +  E_{s-1; s; s+1}e_s=0,\quad
f_s F_{s-1; s; s+1} +  F_{s-1; s; s+1}f_s=0,\\
e_s E_{s-1; s; s+2} +  E_{s-1; s; s+2}e_s=0, \quad
f_s F_{s-1; s; s+2} +  F_{s-1; s; s+2}f_s=0;
\end{aligned}
\end{eqnarray*}
where
\[
\begin{aligned}
E_{s-1; s; j} =& e_{s-1}(e_s e_{j}-q_{j}^{a_{j s}} e_{j} e_s)-
q_{s-1}^{a_{s-1, s}}(e_s e_{j}-q_{j}^{a_{j s}} e_{j} e_s) e_{s-1},\\
F_{s-1; s; j} =& f_{s-1}(f_s f_{j}-q_{j}^{a_{j s}} f_{j} f_s)-
q_{s-1}^{a_{s-1, s}}(f_s f_{j}-q_{j}^{a_{j s}} f_{j} f_s) f_{s-1}.
\end{aligned}
\]
\end{enumerate}
For the other root data of $\fg$,  the higher order quantum Serre relations vary
considerably with the choice of the root datum, thus we will not spell them out explicitly here.

It is known that $\U_q(\fg, \phi; T)$ has the structure of a Hopf superalgebra,
with co-associative co-multiplication $\Delta_q$, co-unit $\epsilon_q$
and antipode $S_q$ respectively given by
\begin{eqnarray*}
\begin{aligned}
&\Delta_q(h_i)=h_i\otimes 1 + 1\otimes h_i, \quad  \Delta_q(e_i)=e_i\otimes q^{h_i} + 1\otimes e_i,\\
& \Delta_q(f_i)=f_i\otimes 1 + q^{-h_i}\otimes f_i; \\
&  \epsilon_q(h_i)=\epsilon_q(e_i)=\epsilon_q(f_i)=0; \\
& S_q(h_i)=-h_i, \quad S_q(e_i)= -e_i q^{-h_i},\quad S_q(f_i)= - q^{h_i}f_i.
\end{aligned}
\end{eqnarray*}

Let $2\rho$ be  the sum of the even positive roots minus
the sum of the odd positive roots of $\g$.
Let $2 h_\rho$ denote the linear combination of
the $h_i$ such that $[h_\rho, e_i] =(\rho, \alpha_i) e_i$ for
all $i$. Set $K = q^{2 h_\rho}.$
Then $S_q^2(x)= K x K^{-1}$ for all $x\in\U_q(\fg, \phi; T)$.
We shall regard $\U_q(\fg, \phi; T)$ as a quasi Hopf superalgebra with
$\Phi=1\otimes 1\otimes 1 $ and
$ \alpha=\beta=1.$

The most important property
of $\U_q(\fg, \phi; T)$ is its braiding, that is, the existence of a
universal R-matrix $R_q$ \cite{GZB, KT, Y91}, which satisfies the relations
\eqref{eq:braid} with $\Phi=1\otimes 1\otimes 1$.  The explicit form of $R_q$
is in principle known.  One has
$R_q=1\otimes 1 + \frac{t}{2} r + o(t^2)$,
 where $r$ is the classical $r$-matrix \cite{ZGB91b}.
Thus  $\U_q(\fg, \phi; T)$ has the structure of a quasi triangular Hopf superalgebra.

We denote by $\U_q(\fb; T)$ the Hopf subalgebra of $\U_q(\fg, \phi; T)$
generated by the elements $e_i, h_i$ for all $i$,
and by $\U_q(\fh; T)$ the Hopf subalgebra of $\U_q(\fb; T)$ generated by the elements $ h_i$.
Then $\U_q(\fg, \phi; T)$ is the quantum double of  $\U_q(\fb; T)$ \cite{GZB, G06}.

Let $u_q\in \U_q(\fg, \phi; T)$ be the element defined by equation \eqref{eq:u}.
Then $u_qS(u_q) = 1 + o(t)$, hence
there exists an element $v_q=1+o(t)$ such that $v^2=uS(u)$.
Thus
\begin{eqnarray}\label{eq:Uqg}
(\U_q(\fg, \phi; T), \Delta_q, \epsilon_q, S_q, R_q, v_q)
\end{eqnarray}
is a ribbon Hopf superalgebra, which is the quantum supergroup of $\fg$
in the root datum.

\begin{remark}
One may similarly define the quantum general linear supergroup \cite{Z93}. Details will be given in Section \ref{sect:q-GL}.
A different construction  of $\U_q(\gl_{m|n})$ as the commutant of a Hecke algebra super action on $V^{\ot r}$ may be found in \cite{Du14}.
\end{remark}

\subsection{Isomorphism theorem}
\subsubsection{The Etingof-Kazhdan quantisation in a nutshell}\label{sect:E-K-quantisation}

In a series of papers  (the most relevant to us are  \cite{EK-I, EK-II}),
Etingof and Kazhdan  developed a functorial method of quantisation, which turns 
quasi-triangular Lie bialgebras to quasi-triangular Hopf algebras \cite{D86}. 
Their construction also commutes with Drinfeld's double constructions 
for Lie bialgebras and Hopf algebras
\cite{D86}. Their method applies in the category of $\Z_2$-graded vector 
spaces, thus can be used to quantise Lie super bialgebras \cite{GZB}.

Let $\fg$ be either a simple contragredient Lie superalgebra or $\gl_{m|n}$ 
with the quasi-triangular Lie super bialgebra structure described in \cite[Example 8]{GZB}. 
The corresponding classical $r$-matrix $r$ of $\fg$ (see above) has the property that $r+\tau(r)=2C$,
where $\tau$ is the $\Z_2$-graded permutation defined by \eqref{eq:tau}.
Let $\cM_\fg$ be the Drinfeld category of $\fg$,  whose objects are 
$\U(\fg; T)$-modules which are topologically free over $T$,  
and whose morphisms are defined in the following way. 
Corresponding to the objects $V[[t]]$ and $W[[t]]$, 
we have the $\fg$-modules $V$ and $W$ over $\C$. 
Then $\Hom_{\cM_\fg}(V[[t], W[[t]]):=\Hom_\fg(V, W)[[t]]$. 
This makes $\cM_\fg$ into a tensor category with associativity constraint 
given by the associator $\Phi_{KZ}$.

Given a Borel subalgebra $\fb$ of $\fg$, we denote by $\bar\fb$ the 
opposite Borel subalgebra. Then we have the Verma module
$M_+=\U(\fg; \C)\otimes_{\U(\fb; \C)} \C_0$ with highest weight $0$, 
and opposite Verma module $M_-=\U(\fg; \C)\otimes_{\U(\bar\fb; \C)} \C_0$
with lowest weight $0$. Let $\cF$ be the functor from $\cM_\fg$ to 
the category of $\cV$ topologically free $T$-modules (with continuous morphisms) 
defined by $\cF(W[[t]])=\Hom_\fg(M_+\otimes M_-, W)[[t]]$. There exist a family of isomorphisms
$J_{V_1, V_2}: \cF(V_1)\otimes \cF(V_2)\longrightarrow \cF(V_1\otimes V_2)$
natural in $V_1$ and $V_2$, making $\cF$ into a tensor functor 
\cite[Proposition 2.2]{EK-I} (in the notation of Appendix \ref{sect:tensor}, 
$\phi_2$  is $J$,  and $\phi_0$ is the obvious map).   Since the 
modules $M_\pm$ are defined with respect to $\fb$ and $\bar\fb$ respectively,  
$\cF$ depends on the choice of the Borel subalgebra.

The following results were all proved in \cite[\S 3]{EK-I}
(see \cite[Propositions 3.6, 3.7]{EK-I} in particular).
The algebra of endomorphisms $\End(\cF)$ of the tensor functor $\cF$ 
has a natural braided quasi Hopf superalgebra structure, and is equivalent to
$\U(\fg; T)$ as given by \eqref{eq:U-T}.
One can turn $\End(\cF)$ into a Hopf superalgebra $\U_t^{EK}(\fg, \phi; T)$
by a gauge transformation. This Hopf superalgebra admits a quasi triangular structure, 
which in the semi-classical limit reduces to the quasi  
triangular Lie bi superalgebraic structure of $\fg$ discussed earlier.  
The quasi triangular Hopf superalgebra $\U_t^{EK}(\fg, \phi; T)$ 
is the Etingof-Kazhdan quantisation of the quasi triangular Lie super bialgebra $\fg$.

It was proved by Geer \cite[Theorem 1.1]{G07} that for any simple contragredient
 Lie superalgebra $\fg$,  the Etingof-Kazhdan quantum superalgebra
$\U_t^{EK}(\fg, \phi; T)$ is equal to
the Drinfeld-Jimbo quantum supergroup $\U_q(\fg, \phi; T)$ as quasi triangular
 Hopf superalgebra.  
Geer's proof was based on the fact that the Etingof-Kazhdan quantisation commutes 
with Drinfeld's double constructions. It generalises to $\gl_{m|n}$,
as we will see in Section \ref{sect:proof-main-gl}.

\subsubsection{Isomorphism theorem over power series ring}
We have the following result.
\begin{theorem}\label{thm:iso}\label{eq:f-F-g}
Let $\fg$ be either a finite dimensional simple Lie superalgebra or $\gl_{m|n}$.
Let $(\U_q(\fg, \phi; T), \Delta_q, \epsilon_q, S_q, R_q, v_q)$
be the quantum supergroup described in Section \ref{sect:DJ-alg}, which can be
regarded as a ribbon quasi Hopf superalgebra with a trivial associator.
Let $(\U(\fg; T), \Delta, \epsilon, \Phi_{KZ}, S, \alpha, \beta, R, v)$ be the
ribbon quasi Hopf superalgebra constructed in Section \ref{sect:uea}.
Then the two ribbon quasi Hopf superalgebras  are equivalent
(cf. Definition \ref{def:equiv-ribbon-quasi-Hopf}) for any root datum of $\fg$.
\end{theorem}
\begin{proof}
If $\fg$ is a simple Lie superalgebra, this follows from the discussion in Section \ref{sect:E-K-quantisation} above. 
The case of $\gl_{m|n}$ will be proved in
Section \ref{sect:proof-main-gl}.
\end{proof}
We denote the equivalence of Theorem \ref{thm:iso} by
\begin{eqnarray*}
(f, F, g): (\U_q(\fg, \phi; T), \Delta_q, \epsilon_q, S_q, R_q, v_q)
\longrightarrow (\U(\fg; T), \Delta, \epsilon, \Phi_{KZ}, S, \alpha, \beta, R, v),
\end{eqnarray*}
where $f: \U_q(\fg, \phi; T) \longrightarrow \U(\fg; T)$ is the superalgebra isomorphism,
$F$ the gauge transformation on $\U(\fg; T)$ and $g$ the antipode transformation.

\begin{remark}\label{rem:Cartan-deform}
One has $f(h_i)=h_i^0$ for all $i$ (see \cite[\S41, \S4.2]{G06}) and hence
the algebra isomorphism $\U_q(\fh; T)\stackrel{\sim}\longrightarrow \U(\fh; T)$.
Furthermore, $[\Delta(h), F]=0$ for all $h\in \fh$.
\end{remark}

\begin{remark}
For $\fg=\mathfrak{sl}(m|n)$, $\mathfrak{osp}(1|2n)$ or
$\mathfrak{osp}(2|2n)$,  it was shown in \cite{SZ98, SZ07} that the universal enveloping superalgebra $\U(\fg; \C)$ is rigid as associative algebra,
that is, it admits no nontrivial deformation. However,
this is not true in  general, e.g., $\U(\mathfrak{osp}(4|2); \C)$ is not rigid \cite{Z92c}.
\end{remark}

Let $\U_q(\fg, \phi; T)\text{-mod}$ (resp. $\U(\fg; T)\text{-mod}$) be the category of
$\U_q(\fg, \phi; T)$-modules (resp. $\U(\fg; T)$-modules) which are topologically $T$-free modules of finite ranks.  By Theorem \ref{thm:key},
both $\U_q(\fg, \phi; T)\text{-mod}$ and $\U(\fg; T)\text{-mod}$ are ribbon categories.
\begin{corollary}\label{cor:equiv-mod-cats}
There exists a braided tensor equivalence between $\U_q(\fg, \phi; T)\text{-mod}$
and $\U(\fg; T)\text{-mod}$, which preserves
duality and twist.
\end{corollary}
\begin{proof}
Applying Theorem \ref{thm:ribbon-cat-equiv} to $\U_q(\fg, \phi; T)\text{-mod}$ and $\U(\fg; T)\text{-mod}$,
we immediately obtain the result by using Theorem \ref{thm:iso}.
\end{proof}
We say that $\U_q(\fg, \phi; T)\text{-mod}$ and $\U(\fg; T)\text{-mod}$ are equivalent as ribbon categories.

\subsubsection{Isomorphism theorem over Laurent series ring}
\label{sect:Laurent-series}
Let
\[
\U_q(\fg, \phi; T_t):=\U_q(\fg, \phi; T)\otimes_T T_t, \quad \U(\fg; T_t)=\U(\fg; T)\otimes_T T_t,
\]
and extend $T_t$-linearly the structure maps of $\U_q(\fg, \phi; T)$ and $\U(\fg; T)$ to these superalgebras, 
and retain their universal $R$-matrices, associators etc..  Then we obtain two ribbon (quasi) Hopf superalgebras  over $T_t$.
The following result is an immediate consequence of Theorem \ref{thm:iso}.
\begin{corollary} \label{cor:iso-1} For any root datum of $\fg$, there is an equivalence of
ribbon quasi Hopf superalgebras
\begin{eqnarray*}
(\U_q(\fg, \phi; T_t), \Delta_q, \epsilon_q, S_q, R_q, v_q)
\longrightarrow (\U(\fg; T_t), \Delta, \epsilon, \Phi_{KZ}, S, \alpha, \beta, R, v).
\end{eqnarray*}
\end{corollary}
Now we have the category $\U_q(\fg, \phi; T_t)\text{-mod}$ (resp. $\U(\fg; T_t)\text{-mod}$) of finite
dimensional $\U_q(\fg, \phi; T_t)$-modules (resp. $\U(\fg; T_t)$-modules).
Corollary \ref{cor:iso-1}
immediately leads to the following result.
\begin{corollary}\label{cor:equiv-mod-cats-1}
There exists a braided tensor equivalence between $\U_q(\fg, \phi; T_t)\text{-mod}$
and $\U(\fg; T_t)\text{-mod}$, which preserves
duality and twist.
\end{corollary}

\section{Invariant theory of the quantum general linear supergroup} \label{sect:inv-theory-gl}

We now develop the invariant theory of the quantum general linear supergroup.

%
%
%
%
\subsection{The quantum universal enveloping superalgebra of $\gl_{m|n}$}\label{sect:quantum-gl}

Let us start by describing the structure of the general linear superalgebra $\mathfrak{gl}_{m|n}=\gl(\C^{m|n})$ over $\C$ in some detail.
We regard $\C^{m|n}$ as the space of column vectors of length $m+n$,
and denote its standard basis by
$B_{st}=(e_1, \dots, e_m, e_{m+1}, \dots, e_{m+n})$, that is, for each $b$, the column vector $e_b$ has entries
$(e_b)_a=\delta_{a b}$. Here $e_b$ is even if $b\le m$ and odd if $b>m$. Let $B=(v_1, v_2, \dots, v_{m+n})$
be another basis which is obtained from $B_{st}$ by re-ordering the basis elements in an {\em admissible} way,
that is, in such a way that
$e_i$ appears before $e_j$ for all $i<j\le m$, and $e_\mu$ appears before $e_\nu$ for all $\nu>\mu>m$.
Let $E_{a b}\in \End_\C(M)$ ($1\le a, b\le m+n$) be the matrix units relative to the basis $B$.
Then $E_{a b} v_c=\delta_{b c} v_a$ for all $a, b, c$.  The matrix units form a homogeneous basis of $\gl_{m|n}$ with the commutation relations
\[
[E_{a b}, E_{c d}] = \delta_{b c} E_{a d} -(-1)^{([a]+[b])([c]+[d])}\delta_{d a} E_{c b},
\]
where $[b]=0$ if $v_b$ is even and $[b]=1$ if $v_b$ is odd.

Let $\fb$ be a Borel subalgebra of $\gl_{m|n}$
consisting of upper triangular matrices and let $\fh\subset\fb$ be the Cartan subalgebra consisting of diagonal matrices.
Then the elements $E_{a a}$ ($1\le a\le m+n)$ form a basis of $\fh$.
Let  $({\mathcal E}_a\mid 1\le a\le m+n)$ be the dual basis of
$\fh^*$, that is,
\[
{\mathcal E}_a(E_{b b})=\delta_{a b}, \quad \forall a, b.
\]
The supertrace form on $\fh$ induces a bilinear form $(\ , \ )$ on $\fh^*$ such that
\begin{eqnarray}\label{eq:weight-space-norm}
({\mathcal E}_a, {\mathcal E}_b)=(-1)^{[a]}\delta_{a b}, \quad \forall a, b.
\end{eqnarray}
The set of roots of $\gl_{m|n}$ is $\{{\mathcal E}_a-{\mathcal E}_b\mid a\ne b\}$,
and the set of simple roots is
$
\{\alpha_a:={\mathcal E}_a-{\mathcal E}_{a+1} \mid 1\le a<m+n\}.
$

\begin{notation}\label{weight-space}
Denote by $\cE(m|n)$ the $(m+n)$-dimensional vector space,
which has a basis consisting of elements
${\mathcal E}_a$ ($1\le a\le m+n$) and is equipped with the bilinear form \eqref{eq:weight-space-norm}.
For later use,  we let $\epsilon_i$ ($i=1, 2, \dots, m$)
be the basis elements such that if $\epsilon_i={\mathcal E}_{a_i}$, then $[a_i]=0$. We order the
elements so that for any $\epsilon_i={\mathcal E}_{a_i}$ and $\epsilon_j={\mathcal E}_{a_j}$, if
$i<j$, then $a_i<a_j$. Similarly let  $\delta_j={\mathcal E}_{b_j}$ ($j=1, 2, \dots, n$) be the basis elements 
such that $[b_j]=1$ for all $j$ and $b_i<b_j$ if $i<j$.
\end{notation}

Now we will regard $E_{a b}$ as elements in the universal enveloping
superalgebra $\U({\mathfrak{gl}_{m|n}}; \C)$.
The quadratic Casimir operator of $\U({\mathfrak{gl}_{m|n}}; \C)$ is given by
\[
\omega_{\mathfrak{gl}_{m|n}}=\sum_{a, b=1}^{m+n} (-1)^{[b]} E_{a b} E_{b a},
\]
and $C_{\mathfrak{gl}_{m|n}}=\frac{1}{2}\left(\Delta(\omega_{\mathfrak{gl}_{m|n}})-
\omega_{\mathfrak{gl}_{m|n}}\otimes 1-1\otimes \omega_{\mathfrak{gl}_{m|n}}\right)$ is given by
\[
C_{\mathfrak{gl}_{m|n}} = \sum_{a, b=1}^{m+n}  (-1)^{[b]} E_{a b} \otimes E_{b a}.
\]

Clearly Lemma \ref{eq:Casimir} and  Lemma \ref{lem:four-term} are still valid for $C_{\mathfrak{gl}_{m|n}}$.
As usual, we set  $\U({\mathfrak{gl}_{m|n}}; T)=\U({\mathfrak{gl}_{m|n}}; \C)\otimes_\C T$.
We can now construct the ribbon quasi Hopf superalgebra (see Remark \ref{rem:non-simple})
\begin{eqnarray}\label{eq:Ugl}
\left(\U({\mathfrak{gl}_{m|n}}; T), \Delta, \epsilon, \Phi_{KZ}, S, \alpha, \beta, R, v\right),
\end{eqnarray}
where the universal $R$-matrix is given by $R=\exp( t C_{\mathfrak{gl}_{m|n}}/2)$
and the associator $ \Phi_{KZ}$ is constructed using
the KZ equation associated to  $C_{\mathfrak{gl}_{m|n}}$,
\begin{eqnarray*}
{\frac{d G(z)_{\mathfrak{gl}_{m|n}}}{d z}}&=& {\frac{t}{2\pi i}}\left( { \frac{(C_{\mathfrak{gl}_{m|n}})_{1 2}}{z}}
+ { \frac{(C_{\mathfrak{gl}_{m|n}})_{2 3}}{z-1}}\right) G(z)_{\mathfrak{gl}_{m|n}},
\end{eqnarray*}
as explained in Section \ref{sect:uea} (see \eqref{KZ} in particular).

\subsubsection{Relationship to $A(m-1|n-1)$}
It is very informative to see how the ribbon quasi Hopf superalgebra structure of $\U({\mathfrak{gl}_{m|n}}; T)$ is
related to that of $\U(\fg; T)$ where $\fg=A(m-1|n-1)$.
Recall that $\mathfrak{gl}_{m|n}$ contains the special linear superalgebra
$\mathfrak{sl}_{m|n}$. If $m\ne n$, then $\mathfrak{sl}_{m|n}$  is simple
and equal to $A(m-1|n-1)$. If $m=n$, let $I$ be
the identity matrix of size $n\times n$, then $I\in \mathfrak{sl}_{n|n}$
and $\mathfrak{sl}_{n|n}/\C I$.
Note that $I=\sum_{a=1}^{m+2n} E_{a a}$. We regard it as an element in
$ \U({\mathfrak{gl}_{m|n}}; T)$, and let
$\langle I \rangle$ be the $2$-sided ideal
generated by $I$. Denote by
$\pi: \U({\mathfrak{gl}_{m|n}}; T)\longrightarrow
\U({\mathfrak{gl}_{m|n}}; T)/\langle I \rangle$ the canonical surjection,
We have the following result.
\begin{lemma} \label{lem:equal-Phi-1} \label{lem:equal-Phi}
Denote $\fg=A(m-1|n-1)$.
Let $R$ and $\Phi_{KZ}$ be the universal $R$ matrix and Drinfeld associator
for $\U({\mathfrak{gl}_{m|n}}; T)$ respectively.
\begin{enumerate}
\item The universal $R$ matrix and the Drinfeld associator
for $\U(\fg; T)$ are given by
\[
R_\fg=\pi\otimes\pi(R), \quad \Phi_{{KZ}, \fg}=\pi\otimes\pi\otimes\pi(\Phi_{KZ}).
\]

\item If $m\ne n$, one can regard $\fg={\mathfrak{sl}_{m|n}}$ as a Lie super subalgebra of 
$\gl_{m|n}$, and hence $\U(\fg; T)$ as a super subalgebra of $\U({\mathfrak{gl}_{m|n}}; T)$. Then
\begin{eqnarray}\label{eq:Rsl-Rgl}
R_{\fg}=R\exp\left(-\frac{t I\otimes I}{2(m-n)}\right), \quad \Phi_{{KZ}, \fg}=\Phi_{KZ}.
\end{eqnarray}
\end{enumerate}
\end{lemma}

\begin{proof}
Part (1) of the lemma is obvious. To see part (2),
note that ${\mathfrak{sl}_{m|n}}$ is the subalgebra of
${\mathfrak{gl}_{m|n}}$ spanned by the elements $E_{a b}$ and $(-1)^{[a]}E_{a a}- (-1)^{[b]}E_{ b b}$ for all $a\ne b$.
Let $\omega$ be the quadratic Casimir operator of $\U(\fg; \C)$, and let
$C$ be the Casimir element in $\U(\fg; \C)\otimes\U(\fg; \C)$ defined by Lemma \ref{eq:Casimir}.
Denote by $\U(\C I; T)$ the universal enveloping algebra of $\C I$, which is the
polynomial algebra in $I$. If $m\ne n$, then $\fg={\mathfrak{sl}_{m|n}}$ and hence
$\U({\mathfrak{gl}_{m|n}}; T)\cong\U(\fg; T)\otimes\U(\C I; T)$. We have
\[
\omega=\omega_{\mathfrak{gl}_{m|n}} - \frac{I^2}{m-n}, \quad C= C_{\mathfrak{gl}_{m|n}} - \frac{I\otimes I}{m-n}.
\]
Since $I$ is central in  $\U({\mathfrak{gl}_{m|n}}; \C)$,
we immediately obtain \eqref{eq:Rsl-Rgl}.
Let
\[
G(z)=z^{-\frac{t}{2\pi i}\frac{I\otimes I\otimes 1}{m-n}} (z-1)^{-\frac{t}{2\pi i}\frac{1\otimes I\otimes I}{m-n}} G(z)_{\mathfrak{gl}_{m|n}}
\]
(regarded as a power series in $t$ with coefficients being functions of $z\in \C\backslash\{0, 1\}$ valued in $\U(\gl_{m|n}; \C)$).
We have
\begin{eqnarray*}
{\frac{d G(z)}{d z}}&=& {\frac{t}{2\pi i}}\left( { \frac{C_{1 2}}{z}}
+ { \frac{C_{2 3}}{z-1}}\right) G(z),
\end{eqnarray*}
which is the KZ equation associated to  $C$.  Inspection of the construction of the Drinfeld associator \eqref{KZ}
reveals that
\begin{eqnarray}
\Phi_{KZ}= (G_0(z)_{\mathfrak{gl}_{m|n}})^{-1} G_1(z)_{\mathfrak{gl}_{m|n}}= (G_0(z))^{-1} G_1(z)= \Phi_{{KZ}, \fg}.
\end{eqnarray}
\end{proof}

\subsection{The quantum general linear supergroup}\label{sect:q-GL}

We now consider the quantum general linear supergroup $\U_q(\mathfrak{gl}_{m|n}, \phi; T)$ following \cite{Z92, Z93}.
It is generated by
$E_{a a}$ with $a=1, 2, \dots, m+n$ and $e_i, \ f_i$ with  $i=1, 2, \dots, m+n-1$.
Let $K_a = q^{(-1)^{[a]}E_{a a}}$, where
$[a]=0$ if $({\mathcal E}_a, {\mathcal E}_a)=1$, and $[a]=1$ if $({\mathcal E}_a, {\mathcal E}_a)=-1$.
Then the defining relations of $\U_q(\mathfrak{gl}_{m|n}, \phi)$ are
\[
\begin{aligned}
&K_a K_b = K_b K_a, \quad K_a K_a^{-1}=1, \quad \forall a, b,\\
&K_a e_i K_a^{-1} =q^{ ({\mathcal E}_a, \alpha_i)} e_i,  \quad
K_a f_i K_a^{-1} =q^{- ({\mathcal E}_a, \alpha_i)} f_i, \quad \forall a, i, \\
&\text {plus the relations in \eqref{eq:q-group} with $h_i = E_{i i}- (-1)^{[i+1]}E_{i+1, i+1}$}.
\end{aligned}
\]
The Hopf superalgebra structure of $\U_q(\fg, \phi)$ with $\fg=A(m-1|n-1)$ extends
to $\U_q(\mathfrak{gl}_{m|n}, \phi)$ with
$\Delta_q(E_{a a})=E_{a a}\otimes 1+1\otimes E_{a a}$,  $S_q(E_{a a})=-E_{a a}$
and $\epsilon_q(E_{a a})=0$.
The resulting Hopf superalgebra admits a universal R-matrix $R_q$,
which can be expressed as
\[
R_q=K\Theta, \quad K=q^{\sum_a (-1)^{[a]}E_{a a}\otimes E_{a a}}, \quad  \Theta=1\otimes 1 + (q-q^{-1})\sum_s E_s\otimes F_s,
\]
where  all $E_s$ belong to  the subalgebra $\langle e_1, \dots, e_{m+n-1} \rangle$, and
$F_s$ to the subalgebra $\langle f_1, \dots, f_{m+n-1} \rangle$.  The quantum general linear supergroup
\begin{eqnarray}\label{eq:Uqgl}
\left(\U_q(\mathfrak{gl}_{m|n}, \phi; T), \Delta_q, \epsilon_q, S_q, R_q, v_q\right),
\end{eqnarray}
is a quasi triangular ribbon Hopf superalgebra.

We have the quantum special linear supergroup $\U_q(\mathfrak{sl}_{m|n}, \phi; T)$, which
is the subalgebra of  $\U_q(\mathfrak{gl}_{m|n}, \phi; T)$ generated by
$e_i, \ f_i, \ h_i$ ($1\le i\le m+n-1$). If $m\ne n$,  its universal
$R$-matrix is given by $K'\Theta$ with $K'=K q^{-\frac{I\otimes I}{m-n}}$.
If $m=n$, we have seen that $I\in \U_q(\mathfrak{sl}_{m|n}, \phi; T)$.  Let $\langle I \rangle$ be the $2$-sided ideal
in $\U_q(\mathfrak{gl}_{m|n}, \phi; T)$ generated by $I$.
Then $\U_q(\fg, \phi; T)$ is the image of $\U_q(\mathfrak{sl}_{m|n}, \phi; T)$ in the quotient
$\U_q(\mathfrak{gl}_{m|n}, \phi; T)/\langle I \rangle$.
The universal $R$-matrix of $\U_q(\fg, \phi)$ is the image of $R_q$ under this quotient map
(more precisely, the tensor product of the quotient map with itself).
Note that the quotient map does not affect $\Theta$.

\subsection{Isomorphism theorem} \label{sect:proof-main-gl}
We use $\U(T)$ and $\U_q(T)$ to denote $\U(\mathfrak{gl}_{m|n}; T)$ and
$\U_q(\mathfrak{gl}_{m|n}, \phi; T)$ respectively, and define the quasi Hopf superalgebras
\begin{eqnarray}\label{eq:U(R)}
\U(T_t) := \U(T)\otimes_T T_t, \quad \U_q(T_t) := \U_q(T)\otimes_T T_t
\end{eqnarray}
as in Section \ref{sect:Laurent-series}.
In the remainder of this subsection, we set $\K=T$ or $T_t$. Then we have the ribbon Hopf superalgebra
$(\U_q(\K), \Delta_q, \epsilon_q, S_q, R_q, v_q)$ and also the
ribbon quasi Hopf superalgebra $(\U(\K), \Delta, \epsilon, \Phi_{KZ}, S, \alpha, \beta, R, v)$.
Both $\U_q(\K)\text{-mod}$ and $\U(\K)\text{-mod}$ (cf. Notation \ref{notation}) are ribbon categories.
\begin{theorem}\label{thm:equiv-mod-gl}
Let $\K$ be $T$ or $T_t$, and
$\fg=\mathfrak{gl}_{m|n}$.
\begin{enumerate}
\item There is an equivalence of ribbon quasi Hopf superalgebras

$(\U_q(\K), \Delta_q, \epsilon_q, S_q, R_q, v_q)
\longrightarrow(\U(\K), \Delta, \epsilon, \Phi_{KZ}, S, \alpha, \beta, R, v)$.
\item There exists a braided tensor equivalence $\U(\K)\text{-mod}
\longrightarrow \U_q(\K)\text{-mod}$, which preserves
duality and twist.
\end{enumerate}
\end{theorem}

\begin{proof}
If the theorem holds over $T$, it remains true over $T_t$ by extending scalars.
Further,  part (1) implies part (2). Thus it remains only to prove part (1) over $T$,
which is  the $\gl_{m|n}$ case of Theorem \ref{thm:iso}.  This will be proved below.
\end{proof}

\begin{proof}[Proof of Theorem \ref{thm:iso} for $\gl_{m|n}$]
We now complete the proof of Theorem \ref{thm:iso}. The arguments used are quite similar to those in \cite{G07}, so we will be brief.
We start by realising $\fg=\gl_{m|n}$ in terms of a Drinfeld double. The Borel subalgebra $\fb$ of $\fg$ is a Lie super bialgebra with 
co-multiplication $\delta: \fb\longrightarrow \fb\otimes\fb$ given by
\[
\delta(h)=0, \  \ h\in\fh, \quad  \delta(E_{a, a+1})=E_{a, a+1}\otimes h_a- h_a\otimes E_{a, a+1}, \ \forall a,
\]
where $h_a = (-1)^{[a]}E_{a a } -  (-1)^{[a+1]} E_{a+1, a +1}$. Drinfeld's classical double construction \cite{GZB} 
applied to $\fb$ yields a quasi triangular Lie super bialgebra $D(\fb)=\fb\oplus\fb^*$. Then the quotient of $D(\fb)$ 
obtained by identifying $\fh^*$ with $\fh$ is isomorphic to $\fg$ as a quasi triangular Lie super bialgebra.

We take the Hopf superalgebra $\tilde\U_q(\fb; T)$ generated by all the elements $E_{a, a+1}$ and $\fh$ subject to relations
\[
[h, h']=0, \  \forall h, h'\in\fh, \quad [h, E_{a, a+1}]=\alpha_a(h)E_{a, a+1}, \ \forall a,
\]
where the bracket denotes commutator. The co-multiplication is taken to be
\begin{eqnarray}\label{eq:Delta-sym}
\Delta(h)=h\otimes 1 + 1\otimes h, \ \
\Delta(E_{a, a+1})=E_{a, a+1}\otimes q^{h_a/2}+ q^{-h_a/2}\otimes  E_{a, a+1}.
\end{eqnarray}
There exists a $T$-bilinear form
$
\langle \ , \ \rangle:   \tilde\U_q(\fb; T)\times \tilde\U_q(\fb; T) \longrightarrow T
$
with the properties
\[
\begin{aligned}
&\langle x y, z \rangle = \langle x\otimes y, \Delta(z) \rangle, \quad
\langle z,  x y \rangle = \langle \Delta(z), x\otimes y \rangle, \ \forall x, y, z, \\
&\langle q^h,  q^{h'}\rangle = q^{-(h, h')}, \quad h, h'\in \fh, \\
&\langle E_{a, a+1}, E_{b, b+1}\rangle =\delta_{a b}(-1)^{[a]}/(q-q^{-1}), \quad \forall a, b,
\end{aligned}
\]
where $(h, h')$ is the supertrace form on $\fh$. By the same computations as those in
\cite{G06} (also see \cite{Y91, Y94}), we can show that the radical ${\rm Rad}$ of the bilinear form is generated 
by the Serre relations and higher order Serre relations obeyed by the elements $E_{a, a+1}$. 
Thus $\U_q(\fb; T):=\tilde\U_q(\fb; T)/{\rm Rad}$ coincides with the quantum
Borel subalgebra of $\U_q(\fg, \phi; T)$. This is the Etingof-Kazhdan quantisation
of the Lie super bialgebra $\fb$.

Since the Etingof-Kazhdan quantisation commutes with double constructions,
the quantum double $D(\U_q(\fb; T))$ of $\U_q(\fb; T)$ is the Etingof-Kazhdan quantisation $\U_t^{EK}(D(\fb); T)$ of $D(\fb)$.
It is a general fact \cite{EK-I} that
$\U_t^{EK}(D(\fb); T)$  is equivalent to $\U(D(\fb); T)$ as braided quasi Hopf superalgebras.
The quotient of $\U(D(\fb); T)$ obtained by
identifying $\fh^*$ with $\fh$ is equal to $\U(\fg; T)$,
and the corresponding quotient of $\U_t^{EK}(D(\fb); T)$ is $\U_q(\fg, \phi; T)$.
Theorem \ref{thm:iso} for $\fg=\gl_{m|n}$ therefore follows, subject to
the following observation.

The comultiplication of $\U_q(\gl_{m|n}, \phi; T)$ given by \eqref{eq:Delta-sym}
is not that given in Section \ref{sect:DJ-alg}.
However that there exists an algebra automorphism
$\eta:  \U_q(\gl_{m|n}, \phi; T)\longrightarrow  \U_q(\gl_{m|n}, \phi; T)$  defined by
\[
h\mapsto h, \ h\in \fh, \quad E_{a, a+1} \mapsto q^{h_a/2}E_{a, a+1},
\quad   E_{a+1, a} \mapsto E_{a+1, a}q^{-h_a/2}, \quad  \forall a.
\]
Then $\eta(E_{a, a+1})$ and $\eta(E_{a+1, a})$ have the co-products given in
Section \ref{sect:DJ-alg}.

This completes the proof.
\end{proof}

\begin{remark}
We can prove the theorem directly when $m\ne n$.  In this case,
\[
\begin{aligned}
&\U_q(\mathfrak{gl}_{m|n}, \phi; T)=\U_q(\mathfrak{sl}_{m|n}, \phi; T)\otimes\U_q(\C I; T), \\ 
&\U(\mathfrak{gl}_{m|n}; T)=\U(\mathfrak{sl}_{m|n}; T)\otimes \U(\C I; T),
\end{aligned}
\]
where $\U_q(\C I; T)=\U(\C I; T)$ is the algebra generated by $I$. The superalgebra isomorphism
$f: \U_q(\mathfrak{sl}_{m|n}, \phi; T) \longrightarrow \U(\mathfrak{sl}_{m|n}; T)$  given by
Theorem \ref{thm:iso} can be easily extended to a superalgebra isomorphism 
$\tilde{f}: \U_q(\mathfrak{gl}_{m|n}, \phi; T)  \longrightarrow \U(\mathfrak{gl}_{m|n}; T)$, 
whose restriction to  $\U_q(\mathfrak{sl}_{m|n}, \phi; T)$ is $f$, and
to $\U_q(\C I; T)$ is the identity. Now
$\tilde{f}$ together with the gauge transformation $F$
and antipode transformation $g$ given in Theorem \ref{thm:iso} for $\fg=\mathfrak{sl}_{m|n}$
constitutes an equivalence of quasi Hopf superalgebras.
It is clear from part (2) of Lemma \ref{lem:equal-Phi} that this is an
equivalence of ribbon quasi Hopf superalgebras.
\end{remark}

\subsection{Invariant theory over the Laurent series ring}\label{sect:gl-inv-Tt}
Retain the notation of the last section.
Let $V=\C^{m|n}$ be the natural module for the general linear Lie superalgebra $\fg=\mathfrak{gl}_{m|n}$ over $\C$.
Write $V_\K=V\otimes_\C \K$ and $V_\K^*=\Hom_\K(V_\K, \K)$.
Then $V_\K$ and $V_\K^*$ naturally have $\U(\K)$-module structures.
Given any $\varepsilon=(\varepsilon_1, \varepsilon_2, \dots, \varepsilon_k)\in\{+1, -1\}^k$,
we let $\boxtimes V_\K^\varepsilon$ be the ordered tensor
product (cf. \eqref{eq:ordered-tensor}) of
the sequence $(V_\K^{\varepsilon_1}, V_\K^{\varepsilon_2}, \dots, V_\K^{\varepsilon_k})$
of $\U(\K)$-modules.
Denote by $\cT(\K)$ the full subcategory of $\U(\K)\text{-mod}$
with objects of the form $\boxtimes V_\K^\varepsilon$ for all sequences
$\varepsilon$ of any lengths.

Let $\cI: \U(\K)\text{-mod}\longrightarrow \U_q(\K)\text{-mod}$ be
the equivalence of categories guaranteed by Theorem \ref{thm:equiv-mod-gl} (2); 
we abuse notation by using
$V_\K$ and $V_\K^*$ to denote $\cI(V_\K)$ and $\cI(V_\K^*)$ respectively.
Then we have the full subcategory $\cT_q(\K)$ of
$\U_q(\K)\text{-mod}$ with objects $\otimes V_\K^\varepsilon$ for sequences $\varepsilon$.

\begin{lemma}\label{lem:functors-gl-1}
Both $\cT(\K)$  and $\cT_q(\K)$ are ribbon categories,
and there is a braided tensor equivalence
$
\cT(\K)\stackrel{\sim}{\longrightarrow} \cT_q(\K)
$
which preserves duality and twist.
\end{lemma}
\begin{proof} The first claim is clear. Now the functor $\cI$
induces a bijection between the sets of objects of the two ribbon categories.
Thus the second part follows from Theorem \ref{thm:equiv-mod-gl}(2).
\end{proof}

Let $\cH(\K)$ be the category of ribbon graphs (cf.  Definition \ref{def:ribbon-direct}).
Given any ribbon category and any object in it, by Theorem \ref{thm:RT}
there exists a unique braided tensor functor which preserves duality and twist
from $\cH(\K)$ to this category. Let
$
\cF_\K: \cH(\K)\longrightarrow \cT(\K)$ and $
\cF_{q, \K}: \cH(\K)\longrightarrow \cT_q(\K)
$
 be the respective braided tensor functors for the
 ribbon categories $\cT(\K)$ and $\cT_q(\K)$. Here we have
 left out of the notation of the braided tensor functors the natural isomorphisms $\varphi_0$ and $\varphi_2$,
 as they are identity maps in both cases.

\begin{lemma}\label{lem:functors-gl-2}
Both $\cF_\K(\cH(\K))$ and $\cF_{q, \K}(\cH(\K))$ are
ribbon categories, and there is a braided tensor equivalence
$
\cF_\K(\cH(\K)) \stackrel{\sim}{\longrightarrow}  \cF_{q, \K}(\cH(\K))
$
which preserves duality and twist.
\end{lemma}
\begin{proof}
The first claim is clear, and the second follows from Lemma \ref{lem:functors-gl-1}
and the uniqueness of the functors $\cF_\K$ and $ \cF_{q, \K}$.
\end{proof}

\begin{theorem} \label{thm:main1}
Let $\K$ be either $T$ or $T_t$.  The  braided tensor functor
$
\cF_\K: \cH(\K)\longrightarrow \cT(\K)$ is full,
as is 
$
\cF_{q, \K}: \cH(\K)\longrightarrow \cT_q(\K)
$.
\end{theorem}
\begin{proof}
From Lemma \ref{lem:functors-gl-2} one sees that the first statement implies the second.
Write $V_\K^{\boxtimes_\K r}
=\underbrace{V_\K\boxtimes_\K\dots\boxtimes_\K V_\K}_r$.
If we show that
$\cF_\K(\cH_r^r(\K))=\End_{\cT(\K)} (V_\K^{\boxtimes_\K r})$
for all $r$, then using the left (and right) dualities of the respective ribbon categories,
it will follow that $\cF_\K$ is full.

Obviously $\End_{\cT(\K)}(V_\K^{\boxtimes_\K r})
=\End_{\U(\K)}(V_\K^{\otimes_\K r})$.
Recall that $V_\K^{\otimes_\K r}= V^{\otimes_\C r}\otimes_\C \K$ and  $\U(\K)= \U(\C)\otimes \K$
where $\U(\C)=\U(\mathfrak{gl}_{m|n}; \C)$.
Thus by exactness of $-\ot_\C\K$, we have $\End_{\U(\K)}(V_\K^{\otimes_\K r})\cong
\End_{\cT(\C)}(V^{\otimes_\C r})\otimes_\C \K$ as associative algebras. Hence
\begin{eqnarray}\label{eq:tensoring}
\End_{\cT(\K)}(V_\K^{\boxtimes_\K r}) \cong
\End_{\U(\C)}(V^{\otimes_\C r})\otimes_\C \K.
\end{eqnarray}
The endomorphism algebra $\End_{\U(\C)}(V^{\otimes_\C r})$
is given by the first fundamental theorem of invariant theory for the classical supergroups
\cite{BR, S0, S1,DLZ, LZ14a},
which can be described as follows.
Denote by $\Sym_r$ the symmetric group
of degree $r$, and let $\nu_r$ be the representation of $\C\Sym_r$
on $V^{\otimes r}$ such that
$s_i=(i , i+1)$ acts by
\begin{eqnarray}\label{eq:Sym}
v_1\otimes v_2\otimes\dots\otimes v_r &\mapsto&
v_1\otimes\dots\otimes \tau(v_i \otimes v_{i+1})\otimes \dots \otimes v_r,
\end{eqnarray}
where $\tau=\tau_{V, V}$ is defined by \eqref{eq:tau}.
Then $\End_{GL_{m|n}}(V^{\otimes_\C r})=\nu_r(\C\Sym_r)$.
We $\K$-linearly extend  $\nu_r$ to a representation of $\K\Sym_r$ on
$V_\K^{\otimes_\K r}$. Then $\nu_r(\K\Sym_r)$ as an associative algebra is generated by
the endomorphisms
\[
\tau_i = \underbrace{\id_{V_\K}\otimes\dots\otimes\id_{V_\K}}_{i-1}\otimes \tau \otimes
\underbrace{\id_{V_\K}\otimes\dots\otimes\id_{V_\K}}_{r-i-1},
\quad i=1, 2, \dots, r-1.
\]
Now we need to show that all $\tau_i$ belong to $\cF_\K(\cH_r^r(\K))$.

Let $Q$ and $b$ respectively denote the action of $C$ and $R=\exp(t C/2)$ on $V_\K\boxtimes V_\K$,
and let $g=\tau b$. Then by considering eigenspaces of $C$ in $V_\K\boxtimes V_\K$, we obtain
\begin{eqnarray}\label{eq:eigen-b}
\begin{aligned}
b= \exp\left(\frac{t}{2}\right) \frac{1+\tau}{2} +  \exp\left(-\frac{t}{2}\right) \frac{1-\tau}{2},  \\
g= \exp\left(\frac{t}{2}\right) \frac{1+\tau}{2} -  \exp\left(-\frac{t}{2}\right) \frac{1-\tau}{2},
\end{aligned}
\end{eqnarray}
which leads to
\begin{eqnarray}\label{eq:eigen-tau}
&\tau = \left(g + g^{-1}\right)\left(\exp\left(\frac{t}{2}\right)+\exp\left(-\frac{t}{2}\right)\right)^{-1}.
\end{eqnarray}
Since by Theorem \ref{thm:RT}, the maps
\[
g_i = \id_{V_\K^{\boxtimes(i-1)}}\boxtimes g \boxtimes \id_{V_\K^{\boxtimes(r-i-1)}}, \quad  i=1, 2, \dots, r-1,
\]
all belong to $\cF_A(\cH_r^r(\K))$, so do also
\begin{eqnarray}\label{eq:gbinv}
\begin{aligned}
\hat\tau_i &= \underbrace{\id_{V_\K}\boxtimes\dots\boxtimes\id_{V_\K}}_{i-1}\boxtimes \tau\boxtimes 
\underbrace{\id_{V_\K}\boxtimes\dots\boxtimes\id_{V_\K}}_{r-i-1},  \\
b_i &= \underbrace{\id_{V_\K}\boxtimes\dots\boxtimes\id_{V_\K}}_{i-1}
\boxtimes b\boxtimes \underbrace{\id_{V_\K}\boxtimes\dots\boxtimes\id_{V_\K}}_{r-i-1}.
\end{aligned}
\end{eqnarray}
Clearly, $\hat\tau_i = (g_i + g_i^{-1})(\exp(\frac{t}{2})+\exp(-\frac{t}{2}))^{-1}$.

The difference between $\tau_i$ and ${\hat\tau}_i$ lies in the change from the
tensor product of maps with respect to $\otimes$ to that with respect to
$\boxtimes$, where the transformation rule is given by \eqref{eq:star2}.
In the present context, the left and right unit constraints are trivial, thus by
\eqref{eq:map-ast-tensor}, only the associator $\Phi_{KZ}$ is involved in \eqref{eq:star2},
as can be seen from  the last formula in \cite[II.C]{Z02}.  This allows us to express
each $\tau_i$ in terms of ${\hat\tau}_i$ and (co-products of) the associator.
For any modules $M_1, M_2, M_3$, the action of
the associator on $M_1\boxtimes M_2\boxtimes M_3$ needed in the formulae is exactly the same
as that on $M_1\otimes M_2\otimes M_3$.  Let
\[
Q_j =  \underbrace{\id_{V_\K}\otimes\dots\otimes\id_{V_\K}}_{j-1}\otimes Q \otimes
\underbrace{\id_{V_\K}\otimes\dots\otimes\id_{V_\K}}_{r-j-1}.
\]
Then the $\varphi$ maps in \eqref{eq:star2} (i.e., the $J$ maps in the last formula in \cite[II.C]{Z02})
can be expressed as power series in $t$ with coefficients depending on the maps $\tau_j$ and $t Q_j$
only. Then
\begin{eqnarray}\label{eq:tau-interate}
{\hat\tau}_i= \tau_i + t \Gamma(\tau_1, \dots, \tau_r, Q_1, \dots, Q_r),
\end{eqnarray}
where $\Gamma(\tau_1, \dots, \tau_r, Q_1, \dots, Q_r)$ is a power series in $t$.
The coefficient of each $t^k$ is a linear combination of products of $\tau_j$ and $Q_j$ ($j=1, 2, \dots, r-1$)
in various orders such that the number of $Q_j$ factors is $k+1$ in every term.

Let $\hat{Q}_i = \underbrace{\id_{V_\K}\boxtimes\dots\boxtimes\id_{V_\K}}_{i-1}\boxtimes Q\boxtimes 
\underbrace{\id_{V_\K}\boxtimes\dots\boxtimes\id_{V_\K}}_{r-i-1}$, 
which can be expressed in terms of $b_i$ as $t\hat{Q}_i=-\sum_{k=1}^\infty ( \id_{V_\K}^{\boxtimes r}-b_i)^k/k$.  
Arguments about $\tau_i$ also
apply to $\hat{Q}_i $ to give
\begin{eqnarray}\label{eq:Q-interate}
\hat{Q}_i =  Q_i + t \Theta(\tau_1, \dots, \tau_r, Q_1, \dots, Q_r),
\end{eqnarray}
where $\Theta(\tau_1, \dots, \tau_r, Q_1, \dots, Q_r)$ has the same properties
as $\Gamma(\tau_1, \dots, \tau_r, Q_1, \dots, Q_r)$.

By iterating \eqref{eq:tau-interate} and \eqref{eq:Q-interate},  we obtain $\tau_i$
and $tQ_i$
in terms of ${\hat\tau}_j$ and $t\hat{Q}_j$. This shows that
\begin{equation}\label{quote}
\text{$\tau_i$ and $tQ_i$ can be expressed in terms of $g_j$ and $b_j$.}
\end{equation}
Therefore, $\tau_i\in \cF_A(\cH_r^r(\K))$
for all $i$, and hence $\cF_\K: \cH(\K)\longrightarrow \cT(\K)$
is indeed a full functor.
\end{proof}

\subsection{Invariant theory over $\C(q)$}
Let $\F=\C(q)$ be the field of rational functions in
the indeterminate $q$.
Denote by $\U_q(\F)$ the Jimbo version of the
quantum general linear supergroup \cite{Z93} over $\F$ associated with $\fg=\mathfrak{gl}_{m|n}$.
It is generated by
$K_a^{\pm 1}$ ($a=1, 2, \dots m+n$) and $e_i, f_i$ ($i=1, 2, \dots, m+n-1$)
subject to the same relations as those of the corresponding quantum supergroup over the power series ring $T$.

The Hopf superalgebra $\U_q(\F)$ almost has the structure of a quasi triangular Hopf superalgebra.
The subtlety here is that there exists no universal $R$-matrix
belonging to $\U_q(\F)\otimes\U_q(\F)$
(or any completion of it), nor ribbon element. However,
$\U_q(\F)$  admits a functorial $R$-matrix $R_{W, W'}$
for any pair of objects $W$ and $W'$ in the category $\U_q(\F)\text{-Mod}_{f, \1}$
of
finite dimensional $\Z_2$-graded $\U_q(\F)$-modules of type $\1=(1, 1, \dots, 1)$.
The family of functorial isomorphisms
\[
c_{W, W'}=\tau\circ R_{W, W'}:  W\otimes W'\longrightarrow W'\otimes W
\]
give rise to a braiding for $\U_q(\F)\text{-Mod}_{f,\1}$.
There also exists a family of functorial isomorphisms $v_W: W\longrightarrow W$
satisfying the requirements of a twist. Thus $\U_q(\F)\text{-Mod}_{f,\1}$ is a ribbon category.

Let $V_q$ be the natural $\U_q(\F)$-module, and denote by $V_q^*$ the dual module of $V_q$.
\begin{definition}\label{def:Tq-gl}
The category $\cT_q(\F)$ of tensor modules for $\U_q(\F)$ is the full subcategory
of $\U_q(\F)\text{-Mod}_{f,\1}$ with objects
$V_q^{\varepsilon_1}\otimes V_q^{\varepsilon_2}\otimes \dots\otimes  V_q^{\varepsilon_k}$
for all $k\in\Z_+$, where $\varepsilon_i=\pm 1$ with $V_q^{+1} = V_q$ and $V_q^{-1}=V_q^*$.
This is a ribbon category.
\end{definition}
Then Theorem \ref{thm:RT} gives rise to a  unique braided tensor functor $\cF_{q, \F}: \cH(\F)\longrightarrow \cT_q(\F)$
in the present context.
We have the following result.
\begin{theorem}[FFT for the quantum general linear supergroup]
\label{thm:main2-gl}
The braided tensor functor $\cF_{q, \F}: \cH(\F)\longrightarrow \cT_q(\F)$
preserves duality and twist. Furthermore, it is full.
\end{theorem}
\begin{proof}
We only need to prove the fullness of $\cF_{q, \F}$.
Let  $\eta(r)=(\underbrace{+, +, \dots, +}_r)$ for any $r$
and denote $\Hom(\eta(r), \eta(r))$ by $\cH_r^r(\K)$.  In order to prove that $\cF_{q, \F}$ is full,
it suffices to show that $\cF_{q, \F}(\cH_r^r(\F)) =\End_{\U_q(\F)}(V_q^{\otimes r})$
because of the dualities of the ribbon categories $\cH(\F)$ and $\cT_q(\F)$.

Let $\K=T_t$, and let $\F\rightarrow \K$ be the field extension
defined by $q\mapsto \exp(t/2)$.
We define a $\U_q(\F)$-action on $\K$ by the composition of the co-unit and
this field extension. Then for any object $W_q$ in $\U_q(\F)\text{-Mod}_{f, \1}$,
we have the corresponding $\U_q(\F)$-module $W_q\otimes_\F\K$.
Now the specialisation $\U_q(\F)\otimes_\F\K$ of $\U_q(\F)$ acts
on $W_q\otimes_\F\K$ in the natural way:  $(x\otimes k) (w\otimes k')
=x w\otimes k k'$ for all $x\in\U_q(\F)$, $w\in W_q$ and $k, k'\in\K$.
Note that $\U_q(\F)\cong \U_q(\F)\otimes 1$, thus for any any object $W_q$ in $\U_q(\F)\text{-Mod}_{f, \1}$, we have
$W_q^{\U_q(\F)}\otimes_\F\K = (W_q\otimes_\F\K)^{\U_q(\F)\otimes 1}=(W_q\otimes_\F\K)^{\U_q(\K)}$.
It then follows that $\End_{\U_q(\F)}(W_q)\otimes_\F\K=
\End_{\U_q(\F)\otimes_\F\K}(W_q\otimes_\F\K)$. This implies that
\[
\End_{\U_q(\F)}(V_q^{\otimes_\F r})\otimes_\F\K=
\End_{\U_q(\F)\otimes_\F\K}(V_q^{\otimes_\F r}\otimes_\F\K)
=\End_{\U_q(\F)\otimes_\F\K}((V_q\otimes_\F\K)^{\otimes_\K r}).
\]

Note that $\U_q(\F)\otimes_\F\K$ is embedded in $\U_q(\K)$ with $e_i\otimes 1$ and
$f_i\otimes 1$ mapped to the corresponding generators of $\U_q(\K)$, and the elements
$K_a\otimes 1$ to power series in $\U_q(\K)$. Regarding $V_\K$
as $\U_q(\F)\otimes_\F\K$-module via the superalgebra embedding, we have
$V_q \otimes_\F\K\cong V_\K$,
and hence $V_q^{\otimes r} \otimes_\F\K\cong V_\K^{\otimes_\K r}$.
Observe that $\U_q(\F)\otimes_\F\K$ is a dense subalgebra
of $\U_q(\K)$ in the $t$-adic topology.
Therefore
$\End_{\U_q(\F)\otimes_\F\K}((V_q\otimes_\F\K)^{\otimes_\K r})\cong \End_{\U_q(\K)}(V_\K^{\otimes_\K r})$
as vector spaces, and hence
\begin{eqnarray}\label{eq:End-End}
\End_{\U_q(\F)}(V_q^{\otimes_\F r})\otimes_\F\K\cong \End_{\U_q(\K)}(V_\K^{\otimes_\K r}).
\end{eqnarray}
Further, by inspecting the actions of the functors $\cF_{q, \F}$ and $\cF_{q, \K}$ on the generators of
$\cH(\F)$ and $\cH(\K)$ respectively, we find a vector space isomorphism
\begin{eqnarray}\label{eq:cH-cH}
\cF_{q, \F}(\cH_r^r(\F)) \otimes_\F\K \cong \cF_{q, \K}(\cH_r^r(\K)).
\end{eqnarray}

From \eqref{eq:End-End} and \eqref{eq:cH-cH},  we obtain
\[
\begin{aligned}
&\dim_\F \End_{\U_q(\F)}(V_q^{\otimes r}) =\dim_\K(\End_{\U_q(\F)}(V_q^{\otimes r}) 
\otimes_\F\K) =\dim_\K \End_{\U_q(\K)}(V_\K^{\otimes r}); \\
&\dim_\F \cF_{q, \F}(\cH_r^r(\F)) = \dim_\K (\cF_{q, \F}(\cH_r^r(\F)) \otimes_\F\K) = \dim_\K \cF_{q, \K}(\cH_r^r(\K)).
\end{aligned}
\]
Now $\cF_{q, \K}(\cH_r^r(\K))=\End_{\U_q(\K)}(V_\K^{\otimes r})$ by Theorem \ref{thm:main1},  hence
\[
\dim_\F \cF_{q, \F}(\cH_r^r(\F)) =\dim_\F \End_{\U_q(\F)}(V_q^{\otimes r}).
\]
This shows that $\cF_{q, \F}(\cH_r^r(\F)) =\End_{\U_q(\F)}(V_q^{\otimes r})$, completing the proof.
\end{proof}

\subsection{Representations of Hecke algebras and walled BMW algebras}

Let $B_r$ denote the braid group of degree $r$ (i.e. on $r$ strings) generated by
 $X^+_i$ ($i=1, ,2 \dots, r-1$) given by
the second ribbon graph in Figure \ref{fig:Br-generators}.
Then $\cH_r^r(\K)=\K B_r$, the group algebra of $B_r$.
Clearly $X^-_i=(X^+_i)^{-1}$, and $I^+$ is the identity.

\begin{figure}[h]
\begin{center}
\setlength{\unitlength}{.7mm}
\begin{picture}(140, 100)(5,0)
\put(10, 80){$I^+=$}
\put(30,90){\vector(0,-1){20}}
\put(40,90){\vector(0,-1){20}}
\put(55, 80){... ... ... ...}
\put(68,70){\tiny $r$}
\put(100,90){\vector(0,-1){20}}
\put(110,90){\vector(0,-1){20}}

\put(10, 50){$X_i^+=$}
\put(30,60){\vector(0,-1){20}}
\put(38,50){...}
\put(36,40){\tiny $i-1$}
\put(50,60){\vector(0,-1){20}}

\put(69,51){\line(-1,1){9}}
\put(71,49){\vector(1,-1){10}}
\put(80,60){\vector(-1,-1){21}}

\put(90,60){\vector(0,-1){20}}
\put(98,50){...}
\put(92,40){\tiny $r-i-1$}
\put(110,60){\vector(0,-1){20}}

\put(10, 20){$X_i^-=$}
\put(30,30){\vector(0,-1){20}}
\put(38,20){...}
\put(36,10){\tiny $i-1$}
\put(50,30){\vector(0,-1){20}}

\put(60,30){\vector(1,-1){21}}
\put(69,19){\vector(-1,-1){9}}
\put(80,30){\line(-1,-1){9}}

\put(90,30){\vector(0,-1){20}}
\put(98,20){...}
\put(92,10){\tiny $r-i-1$}
\put(110,30){\vector(0,-1){20}}
\end{picture}
\end{center}
\caption{Braids}
\label{fig:Br-generators}
\end{figure}

Let $\mathcal{I}_r$ be the two-sided ideal of $\cH_r^r(\F)$ generated
( of ribbon graphs) by
\begin{eqnarray}\label{eq:Hecke-relat}
X_i^+ - X_i^- - (q-q^{-1}) I^+, \quad \text{for \ } i=1, 2, \dots, r-1.
\end{eqnarray}
Then $H_r(q; \F):=\cH_r^r(\F)/\mathcal{I}_r$ is the Hecke algebra of degree $r$.

From Theorem \ref{thm:main2-gl},  we obtain a representation
$\nu_r: \cH_r^r(\F)\longrightarrow \End_\F(V_q^{\otimes r})$
of the braid group by restricting $\cF_{q, \F}$ to $\cH_r^r(\F)$.
Then $\nu_r(X_i^+) =\cF_{q, \F}(X_i^+)$, and we have
\begin{eqnarray}\label{eq:braid-gp}
\nu_r(X_i^+) =\underbrace{\id_{V_q}\otimes \dots \otimes\id_{V_q}}_{r-i-1}\otimes \cF_{q, \F}(X^+) 
\otimes\underbrace{\id_{V_q}\otimes \dots \otimes\id_{V_q}}_{i-1},
\end{eqnarray}
with
$
\cF_{q, \F}(X^+) =c_{V_q, V_q}=\tau\circ R_{V_q, V_q}.
$

\begin{proposition} \label{prop:Hecke}
Let $\U_q(\F)$ be the quantum general linear supergroup associated with $\fg=\mathfrak{gl}_{m|n}$
corresponding to any given choice of Borel subalgebra for $\fg$.
The representation $\nu_r$  constructed above for the braid group $B_r$ factors through
the Hecke algebra $H_r(q; \F)$.
\end{proposition}
\begin{proof} The natural module $V_q$ has the property that
$
V_q\otimes V_q = L_q(s) \oplus L_q(a),
$
where $L_q(s)$ and $L_q(a)$ are simple $\U_q(\F)$-modules which are $q$-analogues of
the $\Z_2$-graded symmetric tensor $L(s):=S^2 V_\C$ and the $\Z_2$-graded skew symmetric tensor
$L(a):=\wedge^2 V_\C$.
Let $P[s]$ and $P[a]$ be idempotents mapping $V_q\otimes V_q$ surjectively onto $L_q(s)$ and $L_q(a)$
respectively. Then $P[s]+P[a]=\id$ and $P[s] P[a]=P[a] P[s]=0$.

To consider the braiding $c_{V_q, V_q}$, we let
$\chi_s:=\frac{1}{2}\omega_{L(s)} - \omega_{V_\C}$
and $\chi_a:=\frac{1}{2}\omega_{L(a)} - \omega_{V_\C},$
where $\omega_L$ is the eigenvalue of the quadratic Casimir of $\fg=\mathfrak{gl}_{m|n}$ in the simple
$\fg$-module $L$. Then it follows from a general property of the universal $R$-matrix that
the braiding operator can be expressed as  $c_{V_q, V_q}=q^{\chi_s} P[s] - q^{\chi_a} P[a]$.

The eigenvalue of the Casimir operator in a finite dimensional simple module $L$ can be computed by
considering $L$ as a highest weight module relative to any chosen Borel subalgebra $\fb\subset\fg$,
and is given by $(\lambda+2\rho, \lambda)$, where $\lambda$ is the highest weight of $L$ and $2\rho$ is the
 graded sum of the positive roots.  It is an easy calculation to show that $\chi_s=1$ and $\chi_a=-1$. Hence
$
c_{V_q, V_q}=q P[s] - q^{-1} P[a].
$
Therefore, $(c_{V_q, V_q}-q)(c_{V_q, V_q}-q^{-1})=0$ and it follows that
\[
(\nu_r(X^+_i)-q)(\nu_r(X^+_i)+q^{-1})=0, \quad \forall i.
\]
Hence the  representation of $B_r$ defined by \eqref{eq:braid-gp} factors through the Hecke algebra.
\end{proof}

\begin{remark}
The above result is in exact analogy with the case of classical $\gl_m$.
\end{remark}

\begin{example}
Consider the quantum general linear supergroup  \cite{Z92, Z93}
defined  with respect to the distinguished root datum of $\fg=\mathfrak{gl}_{m|n}$,
which corresponds to the following admissible ordering
$(\mathcal{E}_1, \mathcal{E}_2, \dots, \mathcal{E}_{m+n})=
(\varepsilon_1, \dots, \varepsilon_m, \delta_1, \dots, \delta_n)$
of the basis elements of $\cE(m|n)$.
Let $\{e_a\mid a=1, 2, \dots, m+n\}$ be the standard basis of $V_q$ which is homogeneous with parity $[e_a]=[a]$,
where $[i]=0$ for $i\le m$ and $[m+j]=1$ for all $j>0$.
Then the action of the universal $R$-matrix on $V_q\otimes V_q$ is given by (see the formula below equation (9) in \cite{Z98}):
\begin{eqnarray}\label{eq:R-V-V}
\begin{split}
R_{V_q,V_q}=&q^{\sum_{a=1}^{m+n}(-1)^{[a]}e_{aa}\otimes e_{aa}}
	+(q-q^{-1})\sum\limits_{a<b}(-1)^{[b]}e_{ab}\otimes e_{ba},
\end{split}
\end{eqnarray}
where $e_{a b}$ are the matrix units such that $e_{a b} e_c=\delta_{b c} e_a$,
and the tensor product $e_{a b}\otimes e_{c d}$ of matrices acts on $V_q\otimes V_q$, in the usual way which respects the $\Z_2$-gradings,  by
\[
e_{a b}\otimes e_{c d}(e_f\otimes e_g) =(-1)^{[f]([c]+[d])}\delta_{b f} \delta_{d g} e_a\otimes e_c.
\]
The first term on the right side of \eqref{eq:R-V-V} is interpreted as
\[
q^{\sum\limits_{a=1}^{m+n}(-1)^{[a]}e_{aa}\otimes e_{aa}}=1\otimes 1 +
\sum_{a=1}^{m+n}\left(q^{(\mathcal{E}_a, \mathcal{E}_a)}-1\right)e_{aa}\otimes e_{aa}.
\]
This is clear when the left side is regarded a power series of matrices over $T$.
\end{example}

\begin{remark}
The representation \eqref{eq:braid-gp} of the Hecke algebra has been known since the early 80s
from the Perk-Schultz models in statistical mechanics. The spectral parameter dependent
$R$-matrix $R(x)$ of those models can be
found, e.g., in \cite[p.1973]{Z92},
and the $R$-matrix $R_{V_q, V_q}$ given by \eqref{eq:R-V-V} is its  limit of infinite spectral parameter,
i.e., $\lim\limits_{x\to\infty}\frac{R(x)}{x}$.
The $R$-matrix \eqref{eq:R-V-V}
was also the main input in the definition of the quantum coordinate superalgebra of
the quantum general linear supergroup as treated in \cite{Z98}.
\end{remark}

Fix non-negative integers $r$ and $s$, let $\eta'(s)$ be the sign tuple $(\underbrace{-, \dots, -}_s)$
and let $\eta(r, s)= \eta(r)\otimes \eta'(s)$. Then
$\eta(r, s)=(\underbrace{+, \dots, +}_r, \underbrace{-, \dots, -}_s)$.
Consider
\[
\cH_{r, s}^{r, s}(\F):=\Hom_{\cT_q(\F)}(\eta(r, s), \eta(r, s));
\]
this is an associative algebra under composition of ribbon diagrams.
We let $\mathcal{IH}$ be the $\F$-span of the ribbon graphs generated, under both
composition and juxtaposition, by
\begin{eqnarray}\label{eq:wBMW-relat}
\begin{aligned}
X^+ - X^- - (q-q^{-1}) I_2^+,
\quad \Omega^- U^+ - z,
\quad  \Omega^+ U^- -z,
\end{aligned}
\end{eqnarray}
for any given $z\in \F$, and set $\cI\cH_{r, s}^{r, s}=\cH_{r, s}^{r, s}(\F)\cap \mathcal{IH}$.

\begin{definition}
Let $H_{r, s}(q, z; \F)=\cH_{r, s}^{r, s}(\F)/\cI\cH_{r, s}^{r, s}$, and call it
the walled BMW algebra with parameter $z$.
\end{definition}

\begin{proposition} \label{prop:walled-BMW}
For all nonnegative integers $r$ and $s$, the functor $\cF_{q, \F}$ restricts to a surjective algebra homomorphism
$H_{r, s}(q, [m-n]_q; \F)\longrightarrow\End_{\U_q(\F)}(V_q^{\otimes r}\otimes {V_q^*}^{\otimes s})$
from the walled BMW algebra with parameter $[m-n]_q=\frac{q^{m-n}-q^{-m+n}}{q-q^{-1}}$
to the endomorphism algebra of $V_q^{\otimes r}\otimes {V_q^*}^{\otimes s}$.
\end{proposition}
\begin{proof} This will follow from Theorem \ref{thm:main2-gl}
if we can show that $\cF_{q, \F}(\mathcal{IH})=\{0\}$.
We have already shown in the proof of Proposition \ref{prop:Hecke} that
\[
\cF_{q, \F}(X^+) - \cF_{q, \F}(X^-) - (q-q^{-1}) \id_{V_q\otimes V_q}=0.
\]
Now
$
\cF_{q, \F}(\Omega^- U^+)=\cF_{q, \F}(\Omega^+ U^-)=\sdim_q(V_q),
$
where $\sdim_q(V_q)=str_{V_q}(K)$ is the quantum superdimension \cite{Z92} of $V_q$
given by $\sdim_q(V_q)=\sum_{a=1}^{m+n} (-1)^{[a]}q^{(\mathcal{E}_a, 2\rho)}$.
It is known \cite{Z92} that  for the distinguished root datum,  $\sdim_q(V_q)=[m-n]_q$,
and by  Lemma \ref{rem:q-sdim} below, this is valid for any root datum of $\mathfrak{gl}_{m|n}$.
Hence $\cF_{q, \F}(\mathcal{IH})=\{0\}$, completing the proof.
\end{proof}

Now we prove the property of the quantum superdimension which was used in the proof of the proposition.
The quantum orthosymplectic supergroup is also treated here for later use.
\begin{lemma}\label{rem:q-sdim}
Let $\U_q(\fg, \phi; \F)$ be the quantum general linear supergroup or
quantum orthosymplectic supergroup, and let $V_q$ be the natural $\U_q(\fg, \phi; \F)$-module.
Then the quantum superdimension $\sdim_q(V_q)$ of $V_q$ is independent of the root datum
of $\fg$ used to define $\U_q(\fg, \phi; \F)$.
\end{lemma}
\begin{proof}
We note that any two root data of $\fg$ are related to each other by 
applying a sequence of odd reflections.
Given a root datum of $\fg$ with the set of positive roots $\Psi^+=\Psi_0^+\cup\Psi_1^+$
such that $\alpha_s$ is an isotropic odd simple root,  consider the
root datum obtained from it by the odd reflection with respect to $\alpha_s$.
If ${\Psi'}^+ = {\Psi'}_0^+\cup{\Psi'}_1^+$ is the set of positive roots of the latter,
then ${\Psi'}^+ =\left(\Psi^+\backslash\{\alpha_s\}\right)\cup\{-\alpha_s\}$.
Thus $2\rho-2\rho'=-2\alpha_s$, where $2\rho=\sum\limits_{\alpha\in\Psi^+_0}\alpha-
\sum\limits_{\gamma\in\Psi^+_1}\gamma$
and $2\rho'$ is similarly defined.  Recall that
\begin{eqnarray}\label{eq:vanishing}
(2\rho, \alpha_s)=(2\rho', \alpha_s)=0
\end{eqnarray}

For $\fg=\mathfrak{gl}_{m|n}$, there exist $i$ and $j$ such that
$\alpha_s =\varepsilon_i - \delta_j$. Then the difference between the quantum superdimensions
of $V_q$ relative to the two root data is
\[
D^+_{i j} - {D'}^+_{i j}, \quad \text{with \ } D^+_{i j}=q^{(2\rho, \epsilon_i)} - 
q^{(2\rho, \delta_j)}, \quad {D'}^+_{i j}= q^{(2\rho', \epsilon_i)} - q^{(2\rho', \delta_j)}.
\]
By \eqref{eq:vanishing}, $D^+_{i j}= q^{(2\rho, \delta_j)} \left(q^{(2\rho, \alpha_s)} - 1\right)=0$ and similarly
${D'}^+_{i j}=0$,  hence $D^+_{i j} - {D'}^+_{i j}=0$.
  If $\fg=\mathfrak{osp}_{m|2n}$, then
$\alpha_s =\varepsilon_i - \delta_j$ or $\varepsilon_i + \delta_j$ for some $i$ and $j$.
The difference between the quantum superdimensions
of $V_q$ relative to the two root data is $\cD_{i j}:=D^+_{i j} - {D'}^+_{i j}+ D^-_{i j} - {D'}^-_{i j}$
with $D^+_{i j}$ and ${D'}^+_{i j}$ as above and
\[
D^-_{i j}= q^{-(2\rho, \epsilon_i)} - q^{-(2\rho, \delta_j)}, \quad
{D'}^-_{i j}=q^{-(2\rho', \epsilon_i)} - q^{-(2\rho', \delta_j)}.
\]
Again $\cD_{i j}$ vanishes by \eqref{eq:vanishing}.
Hence $\sdim_q(V_q)$ is independent of the root datum chosen.
\end{proof}

\section{Invariant theory of the quantum orthosymplectic supergroup}\label{sect:inv-theory-osp}
Throughout this section, $\fg=\mathfrak{osp}_{m|2n}$ and $\ell=\left[\frac{m}{2}\right]$.

\subsection{Inequivalent root data of the orthosymplectic superalgebra} \label{sect:q-OSp}
The roots of  $\fg=\mathfrak{osp}_{m|2n}$ can be described as vectors in $\cE(\ell|n)$
(see Definition \ref{weight-space}) as follows:
\[
\begin{aligned}
&\mathfrak{osp}_{2 \ell+1|2n}:&& \pm\varepsilon_i, \
\pm\varepsilon_i\pm\varepsilon_{i'}\  (i\ne i'), \quad \pm2\delta_j,\
\pm\delta_j\pm\delta_{j'}\  (j\ne j'), \quad \pm \varepsilon_i\pm \delta_j; \\
&\mathfrak{osp}_{2 \ell|2n}:&&  \pm\varepsilon_i\pm\varepsilon_{i'}
\  (i\ne i'), \quad \pm2\delta_j,\  \pm\delta_j\pm\delta_{j'}\  (j\ne j'), \quad \pm \varepsilon_i\pm \delta_j.
\end{aligned}
\]
When $\ell=0$, there is no $\varepsilon_i$.
In analogy with the type $A$ case, the Weyl group conjugacy classes of
Borel subalgebras correspond bijectively
to admissible orderings ${\mathcal E}_1, {\mathcal E}_2, \dots, {\mathcal E}_{ \ell+n}$ of the basis elements
of $\cE(\ell|n)$.
For a given admissible ordering,
\\
$\bullet$ the set $\Pi_\fb$ of simple roots of $\mathfrak{osp}_{2 \ell+1|2n}$ is
\begin{eqnarray}\label{eq:odd-o}
\left\{{\mathcal E}_1 - {\mathcal E}_2, \dots, {\mathcal E}_{ \ell+n-1} - {\mathcal E}_{ \ell+n},  \  {\mathcal E}_{ \ell+n}\right\};
\end{eqnarray}
\noindent $\bullet$ the set $\Pi_\fb$ of simple roots of $\mathfrak{osp}_{2\ell|2n}$ is
\begin{eqnarray}\label{eq:even-o}
&&\left\{{\mathcal E}_1 - {\mathcal E}_2, \dots, {\mathcal E}_{ \ell+n-1} - {\mathcal E}_{ \ell+n}, \
{\mathcal E}_{ \ell+n-1} +  {\mathcal E}_{ \ell+n}\right\}, \  \text{where ${\mathcal E}_{ \ell+n}=\varepsilon_ \ell$}, \ \text{or}
\\
&& \left\{{\mathcal E}_1 - {\mathcal E}_2, \dots, {\mathcal E}_{ \ell+n-1} - {\mathcal E}_{ \ell+n}, \
2 {\mathcal E}_{ \ell+n}\right\}, \  \text{where  ${\mathcal E}_{ \ell+n-1}=\delta_{n-1}$,  ${\mathcal E}_{ \ell+n}=\delta_n$}.
\end{eqnarray}

The even subalgebra of $\fg=\mathfrak{osp}_{m|2n}$ is $\fg_{\bar0}
=\mathfrak{so}_{m}\oplus \mathfrak{sp}_{2n}$, while the odd subspace is
$\fg_{\bar 1}=\C^m\otimes \C^{2n}$. Let
$G_0=O_n(\C)\times Sp_{2n}(\C)$, whose Lie algebra is $Lie(G_0)=\fg_{\bar 0}$.
Then $G_0$ acts on $\fg$ in its adjoint representation, and we obtain a Harish-Chandra pair $(G_0, \fg)$.
Now define an involution $\sigma\in G_0$ as follows.
If $m=2\ell +1$, we let $\sigma=-1$. If $m=2\ell$, let $\sigma\in G_0$ be the element
that acts on the natural  $\fg$-module by interchanging the weight spaces
with weights $\varepsilon_{\ell}$ and $-\varepsilon_{\ell}$ and which fixes all other weight spaces.
Then $G_0$ is  generated by the connected component of the identity and $\sigma$.

Consider the action of $\sigma$  on $\fg$. In the case $m=2\ell+1$,
the action is trivial. If
$m=2\ell$, it interchanges the root spaces of
$\pm({\mathcal E_a}+\varepsilon_\ell)$ and $\pm({\mathcal E_a}-\varepsilon_\ell)$
for each ${\mathcal E_a}\ne\varepsilon_\ell$,
and for any $h^0$ in the Cartan subalgebra, $\sigma.h^0=\sigma h^0 \sigma^{-1}$ satisfies
$\sigma.h^0({\mathcal E_a})=h^0({\mathcal E_a})$ for all ${\mathcal E_a}\ne\varepsilon_\ell$,
and $\sigma.h^0(\varepsilon_\ell)=-h^0(\varepsilon_\ell)$.

The action of $\sigma$ extends to $\U(\mathfrak{osp}_{m|2n}; T)$.
Let $\widetilde{\U}(\mathfrak{osp}_{m|2n}; T)$ be the smash product of $\U(\mathfrak{osp}_{m|2n}; T)$
with the group algebra
$T\Z_2$ of $\Z_2:=\{1, \sigma\}$. Then $\widetilde{\U}(\mathfrak{osp}_{m|2n}; T)$
is a Hopf superalgebra with the usual
Hopf superalgebra structures for $\U(\mathfrak{osp}_{m|2n}; T)$
and for the group algebra. In particular,
\[
\Delta(\sigma)=\sigma\otimes\sigma, \  \  \epsilon(\sigma)=1, \  \  S(\sigma)=\sigma^{-1}.
\]
Clearly $\Delta(\sigma) C\Delta( \sigma^{-1})=C$, 
where $C$ is the quadratic
Casimir. The universal $R$-matrix $R$ and the associator $\Phi_{KZ}$ of
$\U(\mathfrak{osp}_{m|2n}; T)$ (see Section \ref{sect:uea}) satisfy
\[
(\sigma\otimes\sigma) R(\sigma^{-1}\otimes\sigma^{-1})=R, \quad (\sigma\otimes\sigma \otimes\sigma) \Phi_{KZ}(\sigma^{-1}
\otimes\sigma^{-1}\otimes\sigma^{-1})=\Phi_{KZ}.
\]
Thus we have the ribbon quasi Hopf superalgebra
\begin{eqnarray}\label{eq:Uosp}
\left(\widetilde{\U}(\mathfrak{osp}_{m|2n}; T), \Delta, \epsilon, \Phi_{KZ}, S, \alpha, \beta, R, v\right).
\end{eqnarray}

\medskip

\subsection{Invariant theory: the case $m=2\ell+1$}\label{sect:inv-odd-m}

Now we turn to the quantum orthosymplectic supergroup with $m=2\ell+1$.
Let $\Z_2=\{1, \sigma\}$  act on
$\U_q(\mathfrak{osp}_{m|2n}, \phi; T)$ by conjugation and assume that
all elements of $\U_q(\mathfrak{osp}_{m|2n}, \phi; T)$ are fixed.
Let $\widetilde{\U}_q({\mathfrak{osp}}_{m|2n}, \phi; T)$ be the smash product of
$\U_q(\mathfrak{osp}_{m|2n}, \phi; T)$ with the group algebra $T\Z_2$.  It has a natural
Hopf superalgebra structure.
Since the universal $R$-matrix $R_q$ of $\U_q(\mathfrak{osp}_{m|2n}, \phi; T)$ is unique,
we have  $(\sigma\otimes\sigma) R_q(\sigma^{-1}\otimes\sigma^{-1})=R_q$.
Thus we have the ribbon Hopf superalgebra,
\begin{eqnarray}\label{eq:Uqosp}
\left(\widetilde{\U}_q({\mathfrak{osp}}_{m|2n}, \phi; T), R_q, v_q\right),
\end{eqnarray}
where $R_q$ is the universal $R$-matrix of $\U_q(\mathfrak{osp}_{m|2n}, \phi; T)$.
We will also call this ribbon Hopf superalgebra the quantum orthosymplectic supergroup.

\subsubsection{Extending scalars to power series}

Set $\K=T$ or $T_t$. Use $\U(T)$ and $\U_q(T)$ to denote
$\widetilde{\U}(\mathfrak{osp}_{m|2n}; T)$ and
$\widetilde{\U}_q(\mathfrak{osp}_{m|2n}, \phi; T)$
respectively, and define $\U(T_t)$ and  $\U_q(T_t)$
as in equation \eqref{eq:U(R)}.
Then we have the ribbon Hopf superalgebra
$(\U_q(\K), \Delta_q, \epsilon_q, S_q, R_q, v_q)$ and
ribbon quasi Hopf superalgebra $(\U(\K), \Delta, \epsilon, \Phi_{KZ}, S, \alpha, \beta, R, v)$.

\begin{theorem}\label{thm:equiv-mod-osp} Let $\K$ be $T$ or $T_t$, and
$\fg=\mathfrak{osp}_{m|2n}$.
\begin{enumerate}
\item There is an equivalence of ribbon quasi Hopf superalgebras

$(\U_q(\K), \Delta_q, \epsilon_q, S_q, R_q, v_q)
\longrightarrow(\U(\K), \Delta, \epsilon, \Phi_{KZ}, S, \alpha, \beta, R, v)$.
\item There exists a braided tensor equivalence $\U(\K)\text{-mod}
\longrightarrow \U_q(\K)\text{-mod}$, which preserves
duality and twist.
\end{enumerate}
\end{theorem}

\begin{proof}
We extend the superalgebra isomorphism of Theorem \ref{thm:iso} to an isomorphism $\U_q(\K)\longrightarrow \U(\K)$
in such a way that it is the identity map on the group algebra $T\Z_2$. Then the theorem follows immediately
\end{proof}

\medskip

The natural module $V$ for $\mathfrak{osp}_{m|2n}$ over $\C$ admits a
non-degenerate bilinear form $(\ , \ ): V\times V\longrightarrow \C$,
which is
\begin{itemize}
\item even and supersymmetric, i.e.,
for $v_i, w_i\in V_{\bar i}$ with $i=0, 1$,
\[
(v_0, w_1)=(v_1, w_0)=0, \quad
(v_0, w_0)=(w_0, v_0), \quad (v_1, w_1)= -(w_1, v_1);
\]
\item and contravariant, i.e., for all $v, w\in V$,
\begin{eqnarray}\label{eq:contravariance}
(x v, w) = (-1)^{[v][x]} (v, S(x) w), \quad x\in \U(\mathfrak{osp}_{m|2n}; \C).
\end{eqnarray}
\end{itemize}
Given the action of $\Z_2=\{1, \sigma\}$ on $V$ defined earlier,
the form is also contravariant with respect to $\widetilde{\U}(\mathfrak{osp}_{m|2n}; \C)$,
the smash product
of $\U(\mathfrak{osp}_{m|2n}; \C)$ and $\C\Z_2$.

Now  $V_\K=V\otimes_\C\K$  has the structure of a module for the ribbon quasi Hopf superalgebra
$\left(\U(\K), \Delta, \epsilon, \Phi_{KZ}, S, \alpha, \beta, R, v\right)$
(see \eqref{eq:Uosp}).  We extend the form on $V$ $\K$-bilinearly to obtain a form on $V_\K$.
The non-degeneracy and contravariance of this form enables us to define the
$\U(\K)$-module isomorphism
\[
\kappa: V_\K\longrightarrow V_\K^*, \quad \kappa(v)(\alpha w)=(v, w),\  \forall v, w\in V_\K,
\]
where $\alpha$ is involved for the same reasons as it entered in \eqref{eq:ev+exp}.
Recall from \eqref{eq:beta} that we may choose $\alpha=1$.

Let $\cI: \U(\K)\text{-mod}\longrightarrow \U_q(\K)\text{-mod}$ be
the equivalence of categories in Theorem \ref{thm:equiv-mod-osp}(2),
and continue to denote $\cI(V_\K)$ and $\cI(V_\K^*)$ by $V_\K$ and $V_\K^*$
respectively.  Denote by $\cT(\K)$ the full subcategory of $\U(\K)\text{-mod}$
with objects of the form
$\underbrace{V_\K \boxtimes V_\K \boxtimes\dots \boxtimes V_\K}_k$ for all $k$.
Then we have the full subcategory $\cT_q(\K)$ of
$\U_q(\K)\text{-mod}$ with tensor powers $V_\K \otimes V_\K \otimes \dots \otimes _\K$
of $V_\K$ as objects.
The same arguments used in the proof of Lemma \ref{lem:functors-gl-1}
suffice to prove the following result.
\begin{lemma}\label{lem:equiv-functors-1}
Both $\cT(\K)$  and $\cT_q(\K)$ are ribbon categories,
and one has a braided tensor equivalence
$
\cT(\K)\stackrel{\sim}{\longrightarrow} \cT_q(\K)
$
which preserves duality and twist.
\end{lemma}

The map $\kappa_q=\cI(\kappa): V_\K\longrightarrow V_\K^*$ is an isomorphism of
$\U_q(\K)$-modules. We now define a $\K$-bilinear form  on $V_\K$ by
\begin{eqnarray}
(  \  ,  \ )_q: V_\K\times V_\K\longrightarrow\K, \quad \kappa_q(v)(w)=(v, w)_q,\  \forall v, w\in V_\K.
\end{eqnarray}
This form is $\U_q(\K)$-contravariant, that is,
$(x v, w)_q = (-1)^{[v][x]} (v, S_q(x) w)_q$ for all $v, w\in V_\K$ and $x\in\U_q(\K)$.
Furthermore,
$
(v, w)_q = (-1)^{[v][w]} (w, K_{2\rho} v)_q,
$
where $K_{2\rho}$ is the product of $k_i$ such that $K_{2\rho} e_i K_{2\rho}^{-1}=q^{(2\rho, \alpha_i)}e_i$
for all $i$, and hence $S^2(x)=K_{2\rho}xK_{2\rho}^{-1}$ for all $x\in\U_q(\K)$.

Let $\cH'(\K)$ be the category of non-directed ribbon graphs (cf. Definition \ref{def:ribbon-nondirect}).
Applying Theorem \ref{thm:RT-1}, we obtain braided tensor functors
$
\cF_\K: \cH'(\K)\longrightarrow \cT(\K)$ and $
\cF_{q, \K}: \cH'(\K)\longrightarrow \cT_q(\K),
$
which preserve dualities and twists, and each of which is unique.
Then both $\cF_\K(\cH'(\K))$ and $\cF_{q, \K}(\cH'(\K))$ are
ribbon categories, and it follows Theorem \ref{thm:equiv-mod-osp} and the uniqueness
of the braided tensor functors that there exists a braided tensor equivalence
\begin{eqnarray}\label{eq:equal-images}
\cF_\K(\cH'(\K)) \stackrel{\sim}{\longrightarrow}  \cF_{q, \K}(\cH'(\K))
\end{eqnarray}
which preserves duality and twist.
We have the following result.
\begin{theorem} \label{thm:main1-osp}
The  braided tensor functors
$
\cF_\K: \cH'(\K)\longrightarrow \cT(\K)$
and $
\cF_{q, \K}: \cH'(\K)\longrightarrow \cT_q(\K)
$ are both full.
\end{theorem}
\begin{proof}
We prove the theorem by adapting the proof of Theorem \ref{thm:main1}; the notation used
in that proof will be retained here.  Note that
the arguments up to \eqref{eq:tensoring} are all applicable here
if we replace $\cH_r^r(\K)$ by ${\cH'}_r^r(\K)=\Hom_{\cH'(\K)}(r, r)$,
and $\U(\mathfrak{gl}_{m|n}; \C)$ by $\U(\C)=\widetilde{\U}(\mathfrak{osp}_{m|2n}; \C)$.

Let $G$ denote the Harish-Chandra super pair $(O_m\times Sp_{2n}, \mathfrak{osp}_{m|2n})$.
The endomorphism algebra $\End_G(V^{\otimes_\C r})$ is given by the FFT of invariant theory
of the orthosymplectic supergroup \cite{DLZ, LZ14a}. One sees readily that
$\End_{\U(\C)}(V^{\otimes_\C r})= \End_G(V^{\otimes_\C r})$.

Let $\hat{c}_0: V\otimes V \longrightarrow \C$ and $\check{c}_0: \C\longrightarrow V\otimes V$ be maps
respectively defined by $\hat{c}_0(v, w)= (v, w)$ and
$(\id\otimes \hat{c}_0)(\check{c}_0(1)\otimes v)=v$  for all $v, w\in V$.
Then $E=\check{c}_0\circ\hat{c}_0\in \End_{G}(V\otimes V)$.
Define the following maps,
\[
E_i =  \underbrace{\id_V\otimes\dots\otimes\id_V}_{i-1}\otimes E \otimes
\underbrace{\id_V\otimes\dots\otimes\id_V}_{r-i-1}, \quad i=1, 2, \dots, r-1,
\]
which belong to $\End_{G}(V^{\otimes r})$.
Let  $B_r(m-2n, \C)$ be the complex Brauer algebra of degree $r$ with parameter $m-2n$.
Then the maps $E_i$ together
with the representation of $\Sym_r$ defined by \eqref{eq:Sym}
generate a representation $\nu_r$ of $B_r(m-2n, \C)$ on $V^{\otimes r}$.
The first fundamental theorem of invariant theory for
the orthosymplectic supergroup states that \cite{DLZ, LZ14a} $\nu_r$
surjects onto $\End_{G}(V^{\otimes r})$.
Denote by $B_r(m-2n, \K)$ the Brauer algebra over $\K$, and extend
this representation $\K$-linearly to $\nu_r: B_r(m-2n, \K)
\longrightarrow\End_{\U(\K)}(V_\K^{\otimes_\K r})$.
It then follows from the analogue of \eqref{eq:tensoring} that $\nu_r$ is surjective.
Thus in order to prove the theorem,
we only need to to show that $\tau_i, E_i\in \cF_\K({\cH'}_r^r(\K))$ for all $i$,
where ${\cH'}_r^r(\K)=\Hom_{\cH'(\K)}(r, r)$.

Assume that the analogue of the statement \eqref{quote} remains valid in the present case.
Then  $\tau_i\in \cF_\K({\cH'}_r^r(\K))$.
Now $E= \cF_\K(U) \cF_\K(\Omega)$, thus
$\tilde{E}_i:= \id_{V_\K^{\boxtimes(i-1)}}\boxtimes E \boxtimes \id_{V_\K^{\boxtimes(r-i-1)}}$
belongs to $\cF_\K({\cH'}_r^r(\K))$ for all $i$.
These endomorphisms involve only $E$ and the associator. Since the associator involves only
quadratic Casimirs, it follows from the statement \eqref{quote} that $E_i\in \cF_\K({\cH'}_r^r(\K))$.

Now we prove  the analogue of the statement \eqref{quote} in the present case.
If $m=2n$, equations \eqref{eq:eigen-b} and \eqref{eq:eigen-tau}
are all valid, and hence the proof of \eqref{quote} still goes through.  If $m\ne 2n$,
the tensor square of the natural module $V$ of $\U(\C)$ decomposes into the direct sum of three irreducibles
$ L(\tilde s)$, $L(a)$ and $L(0)$, where $L(a)$ is the $\Z_2$-graded skew symmetric rank-2 tensor of $V$.
Let $P[\tilde s]$, $P[a]$ and $P[0]$ be the idempotents mapping $V\otimes V$ onto the respective simple modules.
Then $\tau =  P[\tilde s]-P[a]+P[0]$.
The eigenvalue of $C$ on each $L(\psi)$ is given by
$\chi_\psi= \omega_{ L(\psi)}/2-\omega_{V}$, which
were computed in \cite{ZGB91b} and \cite[\S C.1]{Z92a}, and we have
\begin{eqnarray}\label{eq:eigen-Casimir}
\chi_{\tilde s}=-\chi_a=1,  \quad \chi_0= -m+2n+1.
\end{eqnarray}
Let $q=\exp(t/2)$. Then
\[
\begin{aligned}
b= qP[\tilde s]  +  q^{-1}P[a] + q^{-m+2n+1}P[0],  \\
g= qP[\tilde s]  -  q^{-1}P[a] + q^{-m+2n+1}P[0].
\end{aligned}
\]
We can express the idempotents as polynomials of $g$ in the standard fashion, and this in turn leads to
\[
\begin{aligned}
\tau= &\frac{(g+q^{-1})(g-q^{-m+2n+1})}{(q+q^{-1})(q-q^{-m+2n+1})}-  \frac{(g-q)(g-q^{-m+2n+1})}{(q+q^{-1})(q^{-1}+q^{-m+2n+1})}\\
&+\frac{(g+q^{-1})(g-q)}{(q^{-m+2n+1}+q^{-1})(q^{-m+2n+1}-q)}.
\end{aligned}
\]
Since all the $g_i$ belong to $\cF_\K({\cH'}_r^r(\K))$ by Theorem \ref{thm:RT-1},
we conclude that all $\hat\tau_i$ and $b_i$ also belong to  $\cF_\K({\cH'}_r^r(\K))$.
Now the same arguments as those following equation \eqref{eq:gbinv}
show that the statement \eqref{quote}  remains true in the present case.

This completes the proof.
\end{proof}

\subsubsection{Extending scalars to the field $\C(q)$ of rational functions}
Let $\F=\C(q)$. Denote by $\U_q(\F)$ the Jimbo version of the
quantum orthosymplectic supergroup over $\F$, which is generated by
by $k_i^{\pm 1}, e_i f_i$ ($i=1, 2, \dots, \ell+n$)
and $\sigma$ subject to the same relations as those over $T$.
Then  $\U_q(\F)\text{-Mod}_{f, \1}$
has a braiding arising from a functorial $R$-matrix $R_{W, W'}$
for any pair of objects $W$ and $W'$.
There also exists a family of functorial isomorphisms $v_W: W\longrightarrow W$
which gives rise to a twist. Thus $\U_q(\F)\text{-Mod}_{f,\1}$ is a ribbon category.

Let $V_q$ be the natural $\U_q(\F)$-module, which is self-dual, thus admits
a non-degenerate contravariant bilinear. We require the form coincide
with the bilinear form \eqref{eq:contravariance} upon specialising $V_q$ to $T_t$.

\begin{definition}\label{def:Tq-osp}
The category $\cT_q(\F)$ of tensor modules for $\U_q(\F)$ is the full subcategory
of $\U_q(\F)\text{-Mod}_{f,\1}$ with objects  $V_q^{\otimes k}$ for all $k\in\Z_+$.
This is a ribbon category.
\end{definition}

Applying Theorem \ref{thm:RT-1}  in the present context, we obtain the
unique braided tensor functor  $\cF_{q, \F}: \cH'(\F)\longrightarrow \cT_q(\F)$.
We have the following result.
\begin{theorem}[FFT for the quantum supergroup of $\osp_{2\ell+1|2n}$]\label{thm:main2-osp}
The braided tensor functor $\cF_{q, \F}: \cH'(\F)\longrightarrow \cT_q(\F)$
is full and preserves duality and twist.
\end{theorem}
\begin{proof}
The same arguments in the proof of Theorem \ref{thm:main2-gl} apply here.
\end{proof}

\subsection{Invariant theory: the case $m=2\ell$}\label{sect:inv-even-m}

In this case, the action of $\Z_2=\{1, \sigma\}$ on $\osp_{m|2n}$ may send
simple root vectors to non-simple ones. As $\U_q(\mathfrak{osp}_{m|2n}, \phi; T)$ is defined by
generators and relations involving only the simple roots, there is no natural way to extend the action
of $\Z_2$ to the quantum supergroup, so we can not modify the quantum supergroup
as we have done when $m$ is odd (however, see Remark \ref{rem:even-osp-reflection}).
Therefore,  we let $\U(\K)$, $\U_q(\K)$, and $\U_q(\F)$ respectively denote
$\U({\mathfrak{osp}}_{m|2n}; \K)$, $\U_q({\mathfrak{osp}}_{m|2n}, \phi; \K)$
and $\U_q({\mathfrak{osp}}_{m|2n}, \phi; \F)$ for $\K=T$ or $T_t$, and $\F=\C(q)$,
and retain the notation of Section \ref{sect:inv-odd-m}.
Then Theorem \ref{thm:RT-1} yields unique braided tensor functors
\begin{eqnarray}\label{eq:even-osp-functors}
\begin{aligned}
&\cF_\K: \cH'(\K)\longrightarrow \cT(\K), \quad
\cF_{q, \K}: \cH'(\K)\longrightarrow \cT_q(\K), \\
&\cF_{q, \F}: \cH'(\F)\longrightarrow \cT_q(\F).
\end{aligned}
\end{eqnarray}
Clearly Theorem \ref{thm:equiv-mod-osp} holds in the present case, and it follows that
Lemma \ref{lem:equiv-functors-1} and equation \eqref{eq:equal-images} remain valid.
This enables one to obtain the following result.
\begin{proposition} \label{prop:main-even-osp}
For any nonnegative integers $r$ and $s$ such that $r+s<m(2n+1)$,
\begin{eqnarray}\label{eq:main-even-osp}
\cF_{q, \K}({\cH'}_r^s(\F))=\Hom_{\U_q(\F)}(V_q^{\otimes r}, V_q^{\otimes r}).
\end{eqnarray}
\end{proposition}
\begin{proof}
Recall from \cite{LZ14b, LZ15} (see  \cite[Corollary 5.7]{LZ15} in particular) that for any nonnegative integer $k<\frac{m(2n+1)}{2}$,
the endomorphism algebra of the
$k$-th tensor power of the natural module of $\osp(m|2n;\C)$ is a quotient of the Brauer algebra of degree $k$ with
parameter $m-2n$. This leads to $\cF_\K({\cH'}_k^k(\K))= \End_{\U(\K)}(V_\K^{\boxtimes_\K k})$ by similar arguments as
those used in the proof of Theorem \ref{thm:main1-osp}.
Then by using Lemma \ref{lem:equiv-functors-1} and equation \eqref{eq:equal-images} for even $m$,
we obtain $\cF_{q, \K}({\cH'}_k^k(\K))=\End_{\U_q(\K)}(V_\K^{\otimes_\K k})$.  Now we can use the same reasoning
as that in the proof of Theorem \ref{thm:main2-gl} to show that
\[
\cF_{q, \F}({\cH'}_k^k(\F))=\End_{\U_q(\F)}(V_q^{\otimes k}).
\]
This immediately leads to the proposition by using the dualities of the tensor categories involved.
\end{proof}

However,  Theorems \ref{thm:main1-osp} and \ref{thm:main2-osp} both fail, as
the tensor functors in the theorems are no longer full  in this case.

\begin{remark} \label{rem:even-osp-reflection}
If $\Pi_\fb$ is given by \eqref{eq:even-o}, one may define  the action
of $\Z_2=\{1, \sigma\}$ on $\U_q(\mathfrak{osp}_{m|2n}, \phi; T)$ so that
$\sigma$ sends
\[
\begin{aligned}
e_{\ell+n-1}\mapsto e_{\ell+n}, \quad  f_{\ell+n-1}\mapsto f_{\ell+n},
\quad h_{\ell+n-1}\mapsto h_{\ell+n},\\
e_{\ell+n}\mapsto e_{\ell+n-1}, \quad   f_{\ell+n}\mapsto f_{\ell+n-1},
\quad h_{\ell+n}\mapsto h_{\ell+n-1},
\end{aligned}
\]
but leaves all other generators invariant.
Then we can modify quantum $\mathfrak{osp}_{2\ell|2n}$ as in Section \ref{sect:inv-odd-m}, and
we expect equation  \eqref{eq:main-even-osp} to hold at arbitrary $r$ and $s$ for
this modified quantum supergroup.
\end{remark}
\subsection{A representation of the BMW algebra}

We use $\cT_q(\F)$ to denote the category of tensor $\U_q(\F)$-modules for both even and odd $m$.
Let $I$, $X_i^+$ and $X_i^-$ ($1\le i\le r-1$) be non-directed ribbon graphs respectively
depicted in Figure \ref{fig:Br-generators}, and let $E_i$ be the
non-directed ribbon graph shown in Figure \ref{fig:E}.
\begin{figure}[h]
\begin{center}
\setlength{\unitlength}{.7mm}
\begin{picture}(140, 20)(5,40)

\put(10, 50){$E_i=$}
\put(30,60){\line(0,-1){20}}
\put(38,50){...}
\put(36,40){\tiny $i-1$}

\put(50,60){\line(0,-1){20}}
\put(70,60) {\oval(18,18)[b]}
\put(70,40){\oval(18, 18)[t]}

\put(90,60){\line(0,-1){20}}
\put(98,50){...}
\put(92,40){\tiny $r-i-1$}
\put(110,60){\line(0,-1){20}}

\end{picture}
\end{center}
\caption{$E_i$ generators}
\label{fig:E}
\end{figure}
Then ${\cH'}_r^r(\F)$ as an associative algebra is generated by the elements $X_i^+$, $X_i^-$ and $E_i$
by composition of non-directed ribbon graphs. Note that $I$ is the identity of ${\cH'}_r^r(\F)$,
and $X_i^+X_i^-=X_i^-X_i^+=I$.
Thus ${\cH'}_r^r(\F)$ contains the group algebra of the braid group $B_r$
generated by the elements $X_i^+$ and $X_i^-$.

Write $g=\cF_{q, \F}(X^+)$ and $e=\cF_{q, \F}(U\Omega)$. Then $g^{-1}=\cF_{q, \F}(X^-)$ and
\begin{eqnarray}\label{eq:BMW}
\begin{aligned}
&e^2=\sdim_q(V_q) e,  \quad  e g=ge= q^{-\omega_{V}} e, \\
&(\id\otimes g)(g\otimes \id) (\id\otimes g)=
(g\otimes\id)(\id\otimes g)(g\otimes\id), \\
&(\id\otimes e)(g^{\pm 1}\otimes \id) (\id\otimes e)=q^{\pm\omega_V }(\id\otimes e), \\
&(e\otimes\id)(\id\otimes g^{\pm 1})(e\otimes\id)=q^{\pm\omega_V} (e\otimes\id), \\
&(\id\otimes e)(e\otimes \id) (\id\otimes e)=\id\otimes e, \\
&(e\otimes\id)(\id\otimes e)(e\otimes\id)=e\otimes\id,
\end{aligned}
\end{eqnarray}
where $\id=\id_{V_q}$ and $\omega_V=m-2n-1$ is the eigenvalue
of the Casimir operator of $\U(\fg; \C)$ acting on $V$. Also, $\sdim_q(V_q)$ denotes
the quantum superdimension of $V_q$, which was computed explicitly for the distinguished
root datum \cite[\S C.1]{Z92a}, and hence for all root data by Lemma \ref{rem:q-sdim}.
We have
\[
\sdim_q(V_q)= 1 + \frac{q^{m-2n-1} - q^{-m+2n +1}}{q-q^{-1}}.
\]
Note that
$\sdim_q(V_q)=0$ if and only if $m=2n$.   Let
\[
g_i :=\cF_{q, \F}(X^+_i), \quad g^{-1}_i :=\cF_{q, \F}(X^-_i), \quad e_i :=\cF_{q, \F}(E_i).
\]
We have
\[
\begin{aligned}
g_i &=\underbrace{\id\otimes \dots \otimes \id}_{i-1}\otimes g \otimes \underbrace{\id\otimes \dots \otimes \id}_{r-i-1}, \\
g_i^{-1} &=\underbrace{\id\otimes \dots \otimes \id}_{i-1}\otimes g^{-1} \otimes \underbrace{\id\otimes \dots \otimes \id}_{r-i-1}, \\
e_i &=\underbrace{\id\otimes \dots \otimes \id}_{i-1}\otimes e \otimes \underbrace{\id\otimes \dots \otimes \id}_{r-i-1}.
\end{aligned}
\]
Denote by $\nu_r: {\cH'}_r^r(\F)\longrightarrow \End_{\U_q(\F)}(V_q^{\otimes r})$ the representation of ${\cH'}_r^r(\F)$
given by the functor $\cF_{q, \F}$.
\begin{proposition}\label{prop:BMW}
The representation $\nu_r$ of ${\cH'}_r^r(\F)$
factors through the BMW algebra $BMW_r(y, z)$ with
$z=q-q^{-1}$ and $y=q^{-m+2n+1}$ in the notation of \cite[\S 4.2]{LZ10}.
\end{proposition}
\begin{proof} Let us first assume that $m\ne 2n$. In this case,
the theorem can be proven by formal arguments,
which we briefly outline below.
The tensor square of $V_q$ is semi-simple, and we have
\[
V_q\otimes V_q = L_q(\tilde s) \oplus L_q(a)\oplus L_q(0),
\]
with $\dim L_q(0)=1$. Here the simple submodule $L_q(\tilde s)$ is the
$q$-analogue of the supertraceless submodule $L(\tilde s)$ of the $\Z_2$-graded symmetric
tensor $S^2 V$, $L_q(0)$ is that of the supertrace $L(0)\subset S^2 V$, and
$L_q(a)$ is that of the $\Z_2$-graded skew symmetric tensor
$L(a)=\wedge^2 V$.

Denote by $P[\psi]$ the idempotent
mapping $V_q\otimes V_q$ onto $L_q(\psi)$ for $\psi=\tilde{s}, a, 0$.
Then
\begin{eqnarray}\label{eq:spectr}
\cF_{q, \F}(X^+)=\tau\circ R_{V_q, V_q}= q^{\chi_{\tilde s}}P[\tilde{s}] - q^{\chi_a}P[a]+q^{\chi_0}P[0],
\end{eqnarray}
where the scalars $\chi_\psi$ are given by \eqref{eq:eigen-Casimir}

Note that $e$ is a quasi idempotent such that
$e(V_q\otimes V_q)= L_q(0)$.
As $\sdim_q(V_q)$ is nonzero when $m\ne 2n$,
we must have $e=\sdim_q(V_q)  P[0]$.  By using $P[\tilde{s}] +P[a]+P[0]=1$,
we can eliminate $P[\tilde{s}]$ from the spectral decompositions of  $g$ and $g^{-1}$
to obtain
\begin{eqnarray}\label{eq:g+g-in}
\begin{aligned}
g&= q- (q+q^{-1}) P[a] - \frac{q(q-q^{-1})}{q + q^{m-2n -1}} e, \\
g^{-1}&= q^{-1}- (q+q^{-1}) P[a] +\frac{q^{m-2n-1}(q-q^{-1})}{q + q^{m-2n -1}} e,
\end{aligned}
\end{eqnarray}
where we have also expressed $P[0]$ in terms of $e$. Taking the difference between these equations, we obtain
\begin{eqnarray}\label{eq:BMW-1}
g-g^{-1}=(q-q^{-1})\left(1- e\right).
\end{eqnarray}

From relations in \eqref{eq:BMW} and \eqref{eq:BMW-1}, we easily see that
the $g_i$ and $e_i$ generate a representation of the BMW algebra $BMW_r(y, z)$ with
$z=q-q^{-1}$ and $y=q^{-\omega_V}$ in the notation of \cite[\S 4.2]{LZ10}.

Now we assume $m=2n$. In this case,  $V_q\otimes V_q $ is not semi-simple, but the simple module
$L_q(a)$ is still a direct summand of it.  Explicitly,
$V_q\otimes V_q = L_q(a)\oplus M(s)$, where $M(s)$ is an indecomposable module
containing the $1$-dimensional submodule $e(V_q\otimes V_q)$.
Note that $g^{\pm 1}$, $P[a]$ and $e$ are all well defined irrespectively of $m$ and $n$.
Thus the equations in \eqref{eq:g+g-in} remain valid but with $m=2n$.
They can also be extracted from the long handed calculations for the $R$-matrix in \cite{S00a}
(also see \cite[\S C.1]{Z92a}).
Therefore \eqref{eq:BMW-1} is still satisfied. Hence the $g_i$ and $e_i$ generate
a representation of the BMW algebra with the given parameters.

This completes the proof of the theorem.
\end{proof}

We close by noting that an explicit expression of $R_{V_q, V_q}$ in matrix form is given in \cite{S00a}.

\medskip

\noindent{\bf Acknowledgments}.  This work was supported by the Australian Research Council
and the National Science Foundation of China. Part of this work was carried out during stays of H Zhang and R Zhang
at the School of Mathematical Sciences, the University of Science and Technology of China.



\appendix
\section{Categories of ribbon graphs}

This appendix contains some basic facts on braided tensor categories and
categories of ribbon graphs. The material is not new. It is included here for easy reference.
More details can be found \cite{Kass} and \cite{Turaev}.

%
%
%
\subsection{Making tensor categories strict}\label{sect:tensor}\label{append:tensor-cat}
A tensor category $(\cC, \otimes, a, I, l, r)$ is a category $\cC$ with a functor
$\otimes: \cC\times \cC\longrightarrow \cC$ called the tensor product, an object $I$ called the unit,
an associativity constraint $a$ for the tensor product, a left unit constraint $\ell$ and
right unit constraint $r$ with respect to $I$,
such that the pentagon and triangle axioms are satisfied
(see \cite[\S XI-\S XIV]{Kass}). If the associativity and unit constraints are all identities,
the tensor category is said to be strict.

A tensor functor $(F, \varphi_0, \varphi_2): (\cC, \otimes, a, I, \ell, r)
\longrightarrow(\cD, \otimes, a, I', \ell, r)$ between two tensor categories consists of
 a functor $F:\cC\longrightarrow \cD$,
 an isomorphism $\varphi_0: I'\longrightarrow F(I)$ in $\cD$,  and
 a family of natural isomorphisms $\varphi_2(V, W): F(V)\otimes F(W)\longrightarrow F(V\otimes W)$
indexed by pairs $(U, V)$ of objects in $\cC$,
which satisfy certain conditions involving the associativity and unit constraints.
The tensor functor is said to be strict if $\varphi_0$ and $\varphi_2$ are identities of $\cD$.

Let $(F, \varphi_0, \varphi_2), (F', \varphi'_0, \varphi'_2): \cC\longrightarrow \cD$ be tensor functors.
A natural tensor transformation $(F, \varphi_0, \varphi_2)\longrightarrow  (F', \varphi'_0, \varphi'_2)$
is a natural transformation $\eta: F\longrightarrow  F'$ such that the following diagrams
commute
\[
\xymatrix{
&I'\ar[dl]_{\varphi_0}\ar[dr]^{\varphi'_0}&\\
F(I)\ar[rr] ^{\eta(I)}&&F'(I),
} \quad\quad
\xymatrix{
F(U)\otimes F(V)\ar[d]_{\eta(U)\otimes\eta(V)}\ar[rr]^{\varphi_2(U, V)}&& F(U\otimes V)\ar[d]^{\eta(U\otimes V)}\\
F'(U)\otimes F'(V)\ar[rr]^{\varphi'_2(U, V)}&& F'(U\otimes V)
}
\]
for all objects $U, V$ of $\cC$.  It is called a  natural tensor isomorphism if $\eta$ is a natural isomorphism.
A tensor equivalence between tensor categories is a tensor functor $F:\cC\longrightarrow\cD$ such that there exist
a tensor functor $F': \cD\longrightarrow\cC$ and natural tensor isomorphisms $\eta: \id_{\cD}\longrightarrow F F'$ and
$\theta: F' F \longrightarrow\id_{\cC}$.

A tensor category $(\cC, \otimes, a, I, \ell, r)$ has a left duality if corresponding to each object $V$,
there exists an object $V^\vee$ and morphisms
\begin{eqnarray}\label{eq:dual-maps}
\Omega_V: V^\vee\otimes V\longrightarrow I, \quad \Upsilon_V: I\longrightarrow V\otimes V^\vee,
\end{eqnarray}
respectively called the evaluation and co-evaluation maps,
such that the following diagrams commute:
\begin{eqnarray}
\label{eq:dual-1}
\begin{aligned}
\xymatrix{
(V\otimes V^\vee)\otimes V\ar[rr]^{a_{V, V^\vee, V}}& &V\otimes(V^\vee\otimes V)\ar[d]_{\id_V\otimes \Omega_V}\\
I\otimes V\ar[u]_{\Upsilon_V\otimes\id_V}  &     & V\otimes I\ar[d]_{r_V}\\
V\ar[u]_{\ell_V^{-1}}\ar[rr]^{\id_V}&& V,
}
\end{aligned}
\end{eqnarray}
\begin{eqnarray}
\label{eq:dual-2}
\begin{aligned}
\xymatrix{
V^\vee\otimes (V\otimes V^\vee)\ar[rr]^{a^{-1}_{V^\vee, V, V^\vee}}& &(V^\vee\otimes V)\otimes V^\vee\ar[d]_{\Omega_V\otimes\id_{V^\vee}}\\
V^\vee\otimes I\ar[u]_{\Upsilon_V\otimes\id_V}  &     & I \otimes V^\vee\ar[d]_{\ell_{V^\vee}}\\
V^\vee\ar[u]_{r_{V^\vee}^{-1}}\ar[rr]^{\id_{V^\vee}}&& V^\vee.
}
\end{aligned}
\end{eqnarray}
%
Right duality can be defined similarly, which need not coincide with left duality.

Mac Lane's coherence theorem enables one to turn any given tensor category
$(\cC, \otimes, a, I, l, r)$  into a strict one (see, e.g.,  \cite[\S XI.5]{Kass}).
Let $\cS$ be the class of sequences of objects in $\cC$, which have finite lengths,
where the length of a sequence has the obvious meaning, namely,  the length of
the empty $\emptyset$ is $0$, and the length of a sequence of the form
 $S=(V_1, V_2, \dots, V_k)$ is $k$.
We define $*: \cS\times\cS\longrightarrow\cS$ for all sequences $S=(V_1, V_2, \dots, V_k)$
and $S'=(V_{k+1}, V_{k+2}, \dots, V_{k+n})$ by
\begin{eqnarray}\label{eq:star1}
S*S'=(V_1, \dots, V_k, V_{k+1}, \dots, V_{k+n}), \quad S*\emptyset = \emptyset*S=S.
\end{eqnarray}
Clearly $*$ is associative.
To each sequence in $\cS$, we assign an object $\cF(S)$ of $\cC$ defined by
$\cF(\emptyset)=I$, $\cF((V))=V$, and for $S=(V_1, V_2, \dots, V_k)$,
\begin{eqnarray}\label{eq:strictify1}
\begin{aligned}
\cF(S)=  (\dots(V_1\otimes V_2)\otimes \dots\otimes V_{k-1})\otimes V_k.
\end{aligned}
\end{eqnarray}

We denote by $\cC_s$ the category such that the objects are elements of $\cS$, and  morphisms are given by
$
\Hom_{\cC_s}(S, S') = \Hom_{\cC}(\cF(S), \cF(S)')
$
for all $S, S'$ in $\cS$.  The composition of morphisms and the identity morphisms are those in $\cC$.
Extend $\cF$ to a functor by letting
\begin{eqnarray}
\cF(f) = f, \quad \text{for any morphism $f$ in $\cC_s$}.
\end{eqnarray}
Then $\cF: \cC_s\longrightarrow \cC$ is fully faithful.
It is also essentially surjective as any object in $\cC$ is isomorphic
to the image under $\cF$ of some sequence in $\cS$.
Therefore
$
\cF: \cC_s\longrightarrow\cC
$
is an equivalence of categories.

Let us equip $\cC_s$ with the structure of a strict tensor category.
We take the empty sequence $\emptyset$ to be
the identity object, and $*$ to be the tensor product on objects.
To define a tensor product on morphisms, we first define the natural isomorphisms
\begin{eqnarray}
\varphi(S, S'): \cF(S)\otimes \cF(S)\longrightarrow \cF(S*S')
\end{eqnarray}
inductively on the lengths of sequences by
\begin{eqnarray}\label{eq:map-ast-tensor}
\begin{aligned}
&\varphi(\emptyset, S)=\ell_{\cF(S)}: I\otimes \cF(S)\longrightarrow \cF(S), \\
&\varphi(S, \emptyset)=r_{\cF(S)}: \cF(S)\otimes I\longrightarrow \cF(S), \\
&\varphi(S, (V))=\id_{\cF(S)\otimes V}: \cF(S)\otimes V\longrightarrow \cF(S*(V)), \\
&\varphi(S, S'*(V)): \cF(S)\otimes \cF(S'*(V))\longrightarrow \cF(S*S'*(V)), \\
&\varphi(S, S'*(V)) = (\varphi(S, S')\otimes\id_V)\circ a^{-1}_{\cF(S), \cF(S'), V}.
\end{aligned}
\end{eqnarray}

%
%
%
%
Let $\Hom_{\cC_s}(S*T, S'*T')=\Hom_{\cC}(\cF(S*T), \cF(S'*T'))$
for all sequences $S, S', T, T'$ in $\cS$.
Now for any  $f: \cF(S)\longrightarrow \cF(S')$ and $g: \cF(T)\longrightarrow \cF(T')$, we define
$
f*g\in \Hom_{\cC_s}(S*T, S'*T')
$
by
\begin{eqnarray}\label{eq:star2}
f*g  =\varphi(S', T') \circ(f\otimes g)\circ \varphi(S, T)^{-1}.
\end{eqnarray}
%
%

The following result is well-known, see e.g., \cite[Proposition XI.5.1]{Kass}.
\begin{theorem}\label{thm:tensor-equiv}\label{def:strict-cat}
The category $\cC_s$ quipped with the tensor product
$*:\cC_s\times\cC_s\longrightarrow\cC_s$ is a strict tensor category.
Furthermore, there is the tensor equivalence $(\cF, \id_I, \varphi): \cC_s\longrightarrow\cC$.
Call $\cC_s$ the strict tensor category associated with $\cC$.
\end{theorem}
\begin{notation}\label{rem:boxtimes}
If $\cC$ is a subcategory of the category of $\Z_2$-graded $\K$-vector spaces, we shall write
the associative tensor product of $\cC_s$ as $\boxtimes$.
Thus for any sequence $(V_1, V_2, \dots, V_k)$ of objects and morphisms $f, g$ in $\cC$,
\[
\begin{aligned}
V_1\boxtimes V_2\boxtimes\dots\boxtimes V_k = (\dots(V_1\otimes V_2)\otimes \dots\otimes V_{k-1})\otimes V_k,
\quad f\boxtimes g = f*g.
\end{aligned}
\]
\end{notation}

\subsection{Braided tensor categories}\label{sect:ribbon-cat}
The notion of braided tensor category is due to Joyal and Street \cite{JS}.
\begin{definition}
A  commutativity constraint $c$ for a tensor category $(\cC, \otimes, a, I, l, r)$
is a family of natural isomorphisms $c_{V, W}: V\otimes W\longrightarrow W\otimes V$
for all pairs of objects $V, W$ in $\cC$.  If $c$ satisfies the hexagon axiom (see \cite[\S XIII.1.1]{Kass}),
the tensor category is called a braided tensor category.
\end{definition}

Assume that both $\cC$ and $\cD$ are braided tensor categories.
A tensor functor $(\cF, \varphi_0, \varphi_2): \cC\longrightarrow \cD$ is braided if for any objects $V, W$ of $\cC$, the
following diagram commutes:
\begin{eqnarray}
\begin{aligned}
\xymatrix{
\cF(V)\otimes \cF(W)\ar[d]_{c_{\cF(V), \cF(W)}}\ar[rr]^{\varphi_2(V, W)}&& \cF(V\otimes W)\ar[d]^{\cF(c_{V, W})} \\
\cF(W)\otimes \cF(V)\ar[rr]^{\varphi_2(W, V)}&& \cF(W\otimes V).
}
\end{aligned}
\end{eqnarray}
A tensor equivalence satisfying this condition is a braided tensor equivalence.

Let $(\cC, \otimes, a, I, l, r, c)$ be a braided tensor category.
Let $(\cF, \id_I, \phi): (\cC_s, *, \emptyset) \longrightarrow
(\cC, \otimes, a, I, l, r)$
be the tensor equivalence from the associated
strict tensor category $\cC_s$ to $\cC$ given in Theorem \ref{thm:tensor-equiv}.
We define $c_{S, S'}\in \Hom_{\cC_s}(S*S', S'*S)$ for any $S, S'\in\cS$ by the
commutativity of the following diagram
\begin{eqnarray}\label{eq:braid-s}
\begin{aligned}
\xymatrix{
\cF(S)\otimes \cF(S')\ar[d]_{c_{\cF(S), \cF(S')}}\ar[rr]^{\varphi(S, S')}&& \cF(S*S')\ar[d]^{\cF(c_{S, S'})} \\
\cF(S)\otimes \cF(S')\ar[rr]^{\varphi(S', S)}&& \cF(S*S').
}
\end{aligned}
\end{eqnarray}
Since $\cF$ is the identity map on morphisms, we have $c_{S, S'}=\cF(c_{S, S'})$. Hence the commutative diagram gives
$
 c_{S, S'}= \varphi(S', S) c_{\cF(S), \cF(S')} \varphi(S, S')^{-1}
$
for any pair of objects $S, S'\in\cC_s$, yielding
a family of natural isomorphisms $c_{S, S'}$ in $\cC_s$.  They
give rise to a braiding for $\cC_s$, that is, the following relations hold
\[
\begin{aligned}
c_{X, Y*Z}= (\id_Y*c_{X, Z})\circ(c_{X, Y}*\id_Z),\\
c_{X*Y, Z} = (c_{X, Z}*\id_Y)\circ (\id_X*c_{Y, Z})
\end{aligned}
\]
for all $X, Y, Z$ in $\cS$.
It immediately follows from the commutativity of the diagram \eqref{eq:braid-s} that
\begin{theorem}\label{thm:equiv-braid-tensor-cat}
The tensor functor $(\cF, \id_A, \varphi): \cC_s\longrightarrow\cC$ of
Theorem \ref{thm:tensor-equiv} is
an equivalence of braided tensor categories.
\end{theorem}

%
%

Let $(\cC, \otimes, I, c)$ be a braided strict tensor category with left duality.
A twist for $\cC$ is a functorial  isomorphism $\theta_V: V\longrightarrow V$ defined for each object of $\cC$ such that
\[
\theta_{V\otimes W} = (\theta_V\otimes \theta_W)c_{W, V} c_{V, W}, \quad \theta_{V^\vee} = (\theta_V)^\vee
\]
for all objects $V$ and $W$, where $(\theta_V)^\vee: V^\vee\longrightarrow V^\vee$ is the left transpose of $\theta_V$.
Given any morphism $f: V\longrightarrow W$ in a strict tensor category with left duality,
its left transpose $f^\vee: W^\vee\longrightarrow V^\vee$ is defined by
\[
f^\vee = (\Omega_W\otimes\id_{V^\vee})(\id_{W^\vee}\otimes f\otimes \id_{V^\vee})(\id_{W^\vee}\otimes\Upsilon_V).
\]

\begin{definition}\label{ribbon-cat}
A braided strict tensor category with a left duality and twist
is called a ribbon category.
\end{definition}

For all objects $V$ and $W$ in a ribbon category, $c_{V^\vee, W}$ and $c_{W, V}$ are related by
\[
c_{V^\vee, W} = (\Omega_V\otimes\id_W\otimes\id_{V^\vee}) (\id_{V^\vee}\otimes c^{-1}_{W, V}
\otimes\id_{V^\vee})(\id_{V^\vee}\otimes\id_W\otimes\Upsilon_V).
\]
Define $\Upsilon'_V: I\longrightarrow V^\vee\otimes V$ and
$\Omega'_V: V\otimes V^\vee\longrightarrow I$ by
\begin{eqnarray}\label{eq:r-duality}
\begin{aligned}
&\Upsilon'_V= (\id_{V^\vee}\otimes\theta_V)c_{V, V^\vee} \Upsilon_V,\\
&\Omega'_V= \Omega_V c_{V, , V^\vee}(\id_V\otimes\theta_{V^\vee}).
\end{aligned}
\end{eqnarray}
Then ${^\vee}V, \Upsilon'_V, \Omega'_V$ for all objects $V$ define a right duality. Hence any ribbon category
automatically has a right duality.

\subsection{Categories of ribbon graphs}\label{append:ribbon-cat}

%
%
%
%

The category of directed ribbon graphs was introduced in \cite{RT90}
(also see \cite{FY, Sh}).
A ribbon is the square $[0,1]\times[0,1]$ smoothly embedded
in $\R^3$. The images of $[0,1]\times 0$ and  $[0,1]\times 1$ are
called the bases, and that of ${\frac{1}{2}}\times [0,1]$ is called
the core of the ribbon.  An annulus is the cylinder
$S^1\times [0,1]$
embedded in $\R^3$, and the image of $S^1\times  {\frac{1}{2}}$
under the embedding is called the core of the annulus. Ribbons and
annuli are oriented as surfaces and their cores are directed.

Let $k, \ l$ be nonnegative integers.  A ribbon $(k, \ell)$-graph is
an oriented and directed surface embedded in $\R^2\times [0, 1]\subset \R^3$,
which is decomposed into the union of ribbons and annuli without intersections.
The intersection of the surface with $\R^2 \times 0$
and $\R^2 \times 1$ are respectively
\begin{eqnarray}\label{eq:bases}
\begin{aligned}
&\left\{\left.\left[i-{\frac{1}{4}}, i+{\frac{1}{4}}\right]\times 0\times 0 \, \right| \,  1\le i\le k\right\}, \\
&\left\{\left.\left[j-{\frac{1}{4}}, j+{\frac{1}{4}}\right]\times 0\times 1 \, \right| \,
1\le j\le l\right\},
\end{aligned}
\end{eqnarray}
where the line segments are the (bottom and top) bases of the ribbons.
Ribbon graphs are defined up to isotopies of $\R^2\times [0, 1]$,
which preserve orientation and are constant on the boundary intervals \eqref{eq:bases}.

For simplicity, we will represent ribbons and annuli by their
directed cores.

There are two operations,  composition and juxtaposition,  of ribbon graphs.
Let $\Gamma$, $\Gamma_1$ and $\Gamma'$ respectively be $(k,\ l)$,  $(l,\ m)$ and
and $(k',\ l')$ ribbon graphs. The composition $\Gamma_1\circ\Gamma$ is defined
in the following way: shift  $\Gamma_1$ by the vector $(0,0,1)$
into $\R^2\times [1,2]$, glue the bottom end of
$\Gamma_1$ to the top end of  $\Gamma$ so that
the orientations and directions of the cores of the ribbons glued together match,
then reduce vertically the size of the resultant picture  by a factor of
$2$, leading to a $(k, \ m)$ ribbon graph.
The juxtaposition  $\Gamma \otimes  \Gamma'$ is to position
$\Gamma'$ on the right of  $\Gamma$, leading to a
$(k+k',\ l+l')$ ribbon graph.

Introduce the set ${\mathcal N}$ consisting of sequences
$(\varepsilon_1, \varepsilon_2, ..., \varepsilon_k)$, where  $k\in {\Z}_+$
and all $\varepsilon_i\in \{+, \ - \}$.
Each ribbon $(k, \ell)$-graph $\Gamma$ is associated with two elements of ${\mathcal N}$,
the source $s(\Gamma)=(\varepsilon_1,   \varepsilon_2, ...,
\varepsilon_k)$ and target $t(\Gamma)=(\varepsilon'_1,  \varepsilon'_2, ...,
 \varepsilon'_\ell)$,
in the following way.
If a ribbon of $\Gamma$ has a base $[p-{\frac{1}{4}},
p+{\frac{1}{4}}]\times 0\times 0$
(resp.  $[r-{\frac{1}{4}}, r+{\frac{1}{4}}]\times 0\times 1$), then
$\varepsilon_p=+$ (resp. $\varepsilon'_r=-)$
if its core is directed towards this base, and $\varepsilon_p=-$
(resp. $\varepsilon'_r=+)$ otherwise.
Fix a commutative ring $\K$. Given any two elements $\eta, \eta'\in{\mathcal N}$,
we denote by $\Hom(\eta, \eta')$ the free $\K$-module with a basis consisting of
isotopy classes of ribbon graphs $\Gamma$ such that $s(\Gamma)=\eta$
and $t(\Gamma)=\eta'$.

\begin{definition}\label{def:ribbon-direct}
The category $\cH(\K)$ of ribbon graphs is the $\K$-linear category, which has the set of
objects $\mathcal N$, and sets of morphisms $\Hom(\eta, \eta')$  as defined above
for any $\eta, \eta'\in{\mathcal N}$. The composition of morphisms
is given by the composition of ribbon graphs.
\end{definition}

It has the structure of a ribbon category.
The associative tensor product $\otimes: \cH(\K)\times\cH(\K)
\rightarrow\cH(\K)$ is defined so that the tensor product of objects
$\eta$ and $\eta'$ is to join $\eta'$ to the right of $\eta$
to form one sequence, and the tensor product of morphism is given by
the juxtaposition of ribbon graphs defined earlier.
This is a braided tensor category with the braiding given by

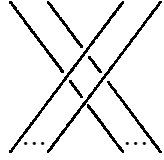
\begin{figure}[h]
\begin{center}
\setlength{\unitlength}{.5mm}
\begin{picture}(80, 35)(-10, 15)
\qbezier(10, 10)(25, 30)(40, 50)
\qbezier(20, 10)(35, 30)(50, 50)
\qbezier(40, 10)(25, 30)(34, 18)
\qbezier(10, 50)(25, 30)(23.8, 31.6)
\qbezier(50, 10)(47, 14)(36, 28.7)
\qbezier(20, 50)(26, 42)(29, 38)
\qbezier(31, 35.3)(32, 34)(34, 31.3)
\qbezier(26, 28.7)(25, 30)(29, 24.7)
\put(13, 12){...}
\put(40, 12){...}
\end{picture}
\end{center}
\caption{Braiding}
\label{fig:braiding}
\end{figure}
\noindent
where the directions of the cores of ribbons
are understood to be consistent with sources and targets of the ribbon graphs.

The braided strict tensor category $\cH(\K)$ has the structure of a ribbon category,
with the duality $^\vee: \eta=(\varepsilon_1, \varepsilon_2, \dots, \varepsilon_k)
\mapsto \eta^\vee=(\varepsilon_k^\vee,   \varepsilon_{k-1}^\vee, ...,\varepsilon_1^\vee)$
on any object, where $+^\vee =-$ and $-^\vee =+$.
The left duality maps \eqref{eq:dual-maps}
are respectively depicted by the first two
ribbon graphs in Figure \ref{fig:duality-twist},
and the twist by the last ribbon graph  in Figure \ref{fig:duality-twist},
where again the directions of the cores of ribbons
are understood to be consistent with sources and targets of the ribbon graphs.

\begin{figure}[h]
\begin{center}
\setlength{\unitlength}{.5mm}
\begin{picture}(180, 50)(90, 0)
\qbezier(90, 50)(110, 5)(130, 50)
\qbezier(100, 50)(110, 20)(120, 50)
\put(92, 47){...}

\qbezier(150, 10)(170, 55)(190, 10)
\qbezier(160, 10)(170, 40)(180, 10)
\put(153, 12){...}

\put(120, -2){Duality}

\put(231, 47){...}
\qbezier(230, 50)(250,0)(258, 30)
\qbezier(240, 50)(250, 10)(255, 30)

\qbezier(255, 30) (252, 45) (247, 30)
\qbezier(258, 30) (253, 52) (246, 35)

\qbezier(244, 33)(243, 31)(241.5, 28)
\qbezier(246, 28)(245, 26)(244, 24)

\qbezier(240, 25)(239, 24)(230, 10)
\qbezier(243, 22)(242, 20)(238, 10)

\put(236, -2){Twist}

\end{picture}
\end{center}
\caption{Duality and twist}
\label{fig:duality-twist}
\end{figure}
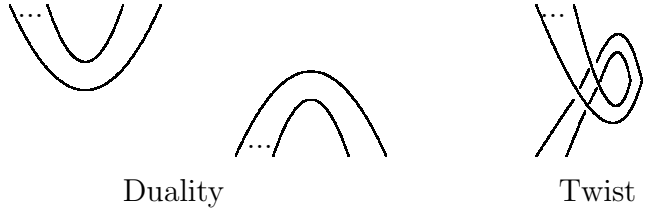

The category $\cH(\K)$ can be presented \cite{RT90, Turaev}
in terms of  the generators depicted in Figure 1, and the relations
given in, e.g., \cite{RT90, Turaev}.

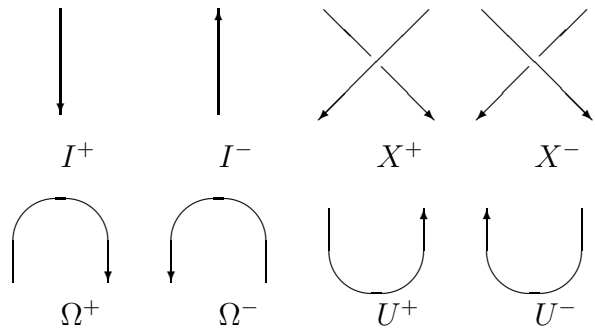
\begin{figure}[h]
\begin{center}
\setlength{\unitlength}{.7mm}
\begin{picture}(140, 65)
\put(10,0){$\Omega^+$}
\put(40,0){$\Omega^-$}
\put(70,0){$U^+$}
\put(100,0){$U^-$}
\put(10, 30){$I^+$}
\put(40,30){$I^-$}
\put(70,30){$X^+$}
\put(100,30){$X^-$}

\put(10,60){\vector(0,-1){20}}
\put(40,40){\vector(0,1){20}}

\put(69,51){\line(-1,1){9}}
\put(71,49){\vector(1,-1){10}}
\put(80,60){\vector(-1,-1){21}}

\put(90,60){\vector(1,-1){21}}
\put(101,51){\line(1,1){9}}
\put(99,49){\vector(-1,-1){10}}

\put(10,15){\oval(18,18)[t]}
\put(1,15){\line(0,-1){7}}
\put(19,15){\vector(0,-1){7}}

\put(40,15){\oval(18,18)[t]}
\put(31,15){\vector(0,-1){7}}
\put(49,15){\line(0,-1){7}}

\put(70,15) {\oval(18,18)[b]}
\put(61,15){\line(0,1){7}}
\put(79,15){\vector(0,1){7}}

\put(100,15){\oval(18,18)[b]}
\put(91,15){\vector(0,1){7}}
\put(109,15){\line(0,1){7}}
\end{picture}
\end{center}
\caption{Generators of ribbon graphs}
\label{fig:generators}
\end{figure}

The following theorem is a special case of the tensor functor \cite{RT90, Sh, Turaev}
from the category of coloured ribbon graphs to any given ribbon category.
The general result may be found in, e.g., \cite[Theorem 2.5]{Turaev}.
\begin{theorem} \label{thm:RT}
Let $(\cC, \otimes, I, c)$ be a $\K$-linear ribbon category with twist $\theta$ and left duality
$(^\vee, \Upsilon, \Omega)$.
Given any object $V$ in $\cC$, there  exists a unique braided tensor functor
$\cF: \cH(\K)\longrightarrow \cC$, which preserves left duality and twist, such that
\begin{eqnarray}\label{eq:functor}
\begin{aligned}
&\cF(\emptyset)=I, \quad \cF(+)=V, \quad \cF(-)=V^\vee;\\
& \cF(I^+)=\id_V, 	\quad F(I^-)=\id_{V^\vee}, \\
&\cF(X^+)=c_{V, V}:  V\otimes V \rightarrow V \otimes V,\\
&\cF(X^-)=c_{V, V}^{-1}:  V \otimes V \rightarrow V\otimes V,\\
& \cF(\Omega^+)= \Omega_V:  V^\vee\otimes  V\rightarrow   I, \\
&\cF(\Omega^-)= \Omega'_V: V\otimes  V^\vee\rightarrow  I, \\
&\cF(U^+)=\Upsilon_V:  I \rightarrow V\otimes  V^\vee,\\
&\cF(U^-)= \Upsilon'_V:  I \rightarrow V^\vee\otimes  V,
\end{aligned}
\end{eqnarray}
where $\Upsilon'$ and $\Omega'$ are defined by \eqref{eq:r-duality}.
\end{theorem}

%
%
%
%

Non-directed ribbon graphs are
oriented surfaces embedded in $\R^2\times[0, 1]$ which can be decomposed into
ribbons and annuli with the usual properties. The only
difference is that the cores of ribbons and annuli in non-directed ribbon graphs are not directed.
Composition of non-directed ribbon graphs is defined in the usual way but without the
requirement on the directions of cores of ribbons.
We still represent non-directed ribbons and annuli by their cores.

\begin{definition} \label{def:ribbon-nondirect}
The category $\cH'(\K)$ of non-directed ribbon graphs
is the $\K$-linear category such that the objects are $0, 1, 2, \dots$,
and for each pair $(k, l)$ of objects,  the set $\Hom(k, \ell)$ of morphisms is the free $\K$-module
with a basis consisting of non-directed ribbon $(k, \ell)$-graphs. The composition of morphisms
is given by the composition of non-directed ribbon graphs.
\end{definition}

Then $\cH'(\K)$ is a braided strict tensor category with tensor product $\otimes$ being the juxtaposition
of non-directed ribbon graphs for morphisms, and $k\otimes \ell=k+\ell$ for objects.
The braiding is given by the non-directed ribbon graphs of the form Figure \ref{fig:braiding}. Furthermore,
$\cH'(\K)$ has the structure of a ribbon category. The duality for objects is $^\vee: k\mapsto k$,
the left duality maps defined by \eqref{eq:dual-maps} are
respectively given by non-directed ribbon graphs of the form shown in the the first two
diagrams in Figure \ref{fig:duality-twist},
and the twist by the non-directed ribbon graph shown in the last diagram in Figure \ref{fig:duality-twist}.

Note that $\cH'(\K)$ is an example of  the braided tensor categories with only self-dual objects
studied in \cite{Sl}.

The ribbon category $\cH'(\K)$ can be presented in terms of the generators depicted
in Figure \ref{fig:generators-nondirect}.
The full set of relations can be obtained from those given in \cite{RT90, Turaev} for
the category of ribbon graphs by ignoring the directions of the cores of ribbons.
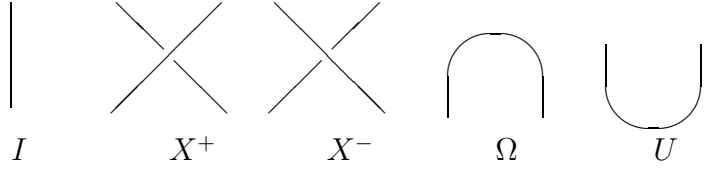
\begin{figure}[h]
\begin{center}
\setlength{\unitlength}{.7mm}
\begin{picture}(60, 40)(45, 30)
\put(40,30){$I$}
\put(70,30){$X^+$}
\put(100,30){$X^-$}

\put(40,40){\line(0,1){20}}

\put(69,51){\line(-1,1){9}}
\put(71,49){\line(1,-1){10}}
\put(80,60){\line(-1,-1){21}}

\put(90,60){\line(1,-1){21}}
\put(101,51){\line(1,1){9}}
\put(99,49){\line(-1,-1){10}}
\end{picture}
\begin{picture}(60, 40)(15, 0)
\put(40,0){$\Omega$}
\put(70,0){$U$}

\put(40,15){\oval(18,18)[t]}
\put(31,15){\line(0,-1){7}}
\put(49,15){\line(0,-1){7}}

\put(70,15) {\oval(18,18)[b]}
\put(61,15){\line(0,1){7}}
\put(79,15){\line(0,1){7}}
\end{picture}
\end{center}
\caption{Generators of non-directed ribbon graphs}
\label{fig:generators-nondirect}
\end{figure}

The proof of the following result is similar to that of Theorem \ref{thm:RT}.

\begin{theorem} \label{thm:RT-1}
Let $(\cC, \otimes, I, c)$ be a $\K$-linear ribbon category with twist $\theta$ and left duality
$(^\vee, \Upsilon, \Omega)$.
Given any object $V$ in $\cC$, which is self-dual in the sense that $V^\vee \cong V$,
there  exists a unique braided tensor functor
$\cF: \cH'(\K)\longrightarrow \cC$, which preserves left duality and twist, such that
\begin{eqnarray}\label{eq:functor-1}
\begin{aligned}
&\cF(0)=I, \quad \cF(1)=V, \quad \cF(I)=\id_V, 	 \\
&\cF(X^+)=c_{V, V}:  V\otimes V \rightarrow V \otimes V,\\
&\cF(X^-)=c_{V, V}^{-1}:  V \otimes V \rightarrow V\otimes V,\\
& \cF(\Omega)= \Omega_V:  V\otimes  V\rightarrow   I, \\
& \cF(U)=\Upsilon_V:  I \rightarrow V\otimes  V.
\end{aligned}
\end{eqnarray}
\end{theorem}

The theorem was already implicitly present in the study of braided tensor categories of type BCD in \cite{TW}.

\section{Braided quasi Hopf superalgebras}\label{sect:quasi-Hopf} \label{append:quasi-Hopf}

We summarise some elementary facts on braided quasi Hopf superalgebras, 
which are used in the paper.
The notion of quasi Hopf algebras was introduced by Drinfeld in the seminal papers \cite{D90, D92}.
Generalisation to the super context was treated in \cite[\S II]{Z02} (see also \cite{G06}).

\subsection{Braided quasi Hopf superalgebras}\label{sect:bqHsa}
Fix any commutative ring $\K$ with identity.
Let $H$ be an associative $\K$-superalgebra
equipped with
even superalgebra homomorphisms
$\Delta: H\rightarrow H\otimes H$ and $\epsilon: H\rightarrow \K$,
respectively called the co-multiplication and co-unit,
such that $(\epsilon\otimes\id)\Delta= \id$ and $(\id\otimes\epsilon)\Delta= \id.$
The superalgebra $H$ is called  a quasi bi-superalgebra if
there exists an invertible even element
$\Phi\in H\otimes H\otimes H$, called the associator,
satisfying the following relations
\begin{eqnarray}\label{eq:associator}
\begin{aligned}
&(\id\otimes \Delta)\Delta(x) = \Phi (\Delta\otimes\id)\Delta(x)\Phi^{-1},  \quad \forall x\in H,\\
&(\id\otimes\epsilon\otimes\id)\Phi=1, \\
&(\id\otimes\id\otimes\Delta)(\Phi) (\Delta\otimes\id\otimes\id)(\Phi)
=\Phi_{2 3 4} (\id\otimes\Delta\otimes \id)(\Phi) \Phi_{1 2 3}.
\end{aligned}
\end{eqnarray}

Let us write
$
\Phi=\sum_{t} X_t\otimes Y_t\otimes Z_t$ and $\Phi^{-1}=\sum_{t}\overline{X}_t \otimes \overline{Y}_t \otimes
\overline{Z}_t.
$
We also write $\Delta(x)=\sum_{(x)}x_{(1)}\otimes x_{(2)}$ for any $x\in H$.
The quasi bi-superalgebra is a quasi Hopf superalgebra if there exist invertible
even elements $\alpha, \beta\in H$, and an algebra anti-automorphism
$S$, such that
\begin{eqnarray}\label{eq:antipode}
\begin{aligned}
&\sum_{(x)} S(x_{(1)})\alpha x_{(2)} = \epsilon(x) \alpha,\quad
\sum_{(x)} x_{(1)} \beta S(x_{(2)}) = \epsilon(x) \beta,
\quad \forall x\in H, \\
&\sum_{t} X_t\beta S(Y_t) \alpha Z_t =
\sum_{t} S(\overline{X}_t) \alpha \overline{Y}_t \beta S(\overline{Z}_t)= 1.
\end{aligned}
\end{eqnarray}

It immediately follows from the definition that
$
\epsilon(\alpha)\epsilon(\beta)=1
$
and
$
\epsilon S = \epsilon.
$
Call $(S, \alpha, \beta)$ an antipode triple (or simply an antipode).
For any invertible element $g\in H$, the conditions (\ref{eq:antipode})
are also satisfied by $\tilde{S}, \tilde{\alpha},  \tilde{\beta}$ with
\begin{eqnarray}\label{eq:S-trans}
\tilde{\alpha}= g \alpha,&
\tilde{\beta}=\beta g^{-1},  & \tilde{S}(a)=g S(a) g^{-1},\ \forall a\in H.
\end{eqnarray}
We call $g$ an antipodal transformation.
Any two antipode triples of a quasi Hopf superalgebra
are related to each other by a unique antipodal transformation \cite{D90, D92}.

A quasi Hopf superalgebra is called braided if there exists an
invertible even  element $R\in H\otimes H$, called the universal
$R$-matrix, which
satisfies the following relations
\begin{eqnarray}\label{eq:braid}
\begin{aligned}
R \Delta(x) &= \Delta^{op}(x) R, \quad \forall x\in H, \\
(\id\otimes \Delta)R &=(\Phi_{2 3 1})^{-1} R_{ 13} \Phi_{2 1 3} R_{ 1
2} (\Phi_{1 2 3})^{-1}, \\
(\Delta\otimes\id)R&=\Phi_{3 1 2} R_{ 13} (\Phi_{1 3 2})^{-1}
R_{2 3} \Phi_{1 2 3},
\end{aligned}
\end{eqnarray}
where $\Delta^{op}$ is the opposite co-multiplication.
It immediately follows from the definition that the universal
$R$-matrix satisfies the following generalized Yang-Baxter equation
\begin{eqnarray}
R_{1 2} \Phi_{3 1 2} R_{ 13} (\Phi_{1 3 2})^{-1} R_{2 3} \Phi_{1 2 3}
=\Phi_{3  2 1} R_{2 3} (\Phi_{2 3 1})^{-1} R_{ 13} \Phi_{2 1 3} R_{1 2}.
\end{eqnarray}

Given a braided quasi Hopf superalgebra, we
write the universal $R$-matrix as $R=\sum_{r} a_r\otimes b_r$, and let
\begin{eqnarray}\label{eq:u}
u=\sum_{r, t} (-1)^{[\overline{X}_t]}
S(b_r \overline{Y}_t\beta S(\overline{Z}_t))\alpha a_r \overline{X}_t,
\end{eqnarray}
where $[\overline{X}_t]=0$ or $1$ is the parity of $\overline{X}_t$.
Then $u$ is invertible, $S^2(x)=u x u^{-1}$ for all $x\in H$,
and  $u S(u)=S(u) u$ belongs to the centre of $H$.
If there exists an even element $v\in H$ such that
\begin{eqnarray}\label{eq:v}
v^2 =u S(u),
\end{eqnarray}
we call $(H, \Delta, \epsilon, \Phi, S, \alpha, \beta, R, v)$
a ribbon quasi Hopf superalgebra with the ribbon element $v$.
We have
$
\Delta(v)=(v\otimes v) (R_{2 1} R)^{-1}.
$

%
%
%
%

Given a ribbon quasi Hopf superalgebra
$(H, \Delta, \epsilon, \Phi, S,  \alpha, \beta, R, v)$,  we take any
invertible even element $F\in H\otimes H$ satisfying
$(\epsilon\otimes \id)F = (\id\otimes\epsilon)F =1,$
and let
\begin{eqnarray}\label{eq:gauge}
\begin{aligned}
&\Delta_F:  H\rightarrow H\otimes H, \quad
    x\mapsto F \Delta(x) F^{-1}, \\
&\Phi_F =F_{2 3} (\id\otimes \Delta) (F) \Phi (\Delta\otimes\id)
(F^{-1}) F_{1 2}^{-1}, \\
&R_F= F_{2 1} R F^{-1}, \\
&\alpha_F=m(S\otimes\id)((1\otimes\alpha)F^{-1}),\\
&\beta_F=m(\id\otimes S)(F(\beta\otimes 1)),
\end{aligned}
\end{eqnarray}
where $m$ denotes the multiplication of $H$.
Then
$(H, \Delta_F, \epsilon, \Phi_F, S, \alpha_F, \beta_F, R_F, v)$ is another ribbon quasi Hopf superalgebra.
Call $F$ a gauge transformation on $H$.
\begin{remark}\label{rem:inv-ribbon-transform}
Note that the anti-homomorphism $S$ is not affected by gauge transformations,
and the ribbon element $v$ is not affected by gauge and antipodal transformations.
\end{remark}

Given ribbon quasi Hopf superalgebras,
\begin{eqnarray*}
(H, \Delta^{(H)}, \epsilon^{(H)},
\Phi^{(H)}, S^{(H)}, \alpha^{(H)}, \beta^{(H)},  R^{(H)},  v^{(H)})\\
(B, \Delta^{(B)}, \epsilon^{(B)}, \Phi^{(B)},
S^{(B)}, \alpha^{(B)}, \beta^{(B)},  R^{(B)},  v^{(B)}),
\end{eqnarray*}
and assume that we have an even superalgebra homomorphism
$f: H\rightarrow B$ and
a gauge transformation $F\in B\otimes B$ on $B$ such that
\begin{eqnarray*}
\begin{aligned}
&(f\otimes f)\Delta^{(H)}= F \cdot ( \Delta^{(B)} f )\cdot F^{-1},\\
&(f\otimes f\otimes f)(\Phi^{(H)})
=\left(F_{2 3}(\id\otimes \Delta^{(B)})(F) \right)^{-1} \Phi^{(B)} F_{1 2}(\Delta^{(B)}\otimes \id)F,\\
&(f\otimes f)R^{(H)}= F_{2 1} R^{(B)} F^{-1}, \quad v^{(B)}=f(v^{(H)}).
\end{aligned}
\end{eqnarray*}
Then there exists a unique invertible $g\in B$ such that
\begin{eqnarray*}
\begin{aligned}
&(f \circ S^{(H)}) (x) =g (S^{(B)}\circ f) (x)  g^{-1}, \quad \forall x\in H, \\
&f(\alpha^{(H)})= g \alpha^{(B)}_F, \quad
f(\beta^{(H)})= \beta^{(B)}_F g^{-1} .
\end{aligned}
\end{eqnarray*}
\begin{definition}\label{def:equiv-ribbon-quasi-Hopf}
Call the triple $(f, F, g)$ a homomorphism of ribbon quasi Hopf superalgebras.
If $f$ is a superalgebra isomorphism,  it is called an equivalence.
\end{definition}

\subsection{Representations of braided quasi Hopf superalgebras}\label{sect:quasi-module-cat}
Given a braided quasi Hopf superalgebra
$(H, \Delta, \epsilon, \Phi, S, \alpha, \beta, R)$,  we 
denote by $H$-Mod the the category of $\Z_2$-graded left $H$-modules equipped with the
tensor product functor $\otimes_\K$. Note that all morphisms in $H$-Mod are even.
Then $H$-Mod is a braided tensor category with the following structure.
\begin{enumerate}
\item The unit object is $\K$, and the left and right unit constraints are the identity.

\item The associativity constraint
\[
\begin{aligned}
&a_{U, V, W}:  (U\otimes V) \otimes W \longrightarrow U\otimes (V \otimes W), \\
&a_{U, V, W}((u\otimes v) \otimes w) = u\otimes (v \otimes w), \quad \forall u\in U, \ v\in V, w\in W
\end{aligned}
\]
is given for all $ u\in U,\  v\in V, \ w\in W$  by
\[
a_{U, V, W}((u\otimes v) \otimes w) =\Phi_{U, V, W}((v\otimes v) \otimes w) = \Phi((u\otimes v) \otimes w).
\]
\item The braiding $c$ is defined as follows.
Let $\check{R}_{V, W}: V\otimes W\longrightarrow W\otimes V$
for any objects $V, W$ of $H$-Mod  be defined by the composition
\[
V\otimes W\stackrel{R}{\longrightarrow} V\otimes W\stackrel{\tau_{V, W}}{\longrightarrow} W\otimes V,
\]
where $\tau_{V, W}$ is the family of natural isomorphisms
\begin{eqnarray}\label{eq:tau}
\tau_{V, W}: V\otimes W\longrightarrow W\otimes V, \quad v\otimes w\mapsto (-1)^{[v][w]} w\otimes v.
\end{eqnarray}
Then the braiding $c_{V, W}: V\otimes W\longrightarrow W\otimes V$ is given by
$
c_{V, W}=\check{R}_{V, W}.
$
\end{enumerate}

Given a gauge transformation $F$ on $H$,
denote $H_F=(H, \Delta_F, \epsilon, \Phi_F, S, \alpha_F, \beta_F, R_F)$.
For any objects $V, W\in \HMod$,
define
\begin{eqnarray}\label{eq:phi-2-F}
\varphi_2^F(V, W)(v\otimes w) = F^{-1}(v\otimes w), \quad \forall v\in V, \ w\in W.
\end{eqnarray}
\begin{lemma}\label{lem:gauge-cat}
With notation as above,
$(\id, \id, \varphi_2^F): \HMod\longrightarrow\text{$H_F$-Mod}$ is a braided tensor equivalence.
\end{lemma}

Let
$(\alpha, F, g): (H, \Delta, \epsilon, \Phi, S, \alpha, \beta, R)
\longrightarrow (H', \Delta', \epsilon', \Phi', S', \alpha', \beta', R')$
be an equivalence of braided quasi bi-superalgebras with
a gauge transformation $F$ on $H'$ and isomorphism
$\alpha: H\longrightarrow H'_F$ of quasi bi-superalgebras.
Let $\alpha^*: H'_F\text{-Mod}\longrightarrow \HMod$ be
the equivalence of categories induced by $\alpha$.

\begin{theorem}\label{thm:braided-gauge}
The braided tensor functor
$(\alpha^*, \id, \varphi_2^F): H'\text{-Mod}\longrightarrow \HMod$ is a braided tensor equivalence.
\end{theorem}
\begin{proof}
We have the obvious tensor equivalence
$(\alpha^*, \id, \id): H'_F\text{-Mod}\longrightarrow \HMod$. It then follows from Lemma \ref{lem:gauge-cat} that
the composition
\[
\xymatrix{
H'\text{-Mod}\ar[rr]^{(\id, \id, \varphi^F_2)}&&H'_F\text{-Mod}\ar[rr]^{(\alpha^*, \id, \id)}&& \HMod
}
\]
gives the desired braided tensor equivalence.
\end{proof}

\subsection{Duality and twist}\label{append:sect-du-tw}
%
%

%
%
%
%

Let $H\text{-Mod}_f$ denote the full subcategory of $H\text{-Mod}$ with objects
which are $\K$-free of finite ranks. Clearly $H\text{-Mod}_f$ is a braided tensor category.
For any object $M$ of $H\text{-Mod}_f$, we denote by
$M^*:=\Hom_\K(M, \K)$ its dual, which is a free $\K$-module of the same rank as that of $M$.
The antipode enables us to turn $M^*$ into a left $H$-module with the action
$H\otimes M^*\rightarrow M^*$, $x\otimes v^*\mapsto x v^*$,  defined by
\begin{eqnarray}\label{eq:dual}
x v^*(w)= v^*\left((-1)^{[x][v^*]} S(x) w\right), \quad \text{for all $w\in M$}.
\end{eqnarray}
Let $\{b_i\mid i=1, \dots, r\}$
be a homogeneous $\K$-basis for $M$, and $\{b^*_i\mid i=1, \dots, r\}$ a $\K$-basis for $M^*$
such that $b^*_i(b_j)=\delta_{i j}$.
We have the following functorial homomorphisms in $H\text{-Mod}_f$.
\begin{eqnarray}\label{eq:ev+exp}
\begin{aligned}
&\Omega_M: M^*\otimes M\longrightarrow \K,
\quad  v^*\otimes w\mapsto v^*(\alpha w), \\
&\Upsilon_M: \K \longrightarrow M\otimes M^*,
\quad 1\mapsto \sum_i \beta b_i \otimes b^*_i.
\end{aligned}
\end{eqnarray}
Equations \eqref{eq:dual} and \eqref{eq:ev+exp} define left duality (cf. \eqref{eq:dual-maps})
for the braided tensor category $H\text{-Mod}_f$. Note the appearance of $\alpha$ and $\beta$
in the above maps. A word of warning is that the usual
dual space pairing is not an $H$-module homomorphism if $\alpha\ne 1$.

\begin{notation}\label{notation}
Let $H\text{-mod}$ be the strict tensor category
associated with $H\text{-Mod}_f$.
\end{notation}
Then $H\text{-mod}$ is a braided strict tensor category (see Theorems \ref{def:strict-cat} and  \ref{thm:equiv-braid-tensor-cat}).
If $H$ is a ribbon quasi Hopf superalgebra with the ribbon element $v$, then
\begin{eqnarray}\label{eq:theta}
\theta_M=v: M\longrightarrow M
\end{eqnarray}
defines a twist (cf. Section \ref{sect:ribbon-cat}) for $H\text{-mod}$.
Therefore, we have the following result.
\begin{theorem}\label{thm:key}
Let $(H, \Delta, \epsilon, \Phi, S, \alpha, \beta, R, v)$
be a ribbon quasi Hopf superalgebra. Then $H\text{-mod}$
is a ribbon category, with the left duality defined by \eqref{eq:ev+exp},
and the twist defined by
\eqref{eq:theta}.
\end{theorem}

%
%

Let  $F$ be a gauge transformation which takes
the ribbon quasi Hopf superalgebra $(H, \Delta, \epsilon, \Phi, S, \alpha, \beta, R, v)$
to $(H_F=H, \Delta_F, \epsilon, \Phi_F, S, \alpha_F, \beta_F, R_F, v)$,
where $\Delta_F$, $\Phi_F$, $R_F$ etc. are defined by \eqref{eq:gauge}.
Lemma \ref{lem:gauge-cat} gives a braided tensor equivalence
$(\id, \id, \varphi_2^F): \text{$H$-mod}\longrightarrow \text{$H_F$-mod}$
with $\varphi_2^F$ defined by \eqref{eq:phi-2-F}.

\begin{lemma}\label{lem:gauge-duality}
The braided tensor equivalence of Lemma \ref{lem:gauge-cat} induced
by any gauge transformation preserves duality and twist.
\end{lemma}
In view of Remark \ref{rem:inv-ribbon-transform},
the proof of this boils down to showing that the following diagrams commute,
\[
\xymatrix{
V^*\otimes V\ar[rd]_{\Omega_V^F}\ar[rr]^{\varphi_2^F(V^*, V)}&& V^*\otimes V\ar[ld]^{\Omega_V}\\
& \K &,
}
\quad \quad
\xymatrix{ &\K\ar[ld]_{\Upsilon_V^F}\ar[rd]^{\Upsilon_V}&\\
V\otimes V^*\ar[rr]^{\varphi_2^F(V^*, V)}&& V\otimes V^*,
}
\]
where ${\Omega_V^F}$ and $\Upsilon_V^F$ are defined by \eqref{eq:ev+exp}
 but using $\alpha_F$ and $\beta_F$. But this immediately follows from the definition of $\alpha_F$ and $\beta_F$.

\subsubsection{Transformations of antipodes}
Recall from \eqref{eq:S-trans} the freedom in the definition of the antipode triple.
Assume that there exists a unit $g$ in $H$ transforming
a given antipode triple $(S, \alpha, \beta)$ to another $(\tilde S, \tilde\alpha, \tilde\beta)$.
Let $\widetilde{H}\text{-mod}$ denote the ribbon category of $H$-modules with the left duality
defined with respect to $(\tilde S, \tilde\alpha, \tilde\beta)$.

If $M$ is an object in $H\text{-mod}$, the action of $H$ on $M^*$ is defined by \eqref{eq:dual}. Now regard
$M^*$ as an object of $\widetilde{H}\text{-mod}$, and denote the $H$-action on $V^*$ by $\bullet$. Then
for any $x\in H$ and $v^*\in M$, we have
$
x\bullet v^* = S^{-1}(g^{-1}) x S^{-1}(g) v^*.
$

For any object $M$ in $H$-mod, we let $M^+=M$ and $M^-=M^*$.
Given finitely $\K$-generated $H$-modules $M_1, M_2, \dots, M_n$, and any sequence
$(\varepsilon_1, \varepsilon_2, \dots, \varepsilon_n)$ with $\varepsilon_i\in\{+, -\}$, we form the sequence
${\bf M}^\varepsilon=(M_1^{\varepsilon_1}, ...,
M_n^{\varepsilon_n})$ of $H$-modules and let
(see Notation \ref{rem:boxtimes})
\begin{eqnarray}\label{eq:ordered-tensor}
\boxtimes{\bf M}^\varepsilon &:=&M_1^{\varepsilon_1}\boxtimes M_2^{\varepsilon_2} \boxtimes ... \boxtimes M_n^{\varepsilon_n}\\
&=&\left(\dots \left(\left(M_1^{\varepsilon_1}\otimes M_2^{\varepsilon_2}\right)\otimes M_3^{\varepsilon_3}\right)
\otimes \dots\right)\otimes  M_n^{\varepsilon_n}.\nonumber
\end{eqnarray}
Introduce the $\K$-linear automorphism of $\boxtimes{\bf M}^\varepsilon$
\begin{eqnarray*}
g_{{\bf M}^\varepsilon}=\left( ...  \left( \left(  S^{-1}(g^{\theta(\varepsilon_1)})
\otimes S^{-1}(g^{\theta(\varepsilon_2)}) \right)\otimes
S^{-1}(g^{\theta(\varepsilon_3)}) \right)\otimes...\right)\otimes
S^{-1}(g^{\theta(\varepsilon_n)}),
\end{eqnarray*}
with $\theta(+)=0$ and $\theta(-)=-1$,
where the action of $g$ on $M_i^*$ is defined by using $S$.

We now define a functor
$
\cG: H\text{-mod}\longrightarrow \widetilde{H}\text{-mod},
$
which restricts to the identity on objects, and for any morphism $f: \boxtimes{\bf M}^\varepsilon\longrightarrow
	\boxtimes{\bf W}^{\varepsilon'}$  in $H\text{-mod}$,
\begin{eqnarray}\label{eq:g-functor}
\begin{aligned}
 \cG(f)=g_{{\bf W}^{\varepsilon'}}\circ f\circ g_{{\bf M}^\varepsilon}^{-1}:
\boxtimes{\bf M}^\varepsilon\longrightarrow \boxtimes{\bf W}^{\varepsilon'} \ \ \text{in $\widetilde{H}\text{-mod}$.}
\end{aligned}
\end{eqnarray}
To prove that this indeed defines a functor, we need to show that $\cG(f)$ commutes
with the action of $H$ for all morphisms $f$.  Consider, for example,
any morphism $f: M_1\otimes M_2^* \longrightarrow W_1^*\otimes W_2$.  Then
\[
\cG(f) (v_1\otimes v_2^*) = (S^{-1}(g^{-1})\otimes \id)f(v_1\otimes S^{-1}(g)v_2^*),
    \quad  \forall v_1\otimes v_2^*\in M_1\otimes M_2^*.
\]
Now for all $x\in H$, a lengthy computation yields
\[
\begin{aligned}
x\bullet \cG(f) (v_1\otimes v_2^*)
=\cG(f) (x\bullet(v_1\otimes v_2^*)).
\end{aligned}
\]
Hence $\cG(f)$ commutes with the action of $H$. The general case can be shown similarly.

We have the tensor functor
$(\cG, \varphi_0=\id, \varphi_2=\id):H\text{-mod}\longrightarrow \widetilde{H}\text{-mod}$,
which clearly preserves braiding and twist. Also for the morphisms $\Omega_M$ and $\Upsilon_M$,
\begin{eqnarray}\label{eq:duality-trans}
\begin{aligned}
\cG(\Omega_M)(v^*\otimes w)
	&=\Omega_M(S^{-1}(g)v^*\otimes w) =(S^{-1}(g)v^*)(\alpha w)\\
	&=v^*(\tilde\alpha w), \quad \forall  v^*\in M^*, \  w\in M, \\
\cG(\Upsilon_M)(1)&=(\id\otimes S^{-1}(g^{-1}))\Upsilon_M(1)=\sum_i \beta b_i \otimes S^{-1}(g^{-1}) b^*_i\\
	&=\sum_i \beta g^{-1} b_i \otimes b^*_i=\sum_i \tilde\beta b_i \otimes b^*_i.
\end{aligned}
\end{eqnarray}
Thus the braided tensor functor $(\cG, \varphi_0=\id, \varphi_2=\id)$ preserves left duality.
This establishes the following result.
\begin{lemma}\label{lem:antipode-duality}
Any antipodal transformation of the form \eqref{eq:S-trans} induces a braided tensor equivalence
$(\cG, \varphi_0=\id, \varphi_2=\id):H\text{-mod}\longrightarrow \widetilde{H}\text{-mod}$ (with
$\cG$ defined by \eqref{eq:g-functor}), which preserves duality and twist.
\end{lemma}

To summarise,
\begin{theorem}\label{thm:ribbon-cat-equiv}
Let $H$ and $H'$ be equivalent ribbon quasi Hopf superalgebras. Then there exists a braided tensor equivalence
$H\text{-mod}\longrightarrow H'\text{-mod}$ which preserves duality and twist.
\end{theorem}
\begin{proof} Let $(f, F, g): H\longrightarrow H'$ be an equivalence of ribbon quasi Hopf superalgebras
(cf. Definition \ref{def:equiv-ribbon-quasi-Hopf}), where $f: H\longrightarrow H'$ is a superalgebra isomorphism,
and $F$ and $g$ are respectively gauge and antipodal transformations on $H'$.
Clearly $f$  induces a braided tensor equivalence
$H'\text{-mod}\longrightarrow H\text{-mod}$ which preserves duality and twist. Combining this observation
with Lemma \ref{lem:gauge-duality} and Lemma \ref{lem:antipode-duality}, we obtain the theorem.
\end{proof}


\begin{thebibliography}{9999}
\bibitem{BKK} Benkart, Georgia; Kang, Seok-Jin; Kashiwara, Masaki
Crystal bases for the quantum superalgebra $U_q(gl_{m\mid n})$.
{\sl J. Amer. Math. Soc. \bf 13} (2000), no. 2, 295--331.

\bibitem{BR} A. Berele and A. Regev, ``Hook Young diagrams with applications
to combinatorics and to representations of Lie superalgebras",  {\sl Adv.  Math. \bf 64} (1987), 118-175.

\bibitem{BL} Birman, J. S.; Lin, X.-S. Knot polynomials and Vassiliev's invariants.
{\sl Invent. Math. \bf 111} (1993), no. 2, 225--270.

\bibitem{BGZ} Bracken, A. J., Gould, M. D. and Zhang, R. B.,
Quantum Supergroups and Solutions of the Yang-Baxter Equation.
{\sl Modern Physics Letters \bf A 5} (1990) no. 11, 831--840.

\bibitem{BGLZ} A.J. Bracken, M.D. Gould, J.R. Links and Y.-Z. Zhang,
A new supersymmetric and exactly solvable model of correlated electrons.
{\sl Physical Review Letters \bf 74} (1995), 2768--2771.

\bibitem{C} Cartier, P.
Construction combinatoire des invariants de Vassiliev-Kontsevich des n{\oe}uds.
{\sl C. R. Acad. Sci. Paris S\'er. I Math.  \bf 316} (1993), no. 11, 1205--1210.

\bibitem{DLZ} P. Deligne, G. I. Lehrer, R. B. Zhang,
The first fundamental theorem of invariant theory for the orthosymplectic super group.
 arXiv:1508.04202

\bibitem{D86} Drinfeld, V. G. Quantum groups.
Proceedings of the International Congress of Mathematicians
Berkeley, California, USA, 1986,  798--820.

\bibitem{D90} Drinfeld, V. G.
On quasitriangular quasi-Hopf algebras and on a group that is closely connected
with $\rm{Gal}(\overline{\Q}/\Q)$. {\sl Algebra i Analiz \bf 2} (1990), no. 4, 149--181.

\bibitem{D92} Drinfeld, V. G.
On the structure of quasitriangular quasi-Hopf algebras.
{\sl Funktsional. Anal. i Prilozhen. \bf 26} (1992), no. 1, 78--80.

\bibitem{Du14}  Du, Jie and Gu, Haixia,
 A realisation of the quantum supergroup $\U_q(\gl_{m|n})$, {\sl J. Algebra \bf 404} (2014), 60--99. 

\bibitem{Du15}  Du, Jie and Gu, Haixia,
Canonical bases for the quantum supergroups $\U_q(\gl_{m|n})$, {\sl Math. Z. \bf 281} (2015), 631--660.

\bibitem{EK-I} Etingof, P.; Kazhdan, D.
Quantization of Lie bialgebras. I. {\sl Selecta Math. (N.S.) \bf 2} (1996), no. 1, 1--41.

\bibitem{EK-II} Etingof, P.; Kazhdan, D.
Quantization of Lie bialgebras. II, III. {\sl Selecta Math. (N.S.) \bf 2} (1998), no. 2, 213--231, 233--269;


\bibitem{FY} Freyd, P. J.; Yetter, D. N. Braided compact closed categories
with applications to low-dimensional topology.
 {\sl Adv. Math. \bf 77} (1989), no. 2, 156--182.

\bibitem{G06}
Geer, N. Etingof-Kazhdan quantization of Lie superbialgebras.
{\sl Adv. Math. \bf 207} (2006), no. 1, 1--38.

\bibitem{G07} Geer, N. Some remarks on quantized Lie superalgebras of classical type.
{\sl J. Algebra \bf 314} (2007), no. 2, 565--580.


\bibitem{GZB91} Gould, M. D.,  Zhang, R. B ., Bracken, A. J.
Generalized Gel'fand invariants and characteristic identities for quantum groups.
{\sl J. Math. Phys. \bf 32} (1991) 2298--2303.

\bibitem{GZB} Gould, M. D.,  Zhang, R. B ., Bracken, A. J.
Quantum double construction for graded Hopf algebras. {\sl Bulletin
Australian Math Society \bf 3} (2993) 353.

\bibitem{J} Jimbo, M. A $q$-difference analog of $\U(g)$ and the Yang-Baxter equation.
{\sl Lett.Math.Phys. \bf 10} (1985) 63--69.

\bibitem{JS}  Joyal, A; Street, R. Braided tensor categories.
{\sl Adv. Math. \bf 102} (1993), no. 1, 20--78.

\bibitem{Kac} Kac, V. G. Lie superalgebras.
{\sl Advances in Math. \bf 26} (1977), no. 1, 8--96.

\bibitem{Kass} C. Kassel, {\em Quantum groups},
  Springer-Verlag, Berlin (1994).

\bibitem{Ko} M. Kontsevich, Vassiliev's knot invariants, in {\em Gelfand Seminars},
Adv Soviet Math {\bf 16} (1993) 137-150.

\bibitem{KT}  Khoroshkin, S. M.; Tolstoy, V. N.
Universal R-matrix for quantized (super)algebras.
{\sl Comm. Math. Phys. \bf 141} (1991), no. 3, 599--617.

\bibitem{L}  Lanzmann, E. The Zhang transformation and
$U_q(osp(1,2l))$-Verma modules annihilators.
{\sl Algebr. Represent. Theory \bf 5} (2002), no. 3, 235--258.



\bibitem{LGZ} Links, J. R.; Gould, M. D.; Zhang, R. B.
Quantum supergroups, link polynomials and representation of the braid generator.
{\sl Rev. Math. Phys. \bf 5} (1993), no. 2, 345--361.

\bibitem{LZ06} Lehrer, G. I.; Zhang, R. B.
Strongly multiplicity free modules for Lie algebras and quantum groups.
{\sl J. Algebra \bf 306} (2006), no. 1, 138--174.

\bibitem{LZ10} Lehrer, G. I.; Zhang, R. B. A Temperley-Lieb analogue for the BMV algebra.
{\em Representation theory of algebraic groups and quantum groups}, 155--190,
Progr. Math., {\bf 284}, Birkhäuser/Springer, New York, 2010.

\bibitem{LZ12} Lehrer, G. I.; Zhang, R. B.  The Brauer Category and Invariant Theory.
{\sl Journal of the European Mathematical Society \bf 17} (2015) 2311-2351.
arXiv:1207.5889

\bibitem{LZ14a} Lehrer, G. I.; Zhang, R. B.
The first fundamental theorem of classical invariant theory for
classical supergroups.  	arXiv:1407.1058 [math.RT].

\bibitem{LZ14b} Lehrer, G. I.; Zhang, R. B.
The second fundamental theorem of invariant theory for the orthosymplectic supergroup.
 arXiv:1407.1058.

\bibitem{LZ15} Lehrer, G. I.; Zhang, R. B.
Invariants of the orthosymplectic Lie superalgebra and super Pfaffians.
arXiv:1507.01329 [math.RT]

\bibitem{LZZ}  Lehrer, G. I.; Zhang, Hechun; Zhang, R. B.
A quantum analogue of the first fundamental theorem of classical invariant theory.
{\sl Comm. Math. Phys. \bf 301} (2011), no. 1, 131--174.

\bibitem{M89} Manin, Yu. I. Multiparametric quantum deformation of the general linear supergroup.
{\sl Comm. Math. Phys. \bf 123} (1989), no. 1, 163--175.


\bibitem{MW}
Mikhaylov, V.; Witten, E. Branes and Supergroups.
{\sl Comm. Math. Phys. \bf 340} (2015), no. 2, 699--832.

\bibitem{MZ} Musson, I. M.; Zou, Y. M. Crystal bases for Uq(osp(1,2r)).
{\sl J. Algebra \bf 210} (1998), no. 2, 514--534.




\bibitem{RT90}  Reshetikhin, N. Yu.; Turaev, V. G.
Ribbon graphs and their invariants derived from quantum groups.
{\sl Comm. Math. Phys. \bf 127} (1990), no. 1, 1--26.

\bibitem{S}  M. Scheunert, {\em The theory of Lie superalgebras},
  Lecture Notes in Mathematics 716, Springer, Berlin (1979).

\bibitem{S00a} M. Scheunert,
The R-matrix of the symplecto-orthogonal quantum superalgebra
$\U_q(spo(2n|2m))$ in the vector representation.  	arXiv:math/0004032.



\bibitem{S0}A. Sergeev, ``An analogue of the classical theory of invariants for Lie superalgebras",
    (Russian) {\sl Funktsional. Anal. i Prilozhen. \bf 26} (1992), no. 3, 88--90;
    translation in {\sl Funct. Anal. Appl. \bf 26} (1992), no. 3, 223--225.

\bibitem{S1}A. Sergeev,  ``An analog of the classical invariant theory for Lie superalgebras. I",
    {\sl Michigan Math. J. \bf 49} (2001), Issue 1, 113-146.

\bibitem{Sh} Shum, Mei Chee,
Tortile tensor categories.
{\sl J. Pure Appl. Algebra \bf 93} (1994), no. 1, 57--110.

\bibitem{SZ98}  Scheunert, M.; Zhang, R. B.
Cohomology of Lie superalgebras and their generalizations.
{\sl J. Math. Phys. \bf 39} (1998), no. 9, 5024--5061.

\bibitem{Sl}Peter Selinger.
Autonomous categories in which $A\cong A^*$ (extended
abstract). Preprint. http://www.mscs.dal.ca/~selinger/papers/halftwist.pdf


\bibitem{SZ07} Su, Yucai; Zhang, R. B. Cohomology of Lie superalgebras
$\mathfrak{sl}_{m\vert n}$ and $\mathfrak{osp}_{2\vert 2n}$.
{\sl Proc. Lond. Math. Soc. (3) \bf 94} (2007), no. 1, 91--136.

\bibitem{Turaev}
 V. G. Turaev, {\em Quantum invariants of knots
  and 3-manifolds}, de Gruyter, Berlin (1994).



\bibitem{TW} Tuba, .I; Wenzl, H. On braided tensor categories of type BCD.
{\sl J. Reine Angew. Math. \bf 581} (2005), 31--69.

\bibitem{W} Witten, E.
Quantum field theory and the Jones polynomial.
{\sl Comm. Math. Phys. \bf 121} (1989), no. 3, 351--399.

\bibitem{Y91} Yamane, H.
Universal R-matrices for quantum groups associated to simple Lie superalgebras.
{\sl Proc. Japan Acad. Ser. A Math. Sci. \bf 67} (1991), no. 4, 108--112.

\bibitem{Y94}  Yamane, H.
A Serre type theorem for affine Lie superalgebras and their quantized enveloping superalgebras.
{\sl Proc. Japan Acad. Ser. A Math. Sci. \bf 70} (1994), no. 1, 31--36.

\bibitem{WZ} Wu, Yuezhu; Zhang, R. B.
Unitary highest weight representations of quantum general linear superalgebra.
{\sl J. Algebra \bf 321} (2009), no. 11, 3568--3593.


\bibitem{ZZ05}  Zhang, Hechun; Zhang, R. B. Dual canonical bases for
the quantum special linear group and invariant subalgebras.
{\sl Lett. Math. Phys. \bf 73} (2005), no. 3, 165--181.

\bibitem{Z} Zhang, Hechun, The quantum general linear supergroup,
 canonical bases and Kazhdan-Lusztig polynomials.
{\sl Sci. China Ser. \bf A 52} (2009), no. 3, 401--416.

\bibitem{ZZ06}  Zhang, Hechun; Zhang, R. B.
Dual canonical bases for the quantum general linear supergroup.
{\sl J. Algebra \bf 304} (2006), no. 2, 1026--1058.

\bibitem{Z92} Zhang, R. B.
Universal $L$ operator and invariants of the quantum supergroup $\U_q(\mathfrak{gl}(m|n))$.
{\sl J. Math. Phys. \bf 33} (1992),  no. 6,  1970 -- 1979.

\bibitem{Z92a}  Zhang, R. B.  Braid group representations arising from quantum supergroups
with arbitrary q and link polynomials.
{\sl J. Math. Phys. \bf 33} (1992), no. 11, 3918--3930.

\bibitem{Z92b}   Zhang, R. B.  Finite-dimensional representations of $\U_q(osp(1/2n))$
and its connection with quantum $so(2n+1)$.
{\sl Lett. Math. Phys. \bf 25} (1992), no. 4, 317--325.

\bibitem{Z92c}   Zhang, R. B.  A two-parameter quantization of $osp(4/2)$.
{\sl J. Phys. \bf A 25} (1992), no. 16, L991--L995.

\bibitem{Z93} Zhang, R. B.  Finite-dimensional irreducible representations of
the quantum supergroup  $\U_q(gl(m/n))$.
{\sl  J. Math. Phys. \bf 34} (1993), no. 3, 1236--1254.

\bibitem{Z94}  Zhang, R. B. Three-manifold invariants arising from $\U_q(osp(1|2))$.
{\sl Modern Phys. Lett. \bf A 9} (1994), no. 16, 1453--1465.

\bibitem{Z95}  Zhang, R. B.
Quantum supergroups and topological invariants of three-manifolds.
{\sl Rev. Math. Phys. \bf 7} (1995), no. 5, 809--831.



\bibitem{Z98}  Zhang, R. B.
Structure and representations of the quantum general linear supergroup.
{\sl Comm. Math. Phys. \bf 195} (1998), no. 3, 525--547.

\bibitem{Z02}  Zhang, R. B. Quantum enveloping superalgebras and link invariants.
{\sl J. Math. Phys. \bf 43} (2002), no. 4, 2029--2048.

\bibitem{Z04} Zhang, R. B.
Quantum superalgebra representations on cohomology groups of non-commutative bundles.
{\sl J. Pure Appl. Algebra \bf 191} (2004), no. 3, 285--314.

\bibitem{Z13} Zhang, R. B.  Serre presentations of Lie superalgebras.
{\em Advances in Lie Superalgebras}
Springer INdAM Series, Vol. {\bf 7}
Papi, Paolo; Gorelik, Maria (Eds.)
(2013) pp235--280.

\bibitem{ZBG91} Zhang, R. B.; Bracken, A. J.; Gould, M. D.
Solution of the graded Yang-Baxter equation associated with
the vector representation of $\U_q(osp(M|2n))$.
{\sl Phys. Lett. \bf B 257} (1991), no. 1-2, 133-139.

\bibitem{ZG}  Zhang, R. B.; Gould, M. D.
Universal R matrices and invariants of quantum supergroup.
{\sl J.  Math. Physics \bf 32}  (1991), no. 12, 3261--3267.

\bibitem{ZGB91a}  Zhang, R. B.; Gould, M. D.; Bracken, A. J.
Quantum group invariants and link polynomials.
{\sl Comm. Math. Phys. \bf 137} (1991), no. 1, 13--27.


\bibitem{ZGB91b}
Zhang, R. B.; Gould, M. D.; Bracken, A. J.
Solutions of the graded classical Yang-Baxter equation and integrable models.
{\sl J. Phys. \bf  A 24} (1991), no. 6, 1185--1197.

\bibitem{Zo98} Zou, Y. M. Integrable representations of $\U_q(osp(1,2n))$.
{\sl J. Pure Appl. Algebra \bf 130} (1998), no. 1, 99--112.

\bibitem{Zo99} Zou, Y. M. On the structure of $\U_q(sl(m,1))$: crystal bases.
{\sl J. Phys. \bf A 32} (1999), no. 46, 8197--8207.

\bibitem{Zo01} Zou, Y. M.  Crystal bases for $\U_q(\Gamma(\sigma_1, \sigma_2, \sigma_3))$.
{\sl Trans. Amer. Math. Soc. \bf 353} (2001), no. 9, 3789--3802.


\end{thebibliography}
\end{document}